\documentclass[11pt]{article}

\usepackage{mathtools}
\usepackage{amsthm}
\usepackage{amsmath}
\usepackage{amssymb}
\usepackage{xcolor}
\usepackage{xypic}
\usepackage{cite}
\usepackage{float}
\usepackage{filecontents}
\usepackage{graphicx}
\usepackage{epstopdf}
\usepackage{epsfig}
\usepackage{mathabx}

\theoremstyle{definition}
\newtheorem{definition}{Definition}[section]
\newtheorem{remark}[definition]{Remark}
\newtheorem{example}[definition]{Example}
\newtheorem{notation}[definition]{Notation}
\newtheorem{algorithm}[definition]{Algorithm}
\theoremstyle{plain}
\newtheorem{lemma}[definition]{Lemma}
\newtheorem{theorem}[definition] {Theorem}
\newtheorem{proposition}[definition] {Proposition}
\newtheorem{main}[definition]{Main Theorem}
\newtheorem{corollary}[definition]{Corollary}

\definecolor{forest}{rgb}{0.0, 0.5, 0.0}
\usepackage[colorlinks=true,linkcolor=blue,citecolor=forest]{hyperref}
\numberwithin{equation}{section}

\theoremstyle{plain}

\begin{document}
\title{A Magnus extension for locally indicable groups}
\author{Carsten Feldkamp}
\date{\today}
\maketitle
\abstract{A group $G$ possesses the Magnus property if for every two elements $u$, $v \in G$ with the same normal closure, $u$ is conjugate to $v$ or $v^{-1}$. O. Bogopolski and J. Howie proved independently that the fundamental groups of all closed orientable surfaces possess the Magnus property.  The analogous result for closed non-orientable surfaces was proved by O. Bogopolski and K. Sviridov except for one case that was later covered by the author. In this article, we generalize those results, which can be viewed as Magnus extensions for free groups, to a Magnus extension for locally indicable groups and consider the influence of adding a group as a direct factor. For this purpose, we also prove versions of the Freiheitssatz for locally indicable groups and of a result by M. Edjvet adding a group as a direct factor.}

\setcounter{tocdepth}{2}

\tableofcontents

\section{Introduction}

A group $G$ possesses the Magnus property if for every two elements $u$, $v \in G$ with the same normal closure, $u$ is conjugate in $G$ to $v$ or $v^{-1}$.
The Magnus property was named after W. Magnus who proved the so-called Freiheitssatz and the Magnus property for free groups (see \cite{ArtMagnus}). S. Brodskii, J. Howie and H. Short independently generalized this Freiheitssatz to locally indicable groups (see Theorem~\ref{HowFS}). Recall that a group $G$ is \emph{indicable} if there exists an epimorphism from $G$ to $\mathbb{Z}$. A group $G$ is \emph{locally indicable} if every non-trivial, finitely generated subgroup of $G$ is indicable. In particular limit groups (in the sense of Sela, see for example \cite{ArtTheorie}) are indicable and locally indicable. Note that every locally indicable group is torsion-free, but the converse does not hold in general. Using the Freiheitssatz for locally indicable groups, M. Edjvet proved a result including the statement that the free product of two locally indicable groups with the Magnus property possesses the Magnus property (see Theorem~\ref{Edj}, Corollary~\ref{corEdj}).

For a group $G$ we denote the normal closure of an element $r \in G$ in $G$ by $\langle \! \langle r \rangle \! \rangle_{G}$ and write $G / \langle \! \langle r \rangle \! \rangle_{G}$ (or short $G / \langle \! \langle r \rangle \! \rangle$) for the group $G$ modulo the normal closure of $r$ in $G$. The commutator $[a,b]$ of two elements $a$ and $b$ is defined by $[a,b] := a^{-1}b^{-1}ab$. We call a result that describes how to construct (possibly under certain conditions) a group $G$ with the Magnus property out of one or more groups $G_{i}$ ($i \in \mathcal{I}$) with the Magnus property a \emph{Magnus extension} if the original groups $G_{i}$ ($i \in \mathcal{I}$) embed canonically into $G$. In this sense M. Edjvet's result Corollary~\ref{corEdj} is a Magnus extension for locally indicable groups. The main theorem of this article is also a Magnus extension for locally indicable groups:

\begin{main} \label{main}
Let $C$ be a group and $G$ an indicable as well as locally indicable group. Further let $u$ be a non-trivial element in $G$. Then the group
\begin{eqnarray*}
\big( G \underset{u=[a,b]}{\ast} F(a,b) \big) \times C
\end{eqnarray*}
possesses the Magnus property if and only if $G \times C$ possesses the Magnus property. In the case $C=\{1\}$ we can omit the assumption that $G$ is indicable.
\end{main}

In Section~\ref{secapp}, we have a look at two applications of Theorem~\ref{main}. The first application is Corollary~\ref{corappl} which generalizes the results of \cite{ArtSurGr}, \cite{ArtHowSurGr}, \cite{ArtOneRel} and \cite{ArtMyMAP} about the Magnus property of fundamental groups of closed surfaces. It is well known that these groups apart from genus 1 and 2 are one-relator groups with a relation of the form $[a,b]u=1$ where $a$, $b$ are generators and $u$ is a non-trivial element in generators different from $a$, $b$. Corollary~\ref{corappl} allows us to consider arbitrary many relations $[a_{i},b_{i}]u_{i}=1$ ($i \in \mathcal{I}$) where $a_{i}$, $b_{i}$ are generators and the $u_{i}$ are non-trivial words in generators from a common generating set which does not contain any $a_{i}$, $b_{i}$. These groups can also be understood as fundamental groups of particular graph of groups (cf. Figure~\ref{AbbArtAnw}). Note that for this application we use Main Theorem~\ref{main} with trivial $C$.

Another natural way of generalizing the result about the Magnus property of fundamental groups of closed surfaces would be to consider fundamental groups of manifolds with dimension greater or equal three. Our second applications hints in that direction by applying Main Theorem~\ref{main} to fundamental groups of  certain Seifert manifolds using the classification due to H. Seifert. For more results on the Magnus property of fundamental groups of Seifert manifolds see \cite{DA}. 

In Section~\ref{secdir}, we provide auxiliary results which will be used for the proof of Main Theorem~\ref{main} in Section~\ref{secmain}. Nevertheless, at least three of the results of this section can also be considered independently. We give a quick overview on those results: First, we want to mention the following result of Section~\ref{secdir} about direct products which we prove using the language of groups in the sense of model theory.

\begin{theorem} \label{thmdirect}
Let $C$ be an arbitrary group. Then the following statements are equivalent.
\begin{itemize}
\item[$(i)$] The direct product $L \times C$ possesses the Magnus property for some limit group $L$ with the Magnus property. 
\item[$(ii)$] All groups $\big( \bigtimes_{i \in \mathcal{I}} L_{i} \big) \times C$, where  $\mathcal{I} \subset \mathbb{N}$ is an index set and $L_{i}$ are limit groups with the Magnus property possess the Magnus property.
\end{itemize}
\end{theorem}

Note that in \cite{ArtTouikanLouder} there is an example of a limit group without the Magnus property. The possibility of using the elementary theory of groups to transfer the knowledge of the Magnus property of one group to all elementary equivalent groups was already mentioned in \cite{ArtSurGr}. 

In 1981, J. Howie proved the following theorem which is known as the \emph{Freiheitssatz for locally indicable groups}. This result was also proved independently by S. Brodskii (see \cite{ArtBrod80},\cite[Theorem 1]{ArtBrod84}) and by H. Short (see \cite{PhDShort}).

\begin{theorem} \label{HowFS} \textbf{\emph{\cite[Theorem 4.3 (Freiheitssatz)]{ArtHowOnpairs}}}
Suppose $G=(A \ast B) / N$, where $A$ and $B$ are locally indicable groups, and $N$ is the normal closure in $A \ast B$ of a cyclically reduced word $R$ of length at least $2$. Then the canonical maps $A \rightarrow G$, $B \rightarrow G$ are injective.
\end{theorem}

We slightly generalize this theorem in the following way. 

\begin{lemma} \label{lemmachi}
Let $A$, $B$ be two locally indicable groups and let $G:= \big(A \ast B \big) \times C$ for a group $C$. Further let $w \in G$ be an element that is neither conjugate to an element of $A \times C$ nor $B \times C$. We define $N:= \langle \! \langle w \rangle \! \rangle_{G} \cap C$ and $\widetilde{C} := C / N$. Then $A \times \widetilde{C}$ canonically embeds into $G / \langle \! \langle w \rangle \! \rangle_{G}$.
\end{lemma}

The proof of this lemma is straight forward. In \cite[Theorem C (i)]{ArtAntKar}, the statement can be found that under the conditions of Lemma~\ref{lemmachi} even the direct product $A \times C$ embeds into $G / \langle \! \langle w \rangle \! \rangle_{G}$, but this statement is not correct (see Example~\ref{exmachi}).

In 1982, J. Howie published the following result.

\begin{theorem} \textbf{\emph{\cite[Theorem 4.2]{ArtHowOnLocInd}}} \label{Howlokind}
Let $A$ and $B$ be locally indicable groups, and let $G$ be the quotient of $A \ast B$ by the normal closure of a cyclically reduced word $r$ of length at least $2$. Then the following are equivalent:
\begin{itemize}
\item[$(i)$] $G$ is locally indicable;
\item[$(ii)$] $G$ is torsion-free;
\item[$(iii)$] $r$ is not a proper power in $A \ast B$.
\end{itemize}

\end{theorem}

Using the Freiheitssatz for locally indicable groups and Theorem~\ref{Howlokind}, M. Edjvet proved the following result:

\begin{theorem} \textbf{\emph{\cite[(Main) Theorem]{ArtEdjvet}}} \label{Edj}
Let $A$, $B$ be locally indicable groups, and let $R$, $S \in A \ast B$ be cyclically reduced words each of length at least $2$. If $\langle \! \langle R \rangle \!  \rangle_{A \ast B} = \langle \! \langle  S \rangle \! \rangle_{A \ast B}$ then $R$ is a conjugate in $A \ast B$ of $S^{\pm 1}$. 
\end{theorem}

By a small consideration (which can be found in the author's M.Sc.~thesis) we directly get the following corollary:

\begin{corollary} \label{corEdj}
Let $\mathcal{J}$ be a non-trivial index set and $A_{j}$ ($j \in \mathcal{J}$) locally indicable groups with the Magnus property. Then $G = \Asterisk_{j \in \mathcal{J}} A_{j}$ possesses the Magnus property.
\end{corollary}

In Section~\ref{secdir}, we proof the following result using the method of M. Edjvet's proof of Theorem~\ref{Edj} along with Lemma~\ref{lemmachi} and Theorem~\ref{thmdirect}.

\begin{theorem} \label{thmedjvetgen}
Let $\mathcal{J}$ be an index set and let $A_{j}$ ($j \in \mathcal{J}$) be indicable as well as locally indicable groups. Further let $C$ be a group. Then $(\bigast_{j \in \mathcal{J}} A_{j} ) \times C$ possesses the Magnus property if and only if all $A_{j} \times C$ ($j \in \mathcal{J}$) possess the Magnus property.
\end{theorem}

The results of Section~\ref{secdir} allow us to prove Main Theorem~\ref{main} in Section~\ref{secmain} in a similar way to \cite{ArtOneRel} and \cite{ArtMyMAP}.

\section{Direct products and the Magnus property} \label{secdir}

\subsection{Indicable groups and limit groups}

The following result was part of the author's M.Sc. thesis \cite{MA}.

\begin{theorem} \label{thmFxF}
All groups $F_{m} \times F_{n}$ with $m$, $n \in \mathbb{N} \cup \{\infty\}$, where $F_{k}$ is the free group of rank $k$, possess the Magnus property.
\end{theorem}

In this section, we want to generalize that result to Theorem~\ref{thmdirect}. First, we cite some further result of the author's M.Sc.~thesis which will be used in this section. To formulate this lemma, we use the notation $r=(r_{i})_{i \in M}$ for an element $r \in G:= \bigtimes^{n}_{i=1} G_{i}$, where $M=\{1,2,\dots,n\}$ and where the $r_{i}$ are the components of $r$ from $G_{i}$. Further, we denote by $\{-1,1\}^{M}$ the set of all elements $(a_{i})_{i \in M}$ with $a_{i} \in \{-1,1\}$ for all $i \in M$. 

\begin{lemma}\textbf{\emph{\cite[Lemma 4.5.2]{MA}}} \label{Lemnus}
Consider a direct product $G = \bigtimes_{i=1}^{n} G_{i}$ with $n \in \mathbb{N} \backslash \{1\}$, where each group $G_{i}$ possesses the Magnus property. Let $M=\{1,2,\dots,n\}$. Then $G$ has the Magnus property if and only if the following holds: For all $(r_{i})_{i \in M} \in G$ and $\varepsilon =( \varepsilon_{i})_{i \in M} \in \{-1,1\}^{M} \backslash \{(-1)_{i \in M}, (1)_{i \in M}\}$ such that $\langle \! \langle (r_{i})_{i \in M} \rangle \! \rangle_{G} = \langle \! \langle (r_{i}^{\varepsilon_{i}})_{i\in M} \rangle \! \rangle_{G}$ we have
\begin{eqnarray*}
r_{i} \sim_{G_{i}} r_{i}^{-1} \ \ \text{for all} \ i \ \text{with} \ \varepsilon_{i}=1 \ \ \ \text{or} \ \ \ r_{i} \sim_{G_{i}} r_{i}^{-1} \ \text{for all} \ i \ \text{with} \ \varepsilon_{i}=-1 .
\end{eqnarray*}

\end{lemma}

This small lemma is a  tool for examining the Magnus property of direct products. As mentioned in \cite[Example (b) to Lemma 4.5.2]{MA} it follows directly from Lemma~\ref{Lemnus}, that $\mathbb{Z}_{4} \times \mathbb{Z}_{6}$ does not have the Magnus property. Thus, not every direct product of two groups with the Magnus property possesses the Magnus property, but note that every factor of a direct product with the Magnus property also has the Magnus property.

Next, we recall two results by R. Lyndon and P. Schupp on aspherical presentations. For the definition of aspherical presentations see \cite[Chapter III]{BuchLS}.

\begin{proposition} \textbf{\emph{\cite[cf. Proposition 11.1]{BuchLS}}} \label{propasphone}
If $G=\langle X \mid R \rangle$ where $R$ consists of a single relator, then the presentation is aspherical.
\end{proposition}

\begin{proposition} \textbf{\emph{\cite[cf. Proposition 10.2]{BuchLS}}} \label{propasphpairs}
If $G = \langle X \mid R \rangle$ is aspherical, and no element of $R$ is conjugate to another or to its inverse, then the following condition holds for $F=\langle X \mid \rangle$ and $N = \langle \! \langle R \rangle \! \rangle_{F}$:

Let $p_{1} \cdots p_{n}=1$ where each $p_{i} = u_{i}r_{i}^{e_{i}} u_{i}^{-1}$ for some $u_{i} \in F$, $r_{i} \in R$ and $e_{i}= \pm 1$. Then the indices fall into pairs $(i,j)$ such that $r_{i} = r_{j}$, $e_{i} = - e_{j}$, and $u_{i} \in u_{j}NC_{i}$ where $C_{i}$ (the centralizer of $r_{i}$) is the cyclic group generated by the root $s_{i}$ of $r_{i} = s_{i}^{m_{i}}$.
\end{proposition}

Using those two propositions we are able to prove the following lemma.

\begin{lemma} \label{lemfirstorder}
Let $m \in \mathbb{N}$ and $\delta_{j}$, $\varepsilon \in \{\pm 1\}$ $(1 \leqslant j \leqslant m)$ with $\overset{m}{\underset{j=1}{\Sigma}} \delta_{j} \neq \varepsilon$. Further let $F$ be a free abelian or free group. Then the sentence
\begin{eqnarray} \label{eqsent1}
\exists g_{j} \ (1 \leqslant j \leqslant m), \ \exists r \colon \neg (r=1) \ \land \ r^{\varepsilon} = \overset{m}{\underset{j=1}{\Pi}}g_{j}^{-1} r^{\delta_{j}} g_{j}
\end{eqnarray}
is false in $F$.
\end{lemma}

\begin{proof}
We prove this lemma by contradiction and assume that \eqref{eqsent1} is true in $F$. First, we consider the case that $F$ is abelian. Then \eqref{eqsent1} corresponds to
\begin{eqnarray}
\exists r \colon \neg (r=1) \ \land \ r^{\varepsilon} = r^{\delta},
\end{eqnarray}
where $\delta := \overset{m}{\underset{j=1}{\Sigma}} \delta_{j}$. If this sentence is true in $F$, then we have $r^{\varepsilon-\delta}=1$. Thus for $\delta \neq \varepsilon$ the group $F$ has torsion. This is a contradiction. For the case that $F$ is not abelian let $\mathcal{X}$  be a basis of $F$. We consider the group presentation $G = \langle \mathcal{X} \mid r \rangle$. It follows directly from Proposition~\ref{propasphone} and Proposition~\ref{propasphpairs} that the set $\{- \varepsilon, \delta_{1},\delta_{2},\dots,\delta_{n}\}$ falls into pairs $(i,j)$ with $i=-j$. Finally, we have
\begin{eqnarray*}
- \varepsilon + \overset{m}{\underset{j=1}{\Sigma}} \delta_{j} = 0 \ \ \Leftrightarrow \ \ \overset{m}{\underset{j=1}{\Sigma}} \delta_{j} = \varepsilon
\end{eqnarray*}
which contradicts the assumption of the lemma.
\end{proof}

Before we can begin with the proof of Theorem~\ref{thmdirect} we need the following lemma concerning indicable groups.

\begin{lemma} \label{lemindC}
Let $C$ be an arbitrary group. If there is an indicable group $G$ such that $G \times C$ possess the Magnus property, then $\mathbb{Z} \times C$ also possess the Magnus property.
\end{lemma}

\begin{proof}
Our proof will be by contradiction. We assume that $\mathbb{Z} \times C$ does not possess the Magnus property and that $G$ is an indicable group such that $G \times C$ possesses the Magnus property. Since $G$ is indicable there exists an epimorphism $\zeta \colon G \rightarrow \mathbb{Z}$. Let $g$ be an element of $G$ with $\zeta(g)=1$ (where 1 is a generator of $\mathbb{Z}$). Further let $\varphi \colon \mathbb{Z} \times C \hookrightarrow G \times C$ be the monomorphism defined by $\varphi(1,c)=(g,c)$ and $\psi \colon G \times C \rightarrow \mathbb{Z} \times C$ be the homomorphism defined by $\psi(h,c) = (\zeta(h),c)$ for all $h \in G$ and $c \in C$. We consider the chain of homomorphisms
\begin{eqnarray} \label{eqhomchain}
\mathbb{Z} \times C \ \  \overset{\varphi}{\hookrightarrow} \ \ G \times C \ \ \overset{\psi}{\longrightarrow} \ \ \mathbb{Z} \times C.
\end{eqnarray}
Note that $\psi \circ \varphi = \text{id}_{\mathbb{Z} \times C}$. By assumption, there are elements $r$, $s \in \mathbb{Z} \times C$  with the same normal closure in $\mathbb{Z} \times C$ such that  $r$ is neither conjugate to $s$ nor to $s^{-1}$. It follows immediately from \eqref{eqhomchain} that $\varphi(r)$ and $\varphi(s)$ have the same normal closure in $G \times C$, but cannot be conjugate in $G \times C$. Therefore $G \times C$ does not have the Magnus property which contradicts the assumption.
\end{proof}

Now we are able to give a proof of Theorem~\ref{thmdirect}.\medskip

\noindent\textit{Proof of Theorem~\ref{thmdirect}.\,}
Because of Lemma~\ref{lemindC}, the statement
\begin{center}
$\text{(i)}^{*}$ \ \ The direct product $\mathbb{Z} \times C$ possesses the Magnus property. 
\end{center}
follows from statement (i) of Theorem~\ref{thmdirect}. Since the statements (i) and $\text{(i)}^{*}$ are also contained in statement (ii) of Theorem~\ref{thmdirect}, it is sufficient to prove that statement (ii) follows from statement $\text{(i)}^{*}$.

Consider a direct product $G:= \bigtimes_{i \in \mathcal{I}} L_{i}$ of limit groups $L_{i}$ with the Magnus property. Note that $C$ possesses the Magnus property because it is a factor of the direct product $\mathbb{Z} \times C$ with the Magnus property. Since every counterexample for the Magnus property of $G \times C$ is contained in a direct product consisting of $C$ and finitely many direct factors of $G$, we assume w.\,l.\,o.\,g. $\mathcal{I} = \{1,2,\dots,n\}$ for some $n \in \mathbb{N}$. Because of Lemma~\ref{Lemnus} the direct product $G \times C$ possesses the Magnus property if and only if the statement 
\begin{eqnarray} \label{eqaim1}
&& \big( \ r_{i} \ \sim_{L_{i}} \ r_{i}^{-1} \ \ \forall \ i \ \text{with} \ \varepsilon_{i}=1 \ \text{and} \ c \ \sim_{C} \ c^{-1} \ \big)\nonumber \\
&\text{or}& \big( \ r_{i} \ \sim_{L_{i}} \ r_{i}^{-1} \ \ \forall \ i \ \text{with} \ \varepsilon_{i} =-1 \ \big)
\end{eqnarray}
is true for all elements $((r_{i})_{i \in \mathcal{I}},c) \in G \times C$ and all $(\varepsilon_{i})_{i \in \mathcal{I}} \in \{( \pm 1 )_{i \in \mathcal{I}} \} \backslash \{(1)_{i\in \mathcal{I}} \}$ with 
\begin{eqnarray} \label{eqt1}
\langle \! \langle \big((r_{i})_{i\in \mathcal{I}},c \big) \rangle \! \rangle_{G \times C} = \langle \! \langle \big((r_{i}^{\varepsilon_{i}})_{i\in \mathcal{I}},c \big) \rangle \! \rangle_{G \times C}.
\end{eqnarray}
By removing all factors $L_{i}$ for $i$ with $r_{i}=1$ we can assume w.\,l.\,o.\,g. $r_{i} \neq 1$ for all $i \in \mathcal{I}$. Let $((r_{i})_{i \in \mathcal{I}},c) \in G \times C$ and $(\varepsilon_{i})_{i \in \mathcal{I}} \in \{( \pm 1 )_{i \in \mathcal{I}} \} \backslash \{(1)_{i\in \mathcal{I}} \}$ with \eqref{eqt1}. Our aim is to show that statement \eqref{eqaim1} holds. Because of \eqref{eqt1} there exist elements $m \in \mathbb{N}$, $\delta_{j} \in \{-1,1\}$ and $w_{j} \in G \times C$ ($1 \leqslant j \leqslant m$) with
\begin{eqnarray} \label{eqhelp1}
\big( ( r_{i}^{\varepsilon_{i}})_{i \in \mathcal{I}},c \big) = \overset{m}{\underset{j=1}{\Pi}} w_{j}^{-1} ((r_{i})_{i\in \mathcal{I}},c)^{\delta_{j}} w_{j}.
\end{eqnarray}
In particular there exist for every $i \in \mathcal{I}$ elements $h(i)_{j} \in L_{i}$ ($1 \leqslant j \leqslant m$) with
\begin{eqnarray} \label{eqhelp2}
r_{i}^{\varepsilon_{i}} \ = \ \overset{m}{\underset{j=1}{\Pi}} h(i)_{j}^{-1} r_{i}^{\delta_{j}} h(i)_{j}.
\end{eqnarray}

We know from Lemma~\ref{lemfirstorder} that the sentence
\begin{eqnarray*}
\exists g_{j} \ (1 \leqslant j \leqslant m), \ r \colon \neg (r=1) \ \land \ r^{\varepsilon_{i}}= \overset{m}{\underset{j=1}{\Pi}} g_{j}^{-1} r^{\delta_{j}} g_{j}
\end{eqnarray*}
is false for $\overset{m}{\underset{j=1}{\Sigma}} \delta_{j} \neq \varepsilon_{i}$ in every free or free abelian group. Since limit groups possess the same existential theory as free or free abelian groups of finite rank it follows from \eqref{eqhelp2} that 
\begin{eqnarray} \label{eqhelp3}
\overset{m}{\underset{j=1}{\Sigma}} \delta_{j} = \varepsilon_{i} \ \ \forall i \in \mathcal{I}.
\end{eqnarray}
In particular all $\varepsilon_{i}$ ($i \in \mathcal{I}$) are equal. By assumption we therefore have $\varepsilon_{i} = -1$ ($i \in \mathcal{I}$) and it follows that \eqref{eqaim1} holds in the case $c=1$. For the case $c \neq 1$ we consider the canonical projection $\pi \colon G \times C \rightarrow L_{1} \times C$. By applying $\pi$ to \eqref{eqhelp1} we get
\begin{eqnarray} \label{eqhelp4}
(r_{1}^{-1},c) \ = \ \overset{m}{\underset{j=1}{\Pi}} (h(1)_{j},d_{j})^{-1} (r_{1},c)^{\delta_{j}} (h(1)_{j},d_{j}),
\end{eqnarray}
where $\Sigma_{j=1}^{m} \delta_{j} = -1$ and $\pi(w_{j})=(h(1)_{j},d_{j})$ ($1 \leqslant j \leqslant m, \ d_{j} \in C$). Let $z$ be a generator of $\mathbb{Z}$. We deduce from \eqref{eqhelp4} that
\begin{eqnarray*}
(z^{-1},c)=\overset{m}{\underset{j=1}{\Pi}} (1,d_{j})^{-1} (z,c)^{\delta_{j}} (1,d_{j}) \ \ \ \text{and} \ \ \ (z,c) = \overset{m}{\underset{j=1}{\Pi}} (1,d_{j})^{-1} (z^{-1},c)^{\delta_{j}} (1,d_{j})
\end{eqnarray*}
hold in $\mathbb{Z} \times C$. Altogether we get $\langle \! \langle (z,c) \rangle \! \rangle_{\mathbb{Z} \times C} = \langle \! \langle (z,c^{-1}) \rangle \! \rangle_{\mathbb{Z} \times C}$. Since $z$ is not conjugate to $z^{-1}$ in $\mathbb{Z}$ and $\mathbb{Z} \times C$ possesses the Magnus property by assumption, it follows from Lemma~\ref{Lemnus} that $c$ is conjugate to $c^{-1}$ in $C$. Finally, \eqref{eqaim1} holds because of $\varepsilon_{i}=-1$ ($i\in \mathcal{I}$).
\qed

\subsection{Small generalization of the Freiheitssatz for locally indicable groups} 

We already recalled the Freiheitssatz for locally indicable groups (see Theorem~\ref{HowFS}). In this subsection, we prove a small generalization of that Freiheitssatz, namely Lemma~\ref{lemmachi}. \medskip

\noindent\textbf{Proof of Lemma~\ref{lemmachi}.}
To the contrary, suppose that $u=(a,c)$ is an element of $(A \times C) \backslash (\{1\} \times C)$ which is trivial in $G/ \langle \! \langle w \rangle \! \rangle_{G}$. Then there exist $n \in \mathbb{N}$, $\varepsilon_{i} \in \{-1,1\}$ and elements $p_{i} \in A \ast B$, $d_{i} \in C$ ($1 \leqslant i \leqslant n$) with
\begin{eqnarray} \label{eqcont1}
(a,c) \ \ = \ \ \overset{n}{\underset{i=1}{\Pi}} (p_{i},d_{i})^{-1} w^{\varepsilon_{i}} (p_{i},d_{i}) \ \ \ \text{in} \ \ \ G.
\end{eqnarray}
Let $\pi \colon G \rightarrow A \ast B$ be the canonical projection. We have $\pi(a)=a \neq 1$. Further $\pi(w)$ is neither conjugate to an element of $A$ nor $B$. Because of $(a,c) \in \langle \! \langle w \rangle \! \rangle_{G}$ the element $\pi((a,c))=a$ is an element of the normal closure of $\pi(w)$ in $\pi(G)=A \ast B$. This contradicts Theorem~\ref{HowFS}. Thus, $\langle \! \langle w \rangle \! \rangle_{G} \cap (A \times C)= \langle \! \langle w \rangle \! \rangle_{G} \cap C$. Because of
\begin{eqnarray*}
(A \times C) / ( \langle \! \langle w \rangle \! \rangle_{G} \cap C) \ \ = \ \ (A \times C) / N \ \ = \ \ A \times (C / N) \ \ = \ \ A \times \widetilde{C}
\end{eqnarray*}
and the trivial canonical embedding
\begin{eqnarray*}
(A \times C) / ( \langle \! \langle w \rangle \! \rangle_{G} \cap (A \times C)) \ \ \hookrightarrow \ \ G / \langle \! \langle w \rangle \! \rangle_{G}
\end{eqnarray*}
we get the canonical embedding of $A \times \widetilde{C}$ into $G / \langle \! \langle w \rangle \! \rangle_{G}$.
\qed

The following example shows that under the condition of Lemma~\ref{lemmachi} it is in general not possible to embed $A \times C$ into $G / \langle \! \langle w \rangle \! \rangle_{G}$. 

\begin{example} \label{exmachi}
Let $A= \langle a \mid \rangle$, $B = \langle b \mid \rangle$ be free groups of rank $1$ and $C= \langle c,d \mid \rangle$ a free group of rank $2$. Further let 
\begin{eqnarray*}
G \ \ :=  \  \ (A \ast B) \times C \ \ = \ \ \langle a, b ,c,d \, \mid \, [a,c], \, [b,c], \, [a,d], \, [b,d] \, \rangle
\end{eqnarray*}
and $w:=abc$. Then we have
\begin{eqnarray*}
d^{-1}w^{-1}dw \ \ = \ \ [d,w]  \ \ = \ \ [d,abc] \ \ = \ \ [d,c]c^{-1} [d,ab]c \ \ = \ \ [d,c] \in C.
\end{eqnarray*}
Therefore, $[d,c]$ is a non-trivial element of $C$, which corresponds to the trivial element in $G/ \langle \! \langle w \rangle \! \rangle$. 
\end{example}

\subsection{Auxiliary results}

Before we consider the statement of Theorem~\ref{Edj}, we prove some auxiliary statements.

\begin{lemma} \label{lemtestudai1}
Let $A$, $B$ be two groups and $\varphi \colon A \ast B \twoheadrightarrow \mathbb{Z}$ an epimorphism mapping all elements of $B$ to the identity element $0 \in \mathbb{Z}$. Further let $a$ be an element of $A$ with $\varphi(a)=1$. Then we have
\begin{eqnarray*}
\ker(\varphi) \ = \ (A \cap \ker(\varphi)) \ast \underset{i \in \mathbb{Z}}{\bigast} a^{-i} B a^{i}.
\end{eqnarray*}
\end{lemma}

\begin{proof}
We choose $\{a^{i} \mid i \in \mathbb{Z} \}$ as a system of coset representatives of double cosets $\ker(\varphi) \backslash (A \ast B) / B$. The set $\ker(\varphi) \backslash (A \ast B) / A$ only consists of the trivial double coset. We choose the representative 1. The choice of these representatives enable us to apply the Kurosh subgroup theorem (see e.\,g. \cite[§5.5 Theorem 14]{BookSerre}). Thus, we have
\begin{eqnarray*}
\ker(\varphi) \ = \ F \ast \big( A \cap \ker(\varphi) \big) \ \ast \ \underset{i \in \mathbb{Z}}{\bigast} \big( a^{-i} B a^{i} \cap \ker(\varphi) \big) \ = \ F \ast \big( A \cap \ker(\varphi) \big) \ast \underset{i \in \mathbb{Z}}{\bigast} a^{-i} B a^{i}
\end{eqnarray*}
for a free group $F$. We further note that the subgroup
\begin{eqnarray*}
N \ := \ \big( A \cap \ker(\varphi) \big) \ast \underset{i \in \mathbb{Z}}{\bigast} a^{-i} B a^{i}
\end{eqnarray*}
of $\ker(\varphi)$ is normal in $A \ast B$ and $a^{i}$ ($i \in \mathbb{Z}$) is a system of coset representatives of $N$ in $A \ast B$. It follows:
\begin{eqnarray*}
\big( ( A \ast B ) / N \big) / ( \ker(\varphi) / N) \ \ \cong \ \ (A \ast B) / \ker(\varphi) \ \ \cong \ \ \langle a \mid \,\rangle \ \ \cong \ \ (A \ast B)/N
\end{eqnarray*}
Finally, we have $\ker(\varphi) = N$.
\end{proof}

\begin{lemma} \label{lemkernphi}
Let $A$, $B$ be two groups and $\psi \colon A \ast B \twoheadrightarrow \mathbb{Z}$ an epimorphism sending an element $a \in A$ and an element $b \in B$ to $1$ (where $1$ is a generator of $\mathbb{Z}$). We consider the group $G' := A \ast B \ast \langle x \rangle$ and extend $\psi$ to an epimorphism $\varphi \colon G' \rightarrow \mathbb{Z}$ with $\varphi(x)=1$. Let $a_{i} := x^{-i} a x^{i-1}$ and $b_{i} := x^{-i}bx^{i-1}$. Then we have
\begin{eqnarray*}
\ker(\varphi) \ \ = \ \ (A \cap \ker(\varphi) \big) \ast \big(B \cap \ker(\varphi) \big) \ \ast \ \langle \{a_{i} \mid i \in \mathbb{Z} \} \cup \{b_{i} \mid i \in \mathbb{Z} \} \mid \rangle.
\end{eqnarray*}
\end{lemma}

\begin{proof}
We choose the representative systems $\{1\}$ of $\ker(\varphi) \backslash (A \ast B) / A$ and $\ker(\varphi) \backslash (A \ast B) / B$ respectively. Applying the Kurosh subgroup theorem (see e.\,g. \cite[§5.5 Theorem 14]{BookSerre}) we get
\begin{eqnarray*}
\ker(\varphi) \ = \ F \ast \big( A \cap \ker(\varphi) \big) \ast \big( B \cap \ker(\varphi) \big)
\end{eqnarray*}
for a free group $F$. Our aim is to show $F= \langle \{a_{i} \mid i \in \mathbb{Z} \} \cup \{b_{i} \mid i \in \mathbb{Z} \} \mid \rangle$. First, we note that there are no non-trivial relations between the generators $a_{i}$, $b_{i}$ ($i \in \mathbb{Z}$) in $G'$. Every non-trivial element of $\langle \{a_{i} \mid i \in \mathbb{Z} \} \cup \{b_{i} \mid i \in \mathbb{Z} \} \mid \rangle$ written in the generators $a$, $b$ and $x$ contain the generator $x$. Thus, there are no non-trivial relations between the elements from $(A \cap \ker(\varphi)) \ast (B \cap \ker(\varphi))$ and $\langle \{a_{i} \mid i \in \mathbb{Z} \} \cup \{b_{i} \mid i \in \mathbb{Z} \} \mid \rangle$. In particular, the free product $N:=(A \cap \ker(\varphi)) \ast (B \cap \ker(\varphi)) \ast \langle \{a_{i} \mid i \in \mathbb{Z} \} \cup \{b_{i} \mid i \in \mathbb{Z} \} \mid \rangle$ from the statement of the lemma is well-defined as a subgroup of $G'$. Note that $N$ lies in the kernel of $\varphi$. It is sufficient to show that every element $w \in \ker(\varphi)$ lies in $N$. For this purpose we exemplary consider a normal form $uvx^{3}u'$ of $w$ with respect to $G'$ where $u$, $u' \in A$ and $v \in B$. We have:
\begin{eqnarray} \label{eqDarwKurosh}
w &=& uvx^{3}u'\nonumber\\
&=& \underbrace{ua^{-\varphi(u)}}_{\in A \cap \ker(\varphi)} \cdot a^{\varphi(u)}x^{-\varphi(u)} \cdot x^{\varphi(u)}b^{-\varphi(u)} \cdot \underbrace{b^{\varphi(u)} v b^{-\varphi(uv)}}_{\in B \cap \ker(\varphi)} \cdot b^{\varphi(uv)} x^{-\varphi(uv)} \cdot x^{\varphi(uv)+3}a^{-\varphi(uvx^{3})}\nonumber\\
& & \cdot \ \underbrace{a^{\varphi(uvx^{3})}u'}_{\in A \cap \ker(\varphi)}\underbrace{a^{\varphi(-uvx^{3}u')}}_{=1}
\end{eqnarray} 
Every element $a^{k}x^{-k}$ ($k \in \mathbb{Z}$) can be written in the generators $a_{i}$ ($i \in \mathbb{Z}$) since
\begin{eqnarray*}
a^{k}x^{-k} \ = \ \overset{k-1}{\underset{j=0}{\Pi}} x^{j}ax^{-j-1} \ = \ \overset{k-1}{\underset{j=0}{\Pi}} a_{-j}.
\end{eqnarray*}
Analogously every element $b^{k}x^{-k}$ ($k \in \mathbb{Z}$) can be written using the generators $b_{i}$ ($i \in \mathbb{Z}$). Therefore, \eqref{eqDarwKurosh} is a presentation of $w$ as an element of $N$.
\end{proof}

The next auxiliary lemma will be helpful at several points in this article.

\begin{lemma} \label{lemnormcleq}
Let $G$ be a group and $\varphi \colon G \rightarrow \mathbb{Z}$ a homomorphism. Further let $r$, $x \in G$ be elements with $\varphi(r)=0$ and $\varphi(x)=1$. We define $r_{i} := x^{-i}rx^{i}$. Then we have
\begin{eqnarray*}
\langle \! \langle r \rangle \! \rangle_{G} \ \ = \ \ \langle \! \langle r_{i} \mid i \in \mathbb{Z} \rangle \! \rangle_{\ker(\varphi)}.
\end{eqnarray*}
\end{lemma}

\begin{proof}
The inclusion of $\langle \! \langle r_{i} \mid i \in \mathbb{Z} \rangle \! \rangle_{\ker(\varphi)}$ in $\langle \! \langle r \rangle \! \rangle_{G}$ is trivial. For the other inclusion it is sufficient to show that for every $w \in G$ the conjugate $w^{-1} r w$ is an element of $\langle \! \langle r_{i} \mid i \in \mathbb{Z} \rangle \! \rangle_{\ker(\varphi)}$. To do that we write
\begin{eqnarray*}
w^{-1}rw \ \ = \ \  \underbrace{w^{-1} x^{\varphi(w)}}_{\in \, \ker(\varphi)} \ \underbrace{x^{-\varphi(w)} r x^{\varphi(w)}}_{= \, r_{\varphi(w)}} \ \underbrace{x^{-\varphi(w)} w}_{\in \, \ker(\varphi)}.
\end{eqnarray*}
Therefore, $w^{-1} r w$ is an element of $\langle \! \langle r_{i} \mid i \in \mathbb{Z} \rangle \! \rangle_{\ker(\varphi)}$.
\end{proof}

For the following technical lemmata and the proof of Theorem~\ref{thmedjvetgen} we introduce additional notation.

\begin{notation} \label{notindex}
Let $A$, $B$, $C$ be groups where $A$, $B$ are two locally indicable and indicable groups. We consider the epimorphism $\psi \colon A \twoheadrightarrow \mathbb{Z}$ and extend it to the epimorphism $\varphi \colon (A \ast B) \times C \twoheadrightarrow \mathbb{Z}$ by setting $\varphi_{\mid A}=\psi$ and $\varphi(b)=\varphi(d)=0$ for all $b \in B$, $d \in C$. Further let $a \in A$ be an element with $\varphi(a)=1$. We consider an arbitrary element $c \in C$ and a presentation $r$ of a non-trivial element of $A \ast B$ in normal form such that $\varphi((r,c))=0$ and $r$ has even length. Let $r_{i} := a^{-i}ra^{i}$. Then it follows by Lemma~\ref{lemnormcleq} that the normal closure of $(r,c)$ in $(A \ast B) \times C$ is the same as the normal closure of the elements $\{(r_{i},c) \mid i \in \mathbb{Z} \}$ in $\ker(\varphi)$. We define $\widetilde{A} := A \cap \ker(\varphi)$ and $B_{i} := a^{-i}Ba^{i}$ ($i \in \mathbb{Z}$). Due to Lemma~\ref{lemtestudai1} we have $\ker(\varphi)= \widetilde{A} \ast \bigast_{i \in \mathbb{Z}} B_{i}$. We denote by $\alpha_{r_{i}}$ the smallest and by $\omega_{r_{i}}$ the greatest index $j$ such that the presentation of $r_{i}$ in normal form with respect to $\widetilde{A} \ast \bigast_{i \in \mathbb{Z}} B_{i}$ contains a piece from $B_{j}$. Finally, we define for $\alpha$, $\omega \in \mathbb{Z} \cup \{\pm \infty \}$ the free products
\begin{eqnarray*}
B_{\alpha,\omega} \ := \ \overset{\omega}{\underset{i= \alpha}{\bigast}} B_{i} \ \text{ if } \ \alpha \leqslant \omega \ \ \ \text{ and } \ \ \ B_{\alpha,\omega} \ := \ \{1\} \ \text{ if } \ \alpha > \omega.
\end{eqnarray*}  
\end{notation}

The next lemma serves the preparation of the presentation of $\ker(\varphi) / \langle \! \langle (r_{i},c) \mid i \in \mathbb{Z} \rangle \! \rangle$ as an amalgamated product (see Lemma~\ref{lemamalgPr}).

\begin{lemma} \label{lemC}
Let $C$ be a group. Using the notation of Notation~\ref{notindex} we consider the groups
\begin{eqnarray*}
G_{i} \ = \ (\widetilde{A} \ast B_{\alpha_{r_{i}},\omega_{r_{i}}}) \times C, \ \ N_{i} \ = \ \langle \! \langle (r_{i},c) \rangle \! \rangle_{G_{i}} \ \ \text{and} \ \ C_{i} \ = \ C/N_{i}.
\end{eqnarray*}
We have $C_{k} = C_{\ell}$ for all $k$, $\ell \in \mathbb{Z}$ and can therefore define $\widetilde{C} := C_{j}$ ($j \in \mathbb{Z}$).
\end{lemma}

\begin{proof}
Let $k$, $\ell \in \mathbb{Z}$. Further let $w$ be an element of $N_{k}$. We show that $w$ is also an element of $N_{\ell}$. There exist numbers $m \in \mathbb{N}$, $\varepsilon_{i} \in \mathbb{Z}$ and elements $u_{i} \in G_{k}$ ($i \in \{1,2,\dots,m\}$) with
\begin{eqnarray} \label{eqw1}
w \ \ = \ \ \overset{m}{\underset{i=1}{\Pi}} u_{i}^{-1} (r_{k},c)^{\varepsilon_{i}} u_{i}.
\end{eqnarray}
We consider the isomorphism $\zeta \colon G_{k} \rightarrow G_{\ell}$ which is given by $\zeta(x)=a^{-\ell+k}x a^{\ell-k}$. Note that, $\zeta(u_{i}) \in G_{\ell}$, $\zeta((r_{k},c))=(r_{\ell},c)$ and that $\zeta_{\mid C}$ is the identity on $C$. In particular, we have $w = \zeta(w)$. Applying $\zeta$ to \eqref{eqw1} we get
\begin{eqnarray*}
w \ \ = \ \ \zeta(w) \ \ = \ \ \overset{m}{\underset{i=1}{\Pi}} \zeta(u_{i})^{-1} (r_{\ell},c)^{\varepsilon_{i}} \zeta(u_{i}) \ \ \in \ \ \langle \! \langle (r_{\ell},c) \rangle \! \rangle_{G_{\ell}},
\end{eqnarray*}
thus $w \in N_{\ell}$. It follows $N_{k} = N_{\ell}$ and therefore $C_{k} = C_{\ell}$ for all $k$, $\ell \in \mathbb{Z}$. 
\end{proof}

With the next lemma we prove two embeddings. Again, we use the notations of Notation~\ref{notindex} and Lemma~\ref{lemC}.

\begin{lemma} \label{lememb12}
Let $m \in \mathbb{Z}$ and $k$, $\ell \in \mathbb{Z} \cup \{\pm \infty\}$ with $k \leqslant \alpha_{r_{m}}$ and $\ell \geqslant \omega_{r_{m}}$. We define $P_{n}:=(\widetilde{A} \ast B_{\alpha_{r_{n}},\omega_{r_{n}}-1} \times \widetilde{C})$ where $n \in \mathbb{Z}$. Then we have the following embeddings:
\begin{itemize}
\item[$(i)$] $P_{m} \hookrightarrow \big( ( \widetilde{A} \ast B_{k, \ell}) \times C \big) / \langle \! \langle (r_{m},c) \rangle \! \rangle$
\item[$(ii)$] $P_{m+1} \hookrightarrow \big( ( \widetilde{A} \ast B_{k, \ell}) \times C \big) / \langle \! \langle (r_{m},c ) \rangle \! \rangle$
\end{itemize}
\end{lemma}

\begin{proof}
For the proof of $(i)$ recall that r is a presentation of a non-trivial element of $A \ast B$ in normal form such that $r$ is of even length. It follows that $r_{m}$ is cyclically reduced and uses the factor $B_{m}$ along with at least one further factor $\widetilde{A}$ or $B_{k}$ with $k \neq m$. In particular, $r_{m}$ uses at least two factors of the free product $\widetilde{A} \ast ( \bigast_{k \leqslant i  \leqslant \ell} B_{i})$. By the definition of $\omega_{r_{m}}$, the presentation $r_{m}$ uses the factor $B_{\omega_{r_{m}}}$. Since $\alpha_{r_{m}}$ and $\omega_{r_{m}}$ are defined such that $r_{m}$ uses exclusively factors from $\widetilde{A}$ or $B_{i}$ with $\alpha_{r_{m}} \leqslant i \leqslant \omega_{r_{m}}$, the second used factor lies in $\widetilde{A} \ast B_{k, \omega_{r_{m}}-1}$. Thus, the normal form of $r_{m}$ with respect to the free product $(\widetilde{A} \ast B_{k,\omega_{r_{m}}-1}) \ast B_{\omega_{r_{m}},\ell}$ has a length of at least $2$. Finally, the embedding of $P_{m}$ in $\big( ( \widetilde{A} \ast B_{k, \ell}) \times C \big) / \langle \! \langle (r_{m},c) \rangle \! \rangle$ follows by Lemma~\ref{lemmachi}. The embedding $(ii)$ can be proved analogously.
\end{proof} 

Now we are able to deduce a presentation of $\ker(\varphi) / \langle \! \langle r_{i} \mid i \in \mathbb{Z} \rangle \! \rangle$ as an amalgamated product. This presentation plays an important role in the proof of our small generalization (see Lemma~\ref{lemmachi}) of the Freiheitssatz for locally indicable groups (see Theorem~\ref{HowFS}). For an illustration of the statement of the next lemma see Figure~\ref{fig1}. 

\begin{figure}
\centering
\includegraphics[scale=0.12]{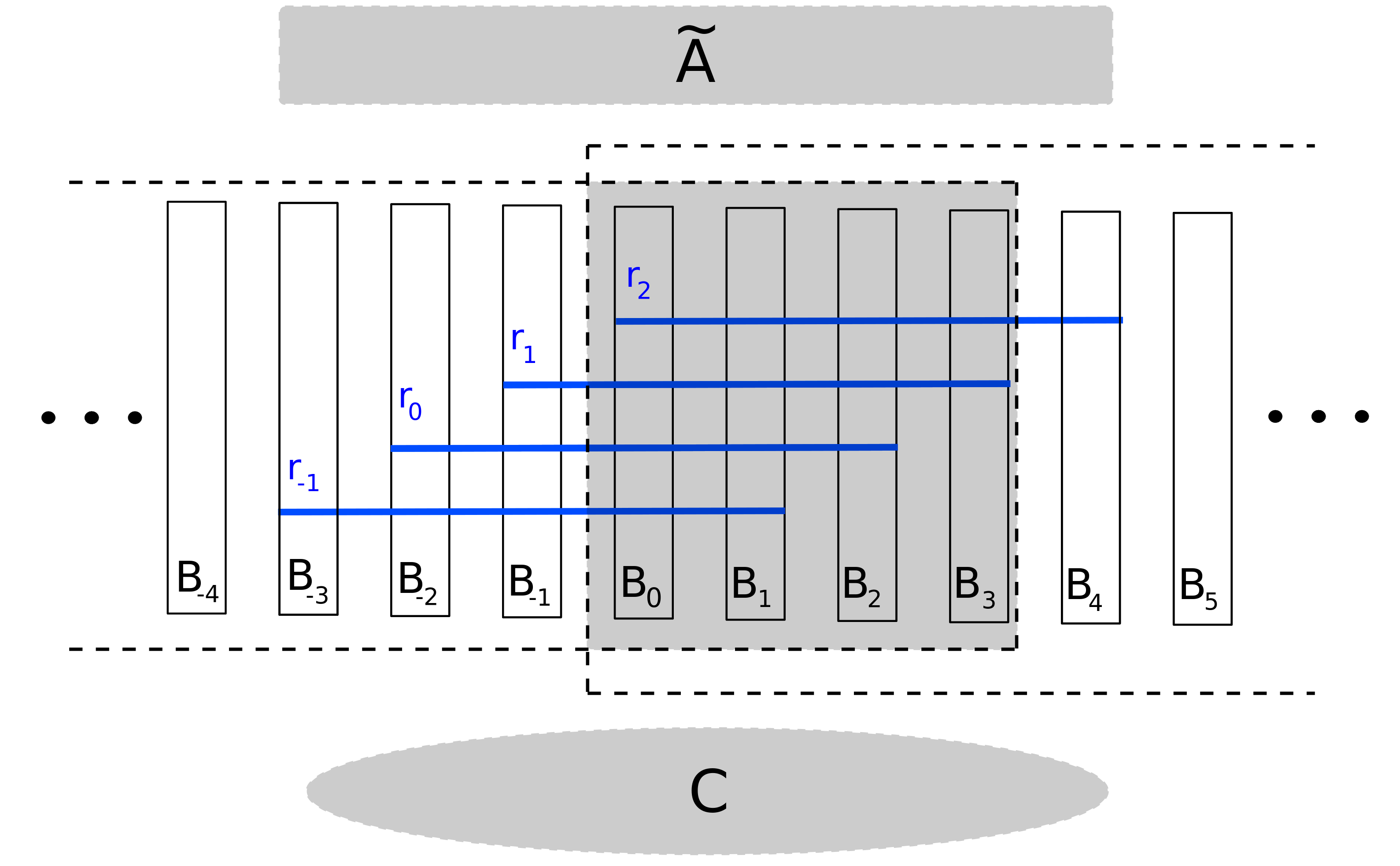}
\caption{Illustration of Lemma~\ref{lemamalgPr} for $i=-1$, $j=2$, $\alpha_{2}=0$ and $\omega_{2}=4$. (The amalgamated group $P_{2}$ is illustrated by the greyed out areas.)}
\label{fig1}
\end{figure}

\begin{lemma} \label{lemamalgPr}
We use the notation introduced in Notation~\ref{notindex} and Lemma~\ref{lemC}. Let $i$, $j \in \mathbb{Z}$ be two numbers with $i<j$. We define $P_{j}:=(\widetilde{A} \ast B_{\alpha_{r_{j}},\omega_{r_{j-1}}}) \times \widetilde{C}$. Then we have
\begin{center}
$\big( ( \widetilde{A} \ast B_{-\infty,\infty}) \times C \big) / \langle \! \langle (r_{i},c),(r_{i+1},c),\dots,(r_{j},c) \rangle \! \rangle$ \\
$\cong \ \big( ( \widetilde{A} \ast B_{-\infty,\omega_{r_{j-1}}}) \times C \big) / \langle \! \langle (r_{i},c), \dots,(r_{j-1},c)\rangle \! \rangle \ \underset{P_{j}}{\ast} \ \big( ( \widetilde{A} \ast B_{\alpha_{r_{j}}, \infty} ) \times C \big) / \langle \! \langle (r_{j},c) \rangle \! \rangle.$
\end{center}
\end{lemma}

\begin{proof}
Our proof will be by induction over $j$ for a fixed element $i \in \mathbb{Z}$. As the induction base we consider the case $j=i+1$. Thus, we have to show:
\begin{gather}
\big( (\widetilde{A} \ast B_{-\infty,\infty}) \times C \big) / \langle \! \langle (r_{i},c),(r_{j},c) \rangle \! \rangle \nonumber \\
\cong \big( (\widetilde{A} \ast B_{-\infty,\omega_{r_{j-1}}}) \times C \big)  / \langle \! \langle(r_{i},c) \rangle \! \rangle \ \underset{P_{j}}{\ast} \  \big( (\widetilde{A} \ast B_{\alpha_{r_{j}},\infty}) \times C \big) / \langle \! \langle(r_{j},c) \rangle \! \rangle \label{eqamal}
\end{gather}
First, we show that the amalgamated product \eqref{eqamal} is well-defined by proving that $P_{j}$ embeds into both factors. The embedding of $P_{j}$ in the right factor follows from Lemma~\ref{lememb12} $(i)$ for $m=j$. The embedding in the left factor follows from Lemma~\ref{lememb12} $(ii)$ for $m=j-1$ because of $i=j-1$. Now the isomorphism between the two sides of \eqref{eqamal} can be easily shown by giving a common presentation.

For the induction step $(j \rightarrow j+1)$ we define
\begin{eqnarray*}
G &:=& \big( (\widetilde{A} \ast B_{-\infty,\omega_{r_{j}}}) \times C \big) \, / \, \langle \! \langle (r_{i},c),(r_{i+1},c), \dots , (r_{j},c) \rangle \! \rangle \\
&\cong & \big( (\widetilde{A} \ast B_{-\infty,\omega_{r_{j}}}) \times \widetilde{C} \big) \, / \, \langle \! \langle (r_{i},c),(r_{i+1},c), \dots , (r_{j},c) \rangle \! \rangle
\end{eqnarray*}
and want to show the presentation
\begin{eqnarray} \label{eqIndSchr2}
& & \big( (\widetilde{A} \ast B_{-\infty,\infty}) \times C \big) \, / \, \langle \! \langle (r_{i},c), (r_{i+1},c), \dots , (r_{j+1},c) \rangle \! \rangle \nonumber \\ 
&\cong & \ G \ \underset{P_{j+1}}{\ast} \ \big( (\widetilde{A} \ast B_{\alpha_{r_{j+1}},\infty}) \times C \big) \, / \, \langle \! \langle(r_{j+1},c) \rangle \! \rangle,
\end{eqnarray}
where $P_{j+1} = ( \widetilde{A} \ast B_{\alpha_{r_{j+1}},\omega_{r_{j}}}) \times \widetilde{C}$. To do that we show that $P_{j+1}$ embeds into the left and the right factor. Note that the embedding of $P_{j+1}$ in the right factor follows from Lemma~\ref{lememb12} (i) with $m=j+1$. It remains to show the embedding of $P_{j+1}$ in the left factor $G$: Because of Lemma~\ref{lememb12} (ii) with $m=j$ the group $P_{j+1}$ embeds into $\big( ( \widetilde{A} \ast B_{\alpha_{r_{j}},\infty} ) \times C \big) / \langle \! \langle (r_{j},c) \rangle \! \rangle$. The group $\big( ( \widetilde{A} \ast B_{\alpha_{r_{j}},\infty} ) \times C \big) / \langle \! \langle (r_{j},c) \rangle \! \rangle$ embeds as a factor of the amalgamated product from the induction hypothesis into
\begin{eqnarray*}
G' \ := \ \big( ( \widetilde{A} \ast B_{-\infty,\infty}) \times C) / \langle \! \langle (r_{i},c), (r_{i+1},c), \dots, (r_{j},c) \rangle \! \rangle. 
\end{eqnarray*}
We define $\varphi \colon P_{j+1} \rightarrow G$ as the composition of the trivial embedding of $P_{j+1}$ into $(\widetilde{A} \ast B_{-\infty,\omega_{r_{j}}}) \times C$ and the canonical homomorphism from $(\widetilde{A} \ast B_{-\infty,\omega_{r_{j}}}) \times C$ in the factor group $G$. Altogether we get the following commuting diagram:
\begin{displaymath}
	\xymatrix@C=2cm@R=2cm{
        P_{j+1} \ \ar@{^{(}->}[rr] \ar@{-->}[rrd]^{\text{\large $\varphi$}} &&  \ G' \\
       && \ G \ar@{->}[u] }
\vspace{0.3cm}
\end{displaymath}
Finally, we can read that $\varphi$ is a monomorphism and $P_{j+1}$ embeds into $G$.
\end{proof}

Analogously one can prove the following ``reflected'' version of Lemma~\ref{lemamalgPr}.

\begin{lemma} \label{lemamalgPrrefl}
We use the notation introduced in Notation~\ref{notindex} and Lemma~\ref{lemC}. Let $i$, $j\in \mathbb{Z}$ be two numbers and let $i < j$. We define $P_{i+1}:=(\widetilde{A} \ast B_{\alpha_{r_{i+1},\omega_{r_{i}}}}) \times \widetilde{C}$. Then we have:
\begin{gather*}
\big( (\widetilde{A} \ast B_{-\infty,\infty}) \times C \big) / \langle \! \langle (r_{i},c), (r_{i+1},c), \dots , (r_{j},c) \rangle \! \rangle \\
\cong \ \big( (\widetilde{A} \ast B_{-\infty,\omega_{r_{i}}}) \times C \big) / \langle \! \langle (r_{i},c) \rangle \! \rangle \underset{P_{i+1}}{\ast} \big( (\widetilde{A} \ast B_{\alpha_{r_{i+1}},\infty}) \times C \big) / \langle \! \langle (r_{i+1},c), \dots , (r_{j},c) \rangle \! \rangle \\
\end{gather*}
\end{lemma}

The next lemma is an auxiliary statement that will be used at several places of the following proofs.

\begin{lemma} \label{lemassum12}
Let $X$, $Y$ be two groups, $W:=X \ast_{Z} Y$ an amalgamated product and $r$, $s$ two elements of $X$. Further let at least one of the following assumptions be true:
\begin{itemize}
\item[$(i)$] The subgroup $Z$ lies in the center of $Y$.
\item[$(ii)$] The factor $Y$ is a free group and $Z$ is a maximal cyclic subgroup of $Y$.
\end{itemize}
Then $r$ and $s$ are conjugate in $X$ if and only if they are conjugate in $W$.
\end{lemma}

\begin{proof}
Note that $r$ and $s$ are conjugate in $W$ if they are already conjugate in $X$. It remains to show that $r$ and $s$ are conjugate in $X$ if the are conjugate in $W$. Thus, let $w \in W$ be an element with $r=w^{-1}sw$. We use the unique length of normal forms of amalgamated products to show that $w$ is an element of $X$: For $w \in X$ there is nothing to prove. So we can consider the normal form
\begin{eqnarray}
w \ = \ g_{1} h_{1} g_{2} h_{2} \cdots g_{n-1}h_{n-1}g_{n}
\end{eqnarray}
of $w$, where $g_{i} \in (X \backslash Z) \cup \{1\}$, $h_{i} \in Y \backslash Z$ and $n \geqslant 2$. To avoid a larger case analysis, we allow $g_{1}$ and $g_{n}$ to be trivial, but all other $g_{i}$ shall be non-trivial. We choose an element $w$ which has minimal length among all elements with the described properties. Then we have
\begin{eqnarray} \label{eqNF} 
r=g_{n}^{-1}h_{n-1}^{-1}g_{n-1}^{-1} \cdots h_{2}^{-1}g_{2}^{-1} \overbrace{h_{1}^{-1} \underbrace{g_{1}^{-1} \ s \ g_{1}}_{=: \, p} h_{1}}^{=: \, q} g_{2}h_{2} \cdots g_{n-1}h_{n-1}g_{n}.
\end{eqnarray}
As a product of elements of $X$ the element $p:=g_{1}^{-1}sg_{1}$ is also an element of $X$. Thus it suffices to consider the following two cases.

\noindent\underline{Case 1.} Let $p$ be an element of $X \backslash Z$.\\
Note that $h_{1}$ cannot be trivial because of $n \geqslant 2$. On the one hand, by the right side of \eqref{eqNF}, the element $r$ possesses a length greater or equal $3$ with respect to the normal form of the amalgamated product $W$. On the other hand $r$ is an element of $X$ and possesses therefore a length of $1$ which is a contradiction.

\noindent\underline{Case 2.} Let $p$ be an element of $Z$.\\
We define $q:=h_{1}^{-1}ph_{1}$. As a product of elements of $Y$ the element $q$ lies in $Y$. By assumption, $r$ is no element of $Y \backslash Z$. If $q$ is an element of $Y \backslash Z$ the normal form of $r$ would have length at least three. This is a contradiction. Thus we have $q \in Z$.

\underline{Case 2.1.} Let assumption $(i)$ of Lemma~\ref{lemassum12} be valid.\\
Since $p$ is in the center of $Y$ und therefore commutes with $h_{1}$, we have
\begin{eqnarray*}
r=g_{n}^{-1}h_{n-1}^{-1}g_{n-1}^{-1} \cdots h_{2}^{-1}g_{2}^{-1}g_{1}^{-1} \ s \ g_{1}g_{2}h_{2} \cdots g_{n-1}h_{n-1}g_{n}.
\end{eqnarray*}
Thus
\begin{eqnarray} \label{eqw'}
w'  :=\underbrace{g_{1}g_{2}}_{=:g'_{1}} \underbrace{h_{2}}_{=:h'_{1}} \cdots \underbrace{g_{n-1}}_{=: \, g'_{n-2}} \ \ \underbrace{h_{n-1}}_{=: \, h'_{n-2}} \ \ \underbrace{g_{n}}_{=: \, g'_{n-1}}
\end{eqnarray}
fulfils the equation $r = w'^{-1}sw'$. Note that $w'$ contains one $H$-piece less than $w$. This contradicts the choice of $w$ as an element of minimal length with $r=w^{-1}sw$.

\underline{Case 2.2.} Let assumption $(ii)$ of Lemma~\ref{lemassum12} be valid.\\
Let $z$ be a generator of $Z$. Then, because of $p$, $q \in Z$ and $q= h_{1}^{-1}ph_{1}$, we have the following in $Y$:
\begin{eqnarray*}
z^{\alpha} \ = \ h_{1}^{-1}z^{\beta} h_{1} \ \ \text{for elements} \ \ \alpha, \beta \in \mathbb{Z} \backslash \{0\} \ \ \text{and} \ \ h_{1} \in Y \backslash Z
\end{eqnarray*}
By assumption $(ii)$ of Lemma~\ref{lemassum12}, $Y$ is a free group. Therefore, the presentation $\langle \mathcal{Y} \mid z \rangle$, where $\mathcal{Y}$ is the basis of $Y$, is aspherical by Proposition~\ref{propasphone}. By Proposition~\ref{propasphpairs}, we have $\alpha=\beta$. Thus, $h_{1}$ is in the centralizer of the element $z^{\alpha}$. Since the centralizer of a non-trivial element of a free group is cyclic and since by assumption (ii) $Z=\langle z \rangle$ is a maximal cyclic subgroup of $Y$, it follows $h_{1}=z^{\ell}$ for an $\ell \in \mathbb{Z}$. This contradicts $h_{1} \in Y \backslash Z$. 
\end{proof}

\subsection{Generalization of a result by M. Edjvet}

In this subsection we want to prove Theorem~\ref{thmedjvetgen}. Our proof will be by a combination of M. Edjvet's methods used for the proof of Theorem~\ref{Edj} and some results from \cite{MA} and the previous subsections. 

\begin{remark} \label{rempres}
Let $G$ be a free group with basis $\mathcal{X}$ and let $H$ be a group. Further let $\mathcal{R}$ be a set of elements from $H$ and let $S$ be a set of elements from $G \ast H$. We define $\widetilde{H} := H / \langle \! \langle \mathcal{R} \rangle \! \rangle$ and the canonical homomorphism
\begin{eqnarray*}
\varphi \colon G \ast H \ \ \rightarrow \ \ (G \ast H ) / \langle \! \langle \mathcal{R} \rangle \! \rangle_{G \ast H} \ \ \cong \ \ G \ast \widetilde{H}.
\end{eqnarray*}
Then we have:
\begin{eqnarray*}
(G \ast H)/\langle \! \langle \mathcal{R},\mathcal{S} \rangle \! \rangle &\cong& \varphi(G \ast H)/\varphi(\langle \! \langle \mathcal{R},\mathcal{S} \rangle \! \rangle_{G \ast H}) \ \ \ \ \ (\text{because of} \ \operatorname{ker}(\varphi) \subseteq \langle \! \langle \mathcal{R}, \mathcal{S} \rangle \! \rangle_{G \ast H}) \\
& \cong &  \varphi(G \ast H)/\langle \! \langle \varphi(\mathcal{S}) \rangle \! \rangle_{\varphi(G \ast H)} \ \ \ \ \ \, (\text{because of} \ \mathcal{R} \subseteq \operatorname{ker}(\varphi)) \\
& \cong & (G \ast \widetilde{H}) / \langle \! \langle \varphi(\mathcal{S}) \rangle \! \rangle_{G \ast \widetilde{H}}
\end{eqnarray*}
\end{remark}

In view of Remark~\ref{rempres} we introduce the following notation.

\begin{notation} \label{notpres}
Let $G= \langle \mathcal{X} \mid \rangle$ and $H = \langle \mathcal{Y} \mid \rangle$ be free groups with basis $\mathcal{X}$ and $\mathcal{Y}$. Further let $\mathcal{R}_{1}$, $\mathcal{R}_{2}$ be subsets of $H$ and let $S$ be a subset of $G \ast H$. We set $\widetilde{H} := H / \langle \! \langle \mathcal{R}_{1} \rangle \! \rangle$ and define the canonical homomorphism
\begin{eqnarray*}
\varphi \colon G \ast H \ \ \rightarrow \ \ (G \ast H) / \langle \! \langle \mathcal{R}_{1} \rangle \! \rangle_{G \ast H} \ \ \cong \ \ G \ast \widetilde{H}.
\end{eqnarray*}
Further we define for $\bar{H} := \widetilde{H} / \langle \! \langle \varphi(\mathcal{R}_{2}) \rangle \! \rangle$ the canonical homomorphism
\begin{eqnarray*}
\psi \colon G \ast \widetilde{H} \ \ \rightarrow \ \ (G \ast \widetilde{H} ) / \langle \! \langle \varphi(\mathcal{R}_{1}) \rangle \! \rangle_{G \ast \widetilde{H}} \ \ \cong \ \ G \ast \bar{H}. 
\end{eqnarray*}
Then we use the following notation:
\begin{eqnarray*}
& & \langle \mathcal{X} \cup \mathcal{Y} \, \mid \, \mathcal{R}_{1} \cup \mathcal{R}_{2} \cup \mathcal{S} \rangle \ = \ \langle G,H \, \mid \, \mathcal{R}_{1} \cup \mathcal{R}_{2} \cup \mathcal{S} \rangle \ = \ \langle G, \widetilde{H} \, \mid \, \varphi(\mathcal{R}_{2}) \cup \varphi(\mathcal{S}) \rangle \\
&=&  \langle G, \widebar{H} \, \mid \, \psi(\varphi(\mathcal{S})) \rangle
\end{eqnarray*}
If there is no danger of misunderstandings, we omit mentioning the homomorphism $\varphi$ and $\psi$ by writing - in abuse of notation - $w$ for the images $\varphi(w) \in G \ast \widetilde{H}$ or $\psi(\varphi(w)) \in G \ast \bar{H}$ of an element $w \in G \ast H$.
\end{notation}

We start our proof of Theorem~\ref{thmedjvetgen} by considering the case $|\mathcal{J}|=2$, i.\,e. we prove the following:

\begin{proposition} \label{propedjgen2}
Let $A$, $B$ be indicable as well as locally indicable groups. Further let $C$ be a group. Then $(A \ast B) \times C$ possesses the Magnus property if and only if $A \times C$ and $B \times C$ possess the Magnus property.
\end{proposition}

\begin{proof}
First, we show by contradiction that the groups $A \times C$ and $B \times C$ possess the Magnus property if $(A \ast B) \times C$ possesses the Magnus property. W.\,l.\,o.\,g. we assume that $A \times C$ does not possess the Magnus property. Then by Lemma~\ref{Lemnus} there exists an element $(a,c) \in A \times C$ with $\langle \! \langle (a,c) \rangle \! \rangle_{A \times C} = \langle \! \langle (a^{-1},c) \rangle \! \rangle_{A \times C}$, but $a$ is neither conjugate to $a^{-1}$ in $A$ nor is $c$ conjugate to $c^{-1}$ in $C$. Because of the unique length of normal forms for free products $a$ cannot be conjugate to $a^{-1}$ in $A \ast B$. So, by Lemma~\ref{Lemnus}, $(A \ast B) \times C$ does not possess the Magnus property.

It remains to prove the Magnus property of $(A \ast B) \times C$ under the condition that $A \times C$ and $B \times C$ possess the Magnus property. Under that condition $A$, $B$ and $C$ also possess the Magnus property. Since $A$ and $B$ are in addition locally indicable, the group $A \ast B$ possesses the Magnus property by Corollary~\ref{corEdj}. Due to Lemma~\ref{Lemnus} $(A \ast B) \times C$ possesses the Magnus property if and only if for all $r \in (A \ast B) \backslash \{1\}$ and $c \in C \backslash \{1\}$ we have:
\begin{eqnarray} \label{eqaim2}
\langle \! \langle (r,c) \rangle \! \rangle_{(A \ast B) \times C} \ = \ \langle \! \langle (r^{-1},c) \rangle \! \rangle_{(A \ast B) \times C} \ \ \ \ \Longrightarrow \ \ \ \ ( r \sim_{(A \ast B)} r^{-1} \ \ \lor \ \ c \sim_{C} c^{-1})
\end{eqnarray}
Let $r \in (A \ast B) \backslash \{1\}$ and $c \in C \backslash \{1\}$ be two elements with
\begin{eqnarray} \label{eqcond1}
\langle \! \langle (r,c) \rangle \! \rangle_{(A \ast B) \times C} \ \ = \ \ \langle \! \langle ( r^{-1},c) \rangle \! \rangle_{(A \ast B) \times C}.
\end{eqnarray}
We want to show implication \eqref{eqaim2}. First, we consider the case that $r$ has length $|r|=1$ with respect to the normal form with respect to $A \ast B$. W.\,l.\,o.\,g. we assume $r \in A$. Let $\pi \colon (A \ast B) \times C \rightarrow A \times C$ be the canonical projection. Because of \eqref{eqcond1} the elements $\pi((r,c))=(r,c)$ and $\pi((r^{-1},c)) = (r^{-1},c)$ possess the same normal closure in $A \times C$. Since $A \times C$ possesses the Magnus property by assumption we have $r \sim r^{-1}$ in $A$ (thus also in $A \ast B$) or $c \sim c^{-1}$ in $C$ due to Lemma~\ref{Lemnus}.

In the following, we assume $|r| \geqslant 2$. Let $r=u(1)v(1)u(2)v(2)\dots u(n)v(n)$ be the normal form of $r$ with respect to $A \ast B$. Since normal closures are invariant under conjugation, we may assume w.\,l.\,o.\,g. that $r$ is cyclically reduced and begins with an $A$-piece. We define the subgroups
\begin{eqnarray*}
A' &:=& \langle u(1),u(2),\dots ,u(n) \rangle \ \ \text{of} \ \ A \ \ \text{and} \\
B' &:=& \langle v(1),v(2),\dots, v(n) \rangle \ \ \text{of} \ \ B.
\end{eqnarray*}
Further let $R := \langle \! \langle (r,c) \rangle \! \rangle_{(A \ast B) \times C}$, $R' := \langle \! \langle (r,c) \rangle \! \rangle_{(A' \ast B') \times C}$, $N:= C \cap R'$ and $\widetilde{C} := C / N$. 

Our next intermediate aim is to prove the following:
\begin{eqnarray} \label{eqIsomAStrich}
((A \ast B) \times C) / R \ \ \cong \ \ (A \times \widetilde{C}) \ \underset{A' \times \widetilde{C}}{\ast} \ ((A' \ast B') \times C) / R' \ \underset{B' \times \widetilde{C}}{\ast} \ (B \times \widetilde{C})
\end{eqnarray}
Note that the amalgamated product is well-defined because the embeddings in the middle factor follow by Lemma~\ref{lemmachi}. The embeddings in the outer factors are embeddings as subgroups. It remains to show that both sides of \eqref{eqIsomAStrich} are isomorphic by finding a common presentation of both sides. Let $\pi \colon A \ast B \ast C \rightarrow A \ast B \ast \widetilde{C}$ and $\pi' \colon A' \ast B' \ast C \rightarrow A' \ast B' \ast \widetilde{C}$ be the canonical projections. For the right side of \eqref{eqIsomAStrich} we write using Notation~\ref{notpres} and Tietze transformations:
\begin{eqnarray*}
&& (A \times \widetilde{C}) \ \underset{A' \times \widetilde{C}}{\ast} \ ((A' \ast B') \times C) / R' \ \underset{B' \times \widetilde{C}}{\ast} \ (B \times \widetilde{C})\\
&=& (A \times \widetilde{C}) \ \underset{A' \times \widetilde{C}}{\ast} \ \langle A',B',C \, \mid \, [A',C], \, [B',C], \, N, \, rc \rangle \ \underset{B' \times \widetilde{C}}{\ast} \ (B \times \widetilde{C})\\
&=& (A \times \widetilde{C}) \ \underset{A' \times \widetilde{C}}{\ast} \ \langle A',B',(C/N) \, \mid \, \pi'([A',C]), \, \pi'([B',C]), \, \pi'(rc) \rangle \ \underset{B' \times \widetilde{C}}{\ast} \ (B \times \widetilde{C})\\
&=& \langle A, \widetilde{C} \, \mid \, [A, \widetilde{C}] \rangle  \underset{A' \times \widetilde{C}}{\ast}  \langle A',B',\widetilde{C} \, \mid \, [A',\widetilde{C}], \, [B',\widetilde{C}], \, \pi'(rc) \rangle  \underset{B' \times \widetilde{C}}{\ast}  \langle B, \widetilde{C} \, \mid \, [B, \widetilde{C}] \rangle\\
&=&  \langle A,A',B,B',\widetilde{C} \, \mid \, [A,\widetilde{C}], \, [A',\widetilde{C}], \, [B,\widetilde{C}], \, [B',\widetilde{C}], \, \pi'(rc) \rangle\\
&\overset{\bf{(\ast)}}{=}& \langle A,B,\widetilde{C} \, \mid \, [A,\widetilde{C}], \, [B,\widetilde{C}], \, \pi'(rc) \rangle = \langle A,B,\widetilde{C} \, \mid \, [A,\widetilde{C}], \, [B,\widetilde{C}], \, \pi(rc) \rangle
\end{eqnarray*}
Recall that we use Notation~\ref{notpres}. Thus, the omission of $A'$, $B'$, $[A',\widetilde{C}]$ and $[B',\widetilde{C}]$ in step $\bf{(\ast)}$ is justified because $A'$ and $B'$ are subgroups of $A$ and $B$. In the last step we used that $\pi_{\mid A' \ast B' \ast \widetilde{C}}$ is the projection $\pi'$. For the left side of \eqref{eqIsomAStrich} we write:
\begin{eqnarray*}
&& ((A \ast B) \times C) / R  \ = \  \langle A,B,C \, \mid \, [A,C], \, [B,C], \, rc \rangle \ = \ \langle A,B,C \, \mid \, [A,C], \, [B,C], \, N, \, rc \rangle\\
&=& \langle A,B,(C/N) \, \mid \, \pi([A,C]), \, \pi([B,C]), \, \pi(rc) \rangle \ = \ \langle A,B,\widetilde{C} \, \mid  \, [A,\widetilde{C}], \, [B,\widetilde{C}], \, \pi(rc) \rangle
\end{eqnarray*}
This finishes the proof of \eqref{eqIsomAStrich}.

As a factor of the amalgamated product in \eqref{eqIsomAStrich} the group $\big( ( A' \ast B') \times C \big) / R'$ embeds into $\big( ( A \ast B) \times C \big) / R$. Since the element $(r^{-1},c)$ is an element of $(A' \ast B') \times C \subseteq (A \ast B) \times C$ which is trivial in $\big( ( A \ast B) \times C \big) /R$, it is also trivial in $\big((A' \ast B') \times C \big) / R'$. By symmetry, we have
\begin{eqnarray}
\langle \! \langle (r,c) \rangle \! \rangle_{(A' \ast B') \times C} \ \ = \ \ \langle \! \langle (r^{-1},c) \rangle \! \rangle_{(A' \ast B') \times C}.
\end{eqnarray}
Note, that to prove \eqref{eqaim2} it is sufficient to show the conjugation of $(r,c)$ to $(r^{-1},c)$ or $(r^{-1},c)^{-1}$ in $(A' \ast B') \times C$, since the conjugation of $r$ to $r^{-1}$ in $A' \ast B'$ also implies the conjugation of $r$ to $r^{-1}$ in $A \ast B$. The remaining proof will be by induction over $|r|$.

For the base of induction we consider the case $|r|=2$. W.\,l.\,o.\,g. let $r=u(1)v(1)$. Then we have $(A' \ast B') \times C \cong F_{2} \times C$, where $F_{2}$ is the free group of rank $2$. Since $A \times C$ possesses the Magnus property by assumption and since $A$ is indicable, the group $\mathbb{Z} \times C$ possesses the Magnus property by Lemma~\ref{lemindC}. Because of Theorem~\ref{thmdirect}, $F_{2} \times C$ also possesses the Magnus property.

In the following we consider two cases for the induction step, but before doing that we make some further remarks. Normal closures are invariant under conjugation. Therefore, we can assume w.\,l.\,o.\,g. that the normal form of $r$ with respect to $A' \ast B'$ has even length. Since $A'$ and $B'$ are indicable, there exist epimorphisms
\begin{eqnarray*}
\varphi \colon A' \twoheadrightarrow \mathbb{Z} \ \ \ \text{and} \ \ \ \psi \colon B' \twoheadrightarrow \mathbb{Z}.
\end{eqnarray*}
For any element $w=u(1)v(1)u(2)v(2) \dots u(n)v(n) \in A \ast B$ in normal form we define
\begin{eqnarray*}
w_{A}:=u(1)u(2)\dots u(n) \ \ \ \ \text{and} \ \ \ \ w_{B} := v(1)v(2)\dots v(n).
\end{eqnarray*}

\noindent\underline{Case 1.} Let $\varphi(r_{A})=0$ or $\psi(r_{B}) = 0$.\\
W.\,l.\,o.\,g. we assume $\varphi(r_{A})=0$. Let $\zeta \colon (A' \ast B') \times C \twoheadrightarrow \mathbb{Z}$ be the epimorphism given by $\zeta(w,d) = \varphi(w_{A})$. Then $(r,c)$ is an element of $\ker(\zeta)$. Further let $a \in A'$ be an element with $\varphi(a)=1$. We define $\widetilde{A} := \ker(\varphi)$ and $B_{i} := a^{-i} B' a^{i}$ ($i \in \mathbb{Z}$). Since $\zeta(c)=0$ for all $c \in C$, we can apply Lemma~\ref{lemtestudai1} and write
\begin{eqnarray} \label{eqkerzeta}
\ker(\zeta) \ \ = \ \ (\widetilde{A} \ast \underset{i \in \mathbb{Z}}{\bigast} B_{i}) \times C.
\end{eqnarray}
For $r_{i} := a^{-i} r a^{i}$ we have 
\begin{eqnarray*}
\langle \! \langle (r_{i},c) | i \in \mathbb{Z} \rangle \! \rangle_{\ker(\zeta)} = \langle \! \langle (r_{i}^{-1},c) | i \in \mathbb{Z} \rangle \! \rangle_{\ker(\zeta)}.
\end{eqnarray*}
by Lemma~\ref{lemnormcleq}. Note that because of $r_{i} = a^{-i} r a^{i}$ and Notation~\ref{notindex} we have
\begin{eqnarray} \label{eqao1}
\omega_{r_{k+\ell}} \ = \ \omega_{r_{k}} + \ell \ \ \ \text{and} \ \ \ \alpha_{r_{k+\ell}} \ = \ \alpha_{r_{k}} + \ell \ \ \ \ \text{for all} \ \ k, \ \ell \in \mathbb{Z}.
\end{eqnarray}
We choose indices $i$, $j \in \mathbb{Z}$ with $i \leqslant j$ such that $(r_{0}^{-1},c)$ is an element of
\begin{eqnarray*}
\langle \! \langle (r_{i},c),(r_{i+1},c), \dots, (r_{j},c) \rangle \! \rangle_{\ker(\zeta)}
\end{eqnarray*} 
and such that $j-i$ is minimal with this property. Our aim is to show $i=j=0$. To the contrary, assume $i < j$. Using Lemma~\ref{lemamalgPr} we deduce $\omega_{r_{0}^{-1}}=\omega_{r_{0}} \geqslant \omega_{r_{j}}$ since for $\omega_{r_{0}} < \omega_{r_{j}}$ the element $(r_{0}^{-1},c)$ would be an element of the left factor of the amalgamated product of Lemma~\ref{lemamalgPr} and would therefore be trivial in that factor. This contradicts the minimality of $i-j$. With the help of Lemma~\ref{lemamalgPrrefl} we analogously deduce $\alpha_{r_{0}^{-1}} = \alpha_{r_{0}} \leqslant \alpha_{r_{i}}$. Overall we have 
\begin{itemize}
\item[(i)] $\alpha_{r_{0}} \leqslant \alpha_{r_{i}} \ \overset{\eqref{eqao1}}{=} \ \alpha_{r_{0}}+i \ \Rightarrow \ 0 \leqslant i$
\item[(ii)] $\omega_{r_{0}} \geqslant \omega_{r_{j}} \ \overset{\eqref{eqao1}}{=} \ \omega_{r_{0}}+j \ \Rightarrow \ 0 \geqslant j$
\end{itemize}
From (i) and (ii) we get $0 \leqslant i \leqslant j \leqslant 0$. Thus, we have $i=j=0$. Finally, $(r_{0}^{-1},c)$ is an element of $\langle \! \langle (r_{0},c) \rangle \! \rangle_{\ker(\zeta)}$. By symmetry, $(r_{0},c)$ is an element of $\langle \! \langle (r_{0}^{-1},c) \rangle \! \rangle_{\ker(\zeta)}$ and altogether we have
\begin{eqnarray} \label{eqequalnc}
\langle \! \langle (r_{0},c) \rangle \! \rangle_{\ker(\zeta)} \ \ = \ \ \langle \! \langle (r_{0}^{-1},c) \rangle \! \rangle_{\ker(\zeta)}.
\end{eqnarray}

Next, we show that $r_{0}$ written as a word in the generators of the free product \eqref{eqkerzeta} uses at least two different factors $B_{i}$. To the contrary, we assume that $r_{0}$ and therefore also $r_{0}^{-1}$ uses only one factor $B_{i}$. Note that this factor has to be $B_{0}$ since the normal form of $r$ with respect to $A' \ast B'$ has even length of at least $2$. By assumption we therefore have $r_{0}=r \in \widetilde{A} \ast B_{0} \cong \widetilde{A} \ast B'$. This contradicts the definition of $A'$ since $\widetilde{A}$ is a proper subgroup of $A'$ because of $a \notin \widetilde{A}$. 

We write
\begin{eqnarray*}
\ker(\zeta) \ = \ (K \ast B_{0}) \times C \ \ \ \text{with} \ \ \ K \ := \ \widetilde{A} \ast \underset{k \neq 0}{\bigast} B_{k}.
\end{eqnarray*}
Since $r_{0}$ uses aside from $B_{0}$ at least one further factor $B_{i}$ with $i \in \mathbb{Z} \backslash \{0\}$, the normal form of $r_{0} \in K \ast B_{0}$ possesses an even length of at least 2. Next, we argue that this length must indeed be smaller than the length of the original normal form of $r \in A' \ast B'$: While rewriting the normal form of $r_{0} \in K \ast B_{0}$ into the normal form of $r \in A' \ast B'$ we get a $B'$-piece for every $B_{0}$-piece because of $B' = B_{0}$. Every piece of $K$ corresponds to a part of the normal form $r$ which is an $A'$-piece or which begins and ends with an $A'$-piece since elsewise the first or last piece of that part from $K=\widetilde{A} \ast \bigast_{k \neq 0} B_{k}$ would be a $B_{0}$-piece. Therefore combining these parts with $B'$-pieces cannot create any cancellation. In particular $r_{0}$ cannot begin and end with a $B_{0}$ piece simultaneously. So for every piece of the normal form of $r_{0} \in K \ast B_{0}$ we get at least one piece of the normal form of $r \in A' \ast B'$. By assumption, at least one $K$-piece of $r_{0}$ is a $B_{i}$-piece. For this $K$-piece we get at least three pieces in the normal form of $r \in A' \ast B'$. Hence the normal form of $r_{0} \in K \ast B_{0}$ is shorter than the normal form of $r \in A' \ast B'$.

In order to be able to apply the induction hypothesis we want to show that $r_{0}$ can be written as an element of a free product of two factors which are not only locally indicable but also indicable. To achieve that, we consider the normal form of $r_{0} \in K \ast B_{0}$ and denote the subgroup generated by the $K$-pieces resp. $B_{0}$-pieces of $r_{0}$ by $K'$ resp. $B_{0}'$. Note that both subgroups are finitely generated, locally indicable and therefore also indicable. We have already seen that the normal form of $r_{0} \in K \ast B_{0}$ possesses an even length greater or equal $2$. Therefore $(r_{0},c) \in (K' \ast B_{0}') \times C$ cannot be conjugate to an element of $K' \times C$ or $B_{0}' \times C$. By applying Lemma~\ref{lemmachi} for $\widetilde{C} := C \cap \langle \! \langle r_{0} \rangle \! \rangle_{(K' \ast B_{0}') \times C}$ we get the embeddings of $K' \times \widetilde{C}$ and $B_{0}' \times \widetilde{C}$ into $\big( ( K' \ast B_{0}') \times C \big) / \widetilde{R}$, where $\widetilde{R} := \langle \! \langle (r_{0},c) \rangle \! \rangle_{(K' \ast B_{0}') \times C}$. This justifies analogously to \eqref{eqIsomAStrich}:
\begin{eqnarray*}
((K \ast B_{0}) \times C) / \langle \! \langle (r_{0},c) \rangle \! \rangle \ \ \cong \ \ (K \times \widetilde{C}) \ \underset{K' \times \widetilde{C}}{\ast} \ ((K' \ast B_{0}') \times C) / \widetilde{R} \ \underset{B_{0}' \times \widetilde{C}}{\ast} \ (B_{0} \times \widetilde{C})
\end{eqnarray*} 
As a factor of the amalgamated product the group $\big( ( K' \ast B_{0}') \times C \big) / \widetilde{R}$ embeds into $\big( ( K \ast B_{0}) \times C \big) / \langle \! \langle (r_{0},c) \rangle \! \rangle$. Since $(r_{0}^{-1},c)$ is trivial in $\big( ( K \ast B_{0}) \times C \big) / \langle \! \langle (r_{0},c) \rangle \! \rangle$ and because of \eqref{eqequalnc} the element $(r_{0}^{-1},c)$ has to be trivial in $\big( ( K' \ast B_{0}') \times C \big) / \langle \! \langle (r_{0},c) \rangle \! \rangle$. By symmetry $(r_{0},c)$ is also trivial in $\big( ( K' \ast B_{0}') \times C \big) / \langle \! \langle (r_{0}^{-1},c) \rangle \! \rangle$. Altogether we get the equality of the normal closures of $(r_{0},c)$ and $(r_{0}^{-1},c)$ in $(K' \ast B_{0}') \times C$. Applying the induction hypothesis, we deduce that $(r_{0},c)$ and $(r_{0}^{-1},c)$ are conjugate in $(K' \ast B_{0}') \times C$. Finally, $(r,c)$ is conjugate to $(r^{-1},c)^{\pm 1}$ in $(A \ast B) \times C$ since $K' \ast B_{0}'$ is a subgroup of $A \ast B$ by construction. This finishes the proof of implication \eqref{eqaim2} in case 1.

\noindent\underline{Case 2.} Let $\varphi(r_{A}) \neq 0$ and $\psi(r_{B}) \neq 0$.\\
We consider the homomorphism $\xi \colon (A' \ast B') \times C \rightarrow \mathbb{Z}$ which is given by
\begin{eqnarray*}
\xi(w,d) \ \ = \ \ \psi(r_{B}) \ \cdot \ \varphi(w_{A}) \ - \ \varphi(r_{A}) \ \cdot \ \psi(w_{B}).
\end{eqnarray*}
Let $r=u(1)v(1)u(2)v(2)\dots u(\frac{n}{2}) v(\frac{n}{2})$ be the normal form of $r$ with respect to $A' \ast B'$, where $u(i) \in A'$ and $v(i) \in B'$ ($1 \leqslant i \leqslant \frac{n}{2}$). Because of $\varphi(r_{A}) \neq 0$ there is at least one $A'$-piece $a:=u(i)$ of $r$ with $\xi(u(i)) \neq 0$. By cyclically permuting $r$ if necessary, we can assume w.\,l.\,o.\,g. $a = u(1)$. We choose an $b \in B'$ with $\psi(b)=-1$ and define $\mu := \xi(a) \neq 0$ along with $\nu := \xi(b)$ ($=\varphi(r_{A}) \neq 0$).

Since $A'$, $B'$ are torsion-free we are able to define:
\begin{eqnarray} \label{eqrefGSchlangeA''}
\widetilde{G} &:=& \big(\langle \widetilde{a} \mid \rangle \underset{\widetilde{a}^{\mu}=a}{\ast} (A' \ast B') \underset{b=\widetilde{b}^{\nu}}{\ast}  \langle \widetilde{b} \mid \rangle \big) \times C \ = \ \big( A'' \ast  B'' \big) \times C \\
\text{with} \ \ A'' &:=& \langle A',\widetilde{a} \, \mid \, \widetilde{a}^{\mu}=a \rangle \ \ \ \text{and} \ \ \ B'' \ := \ \langle B',\widetilde{b} \, \mid \, \widetilde{b}^{\nu}=b \rangle \nonumber
\end{eqnarray}
By Theorem~\ref{Howlokind}, $A''$ and $B''$ are locally indicable groups. We choose an additional generator $x$ and expand $\xi$ by $\xi(\widetilde{a})= \xi(\widetilde{b}) = \xi(x)=1$ to an epimorphism from $\widetilde{G} \ast \langle x \rangle$ onto $\mathbb{Z}$. Finally, we define $\widetilde{a}_{i} := x^{-i} \widetilde{a} x^{i-1}$ and $\widetilde{b}_{i} := x^{-i} \widetilde{b} x^{i-1}$ ($i \in \mathbb{Z}$). Then we have by Lemma~\ref{lemkernphi}\begin{eqnarray*}
\ker(\xi) \ = \  \big( \big( A'' \cap \ker(\xi) \big) \ast \big( B'' \cap \ker(\xi) \big) \ast \langle \widetilde{a}_{i}, \widetilde{b}_{i} \, (i \in \mathbb{Z}) \mid \, \rangle \big) \times C.
\end{eqnarray*}
For $r_{i} := x^{-i} r x^{i}$ we get by Lemma~\ref{lemnormcleq} the equality of the normal closures of the sets $\{(r_{i},c) \mid i \in \mathbb{Z} \}$ and $\{(r_{i}^{-1},c) \mid i \in \mathbb{Z} \}$ in $\ker(\xi)$. Let $\alpha_{r_{i}}$ resp. $\omega_{r_{i}}$ be the smallest resp. greatest index $k$ such that $r_{i} \in \ker(\xi)$ uses a generator $b_{k}$. We have:
\begin{eqnarray*}
\alpha_{r_{j}} \ = \ \alpha_{r_{j}^{-1}}, \ \ \ \omega_{r_{j}} \ = \ \omega_{r_{j}^{-1}} \ \ \ \ \text{and} \ \ \ \ \alpha_{r_{j+k}}=\alpha_{r_{j}}+k, \ \ \ \omega_{r_{j+k}}=\omega_{r_{j}}+k \ \ (\forall j, k \in \mathbb{Z})
\end{eqnarray*}
We choose indices $i$, $j \in \mathbb{Z}$ with $i \leqslant j$ such that $(r_{0}^{-1},c)$ is an element of
\begin{eqnarray*}
\langle \! \langle (r_{i},c),(r_{i+1},c),\dots,(r_{j},c) \rangle \! \rangle_{\ker(\xi)}
\end{eqnarray*}
and $j-i$ is minimal with that property. Our aim is to show $i=j=0$. To the contrary, let $i<j$. We want to apply Lemma~\ref{lemamalgPr} on $B_{i} := \langle \widetilde{b}_{i} \rangle$ and $\widetilde{A} := \big( A'' \cap \ker(\xi) \big) \ast \big(B'' \cap \ker(\xi) \big) \ast \langle \widetilde{a}_{i} \ (i \in \mathbb{Z}) \mid \rangle$. Analogously to case 1 we deduce from Lemma~\ref{lemamalgPr} and Lemma~\ref{lemamalgPrrefl} the chain of equalities $i=j=0$ and thus
\begin{eqnarray} \label{eqequalnc2}
\langle \! \langle (r_{0},c) \rangle \! \rangle_{\ker(\xi)} \ \ = \ \ \langle \! \langle (r_{0}^{-1},c) \rangle \! \rangle_{\ker(\xi)}.
\end{eqnarray}
We now write $\ker(\xi)$ in the form
\begin{eqnarray*}
\ker(\xi) &=& (K \ast L) \times C \ \ \text{with}\\
K &=&  A'' \cap \ker(\xi) \ \ \text{and} \ \ L  \, = \,  \big( B'' \cap \ker(\xi) \big) \ast \langle \widetilde{a}_{i}, \widetilde{b}_{i} \ (i \in \mathbb{Z}) \mid \, \rangle.
\end{eqnarray*}
To show that the normal form of $r_{0}$ with respect to $K \ast L$ is shorter than the normal form $$r=av(1)u(2)v(2) \dots u(\frac{n}{2})v(\frac{n}{2})$$ with respect to $A' \ast B'$, we rewrite the latter normal form into the first: We write
\begin{eqnarray*}
r &=&  av(1)u(2)\dots v(\frac{n}{2}) \\
&=& \widetilde{a}^{\mu} x^{-\mu} \cdot x^{\mu}\widetilde{b}^{-\mu} \cdot \widetilde{v}(1) \cdot \widetilde{b}^{\delta_{1}}x^{-\delta_{1}} \cdot x^{\delta_{1}}\widetilde{a}^{-\delta_{1}} \cdot \widetilde{u}(2)  \dots \widetilde{a}^{\delta_{2n-2}}x^{-\delta_{2n-2}} \cdot  x^{\delta_{2n-2}}\widetilde{b}^{-\delta_{2n-2}} \cdot \widetilde{v}(\frac{n}{2})
\end{eqnarray*}
for elements $\widetilde{u}(i) \in A'' \cap \ker(\xi)$, $\widetilde{v}(i) \in B'' \cap \ker(\xi)$ and $\delta_{i} \in \mathbb{Z}$. Just as in the proof of Lemma~\ref{lemkernphi} we have $\widetilde{a}^{k} x^{-k} = \Pi_{j=0}^{k-1} \widetilde{a}_{-j}$ and $\widetilde{b}^{k}x^{-k}= \Pi_{j=0}^{k-1} \widetilde{b}_{-j}$. Finally, we get the presentation
\begin{eqnarray*}
r &=& \underbrace{\widetilde{a}_{0}\widetilde{a}_{-1} \dots \widetilde{a}_{-\mu+1} \cdot g(1) \cdot \widetilde{v}(1) \cdot g(2) \cdot f(2)}_{\in L} \cdot \underbrace{\widetilde{u}(2)}_{\in K}  \dots \underbrace{f(n-2) \cdot g(n-1) \cdot \widetilde{v}(\frac{n}{2})}_{\in L},
\end{eqnarray*}
where $f(i)$ are elements in the generators $\widetilde{a}_{j}$ ($j \in \mathbb{Z}$) and $g(i)$ are elements in the generators $\widetilde{b}_{j}$ ($j \in \mathbb{Z}$). Note that apart from the first $A'$-piece $a=u(1)$ of $r \in A' \ast B'$ all pieces of $r \in A' \ast B'$ pass over to respectively at most one piece from $r_{0} \in K \ast L$. Since $a$ and $v(1)$ pass over to a common $L$-piece, the length of the normal form of $r_{0} \in K \ast L$ is shorter than $n$. Before we can apply the induction hypothesis we have to pass over to factors which are not only locally indicable but also indicable. To do that we define $K'$ resp.~$L'$ as the subgroup of $K \ast L$ generated by all $K$- resp.~$L$-pieces of $r_{0} \in K \ast L$. For $K' = \{1\}$ we have $A' \cong \mathbb{Z}$. In this case we go back to the beginning of case 2 and go through the proof with reversed roles of $A'$ and $B'$. If we thereby get to the case $A' \cong B' \cong \mathbb{Z}$ we are in the situation $A' \ast B' = F_{2}$ of the induction hypothesis. We can therefore assume w.\,l.\,o.\,g. that the normal form of $r_{0} \in K' \ast L'$ has an even length of at least $2$. Similar to our considerations for \eqref{eqIsomAStrich} we deduce with the help of Lemma~\ref{lemmachi} for $\widetilde{R} := \langle \! \langle (r_{0},c) \rangle \! \rangle_{(K' \ast L') \times C}$, $N:= C \cap \widetilde{R}$ and $\widetilde{C} := C / N$ the presentation
\begin{eqnarray*}
((K \ast L) \times C) / \langle \! \langle (r_{0},c) \rangle \! \rangle \ \ \cong \ \ (K \times \widetilde{C}) \ \underset{K' \times \widetilde{C}}{\ast} \ ((K' \ast L') \times C) / \widetilde{R} \ \underset{L' \times \widetilde{C}}{\ast} \ (L \times \widetilde{C}).
\end{eqnarray*}
Analogously to case 1 we get the equality of the normal closures of $(r_{0},c)$ and $(r_{0}^{-1},c)$ in $(K' \ast L') \times C$. Thus, by the induction hypothesis, $(r_{0},c)$ is conjugate to $(r_{0}^{-1},c)^{\pm 1}$ in $(K' \ast L') \times C$. Since $(K' \ast L') \times C$ is a subgroup of $(\widetilde{G} \ast \langle x \rangle ) \times C$ we get the conjugation of $(r,c)$ to $(r^{-1},c)^{\pm 1}$ in $(\widetilde{G} \ast \langle x \rangle) \times C$. Recall that $r$ is an element of $\widetilde{G}$. By applying the canonical projection from $(\widetilde{G} \ast \langle x \rangle) \times C$ onto $\widetilde{G} \times C$ we get the conjugation of $(r,c)$ and $(r^{-1},c)$ in $\widetilde{G} \times C$. Further we have $r \in A' \ast B'$. Applying Lemma~\ref{lemassum12} on the amalgamated product from \eqref{eqrefGSchlangeA''} twice we finally get the conjugation of $(r,c)$ to $(r^{-1},c)^{\pm 1}$ in $(A' \ast B') \times C \subseteq (A \ast B) \times C$. 
\end{proof}

Now we are able to deduce the statement of Theorem~\ref{thmedjvetgen}.

\noindent\textbf{Proof of Theorem~\ref{thmedjvetgen}.}\\
Since we already proved Proposition~\ref{propedjgen2} and since free products of locally indicable groups are again locally indicable, we get the statement of Theorem~\ref{thmedjvetgen} for finite sets $\mathcal{J}$. We prove the remaining case, that $\mathcal{J}$ is infinite, by contradiction. So let $(r,c)$, $(s,d) \in G := ( \bigast_{j \in \mathcal{J}} A_{j} ) \times C$ be elements with the same normal closure in $G$ such that $(r,c)$ is neither conjugate to $(s,d)$ nor to $(s,d)^{-1}$ in $G$. There exist finite subsets $\mathcal{J}' \subset \mathcal{J}$ such that $(r,c)$ and $(s,d)$ are elements of the subgroup $G' := (\bigast_{j \in \mathcal{J}'} A_{j}) \times C$ of $G$. We define the canonical projection $\pi \colon G \rightarrow G'$. Because of $\pi((r,c)) = (r,c)$ and $\pi((s,d)) = (s,d)$ the normal closures of $(r,c)$ and $(s,d)$ in $G'$ are equal, so $(r,c)$ is conjugate to $(s,d)$ or $(s,d)^{-1}$ in $G' \subset G$ which is a contradiction.
\qed

\section{Proof of the main theorem} \label{secmain}

In this section we prove Main Theorem~\ref{main} using results from the previous section and similar methods as in \cite{ArtOneRel}, \cite{MA}. For the groups $G$, $F(a,b)$, $C$ and the element $u \in G$ of Main Theorem~\ref{main} we define
\begin{eqnarray*}
G' := G \ast_{u=[a,b]} F(a,b) \ \ \ \text{and} \ \ \ \widetilde{G} := G' \times C.
\end{eqnarray*}

\subsection{Reduction to a new group}

First, we show that $G \times C$ possesses the Magnus property if $\widetilde{G}$ possesses the Magnus property. For this purpose we prove with Lemma~\ref{lemcnmp2} a slightly weaker statement and define the following. 

\begin{definition}
Let $G'$ be a group and $G$ be a subgroup of $G'$. We denote $G$ as a \emph{conjugacy-neutral subgroup} of $G'$ if the conjugation of two arbitrary elements $r$, $s \in G$ in $G'$ always implies the conjugation of $r$ and $s$ in $G$, that is
\begin{eqnarray*}
\forall \ r,s \in G \colon \ \big( ( \exists u \in G' \colon u^{-1} r u = s ) \ \ \Rightarrow \ \ ( \exists v \in G \colon \ v^{-1} r v = s) \big).
\end{eqnarray*}
\end{definition}

For conjugacy-neutral subgroups we prove the following simple statement which will shorten the proof of Lemma~\ref{lemcnmp2}. 

\begin{lemma} \label{lemcnmp}
Let $G'$ be a group and $G \subseteq G'$ a conjugacy-neutral subgroup. If $G'$ possesses the Magnus property then $G$ possesses the Magnus property.
\end{lemma}

\begin{proof}
Let $r$, $s \in G$ be two elements with $\langle \! \langle r \rangle \! \rangle_{G} = \langle \! \langle s \rangle \! \rangle_{G}$. Then we also have $\langle \! \langle r \rangle \! \rangle_{G'} = \langle \! \langle s \rangle \! \rangle_{G'}$. By assumption, $G'$ possesses the Magnus property. Thus, $r \sim_{G'} s^{\pm 1}$. Since $G$ is a conjugacy-neutral subgroup of $G'$, we finally get $r \sim_{G} s^{\pm 1}$.
\end{proof}

\begin{lemma} \label{lemcnmp2}
Let $G',C$ be groups and $G$ a conjugacy-neutral subgroup of $G'$. If $G' \times C$ possesses the Magnus property then $G \times C$ possesses the Magnus property.
\end{lemma}

\begin{proof}
Assume that $G' \times C$ has the Magnus property. Then $G'$ and $C$ also possess the Magnus property. Applying Lemma~\ref{lemcnmp} we get the Magnus property of $G$.

To the contrary, assume that $G \times C$ does not possess the Magnus property. Then, by Lemma~\ref{Lemnus}, there exist elements $(r,c)$, $(r,c^{-1}) \in G \times C$ with the same normal closure in $G \times C$ such that neither $r$ is conjugate to $r^{-1}$ in $G$ nor $c$ is conjugate to $c^{-1}$ in $C$. Since $G$ is a subgroup of $G'$ and since, by assumption, $G' \times C$ possesses the Magnus property we have that $r$ is either conjugate to $r^{-1}$ in $G'$ or $c$ is conjugate to $c^{-1}$ in $C$. We have already excluded the latter. Thus, we get
\begin{eqnarray*}
r \sim r^{-1} \ \text{in} \ G', \ \ \text{but} \ \ r \nsim r^{-1} \ \text{in} \ G.
\end{eqnarray*}
Since $r$ is an element of the conjugation-neutral subgroup $G$ of $G'$, this is a contradiction.
\end{proof}

Note that $G$ from Main Theorem~\ref{main} is a conjugacy-neutral subgroup of $G \underset{u=[a,b]}{\ast} F(a,b)$ due to Lemma~\ref{lemassum12}. Thus, we can apply Lemma~\ref{lemcnmp2} to $G' := G \ast_{u=[a,b]} F(a,b)$. Hence $G \times C$ possesses the Magnus property if $\widetilde{G} = G' \times C$ possesses the Magnus property. This proves one direction of Main Theorem~\ref{main}.

It remains to prove the other direction of the equivalence claimed in Main Theorem~\ref{main}, namely:

\begin{theorem} \label{thmonedirection}
Assume that the assumptions of Main Theorem~\ref{main} are given. If $G \times C$ possesses the Magnus property then $\widetilde{G} = \big( G \ast_{u=[a,b]} F(a,b) \big) \times C$ possesses the Magnus property.
\end{theorem}

First, we deduce Theorem~\ref{thmonedirection} from Proposition~\ref{propmain} which we will prove throughout the remaining part of Section~\ref{secmain}. As a preparation we remark the following.

\begin{remark} \label{remexpsum}
Consider the group presentation $G= \langle \widetilde{E} \mid \widetilde{R} \rangle$ with generating set $\widetilde{E}$ and relation-set $\widetilde{R}$. Let $x$ be a generator of $\widetilde{E}$. If all relations of $\widetilde{R}$ considered as words in the free group with basis $\widetilde{E}$ possess an exponential sum of $0$ with respect to $x$ then the exponential sum of an arbitrary element $r \in G$ with respect to $x$ is well-defined.
\end{remark}

\begin{proposition} \label{propmain}
Let $C$ and $G$ be groups such that one of the following two conditions hold.
\begin{itemize}
\item[$(i)$] Let $C=\{1\}$ and $G$ be locally indicable.
\item[$(ii)$] Let $C \neq \{1\}$ and $G$ be locally indicable as well as indicable.
\end{itemize}
We assume further that $G \times C$ possesses the Magnus property and consider the group
\begin{eqnarray*}
\widetilde{H} \ \ = \ \ \big( G \underset{u=[x^{k},b]}{\ast} F(x,b) \big) \times C
\end{eqnarray*}
for a non-trivial element $u \in G$ and $k \in \mathbb{Z} \backslash \{0\}$. Let $(r,c)$, $(s,d)$ be elements of $\widetilde{H}$ with the following properties:
\begin{itemize}
\item The elements $(r,c)$ and $(s,d)$ possess the same normal closure in $\widetilde{H}$.
\item The element $r \in G \underset{u = [x^{k},b]}{\ast} F(x,b)$ is non-trivial and possesses an exponential sum of $0$ with respect to $x$.
\item In the case $k \neq 1$, the exponential sum of $r \in G \underset{u=[x^{k},b]}{\ast} F(x,b)$ with respect to $b$ is not $0$. (Note that this exponential sum is well-defined by Remark~\ref{remexpsum}.)
\end{itemize}
Then $(r,c)$ is conjugate to $(s,d)$ or $(s,d)^{-1}$ in $\widetilde{H}$. 
\end{proposition}

\noindent\textbf{Proof of Theorem~\ref{thmonedirection} assuming Proposition~\ref{propmain}}\\
This proof will be analogous to the deduction of the main theorem from Proposition~2.1 in \cite{ArtOneRel}. Let $(r,c)$ and $(s,d)$ be elements from $\widetilde{G}$ with the same normal closure in $\widetilde{G}$. We want to show that $(r,c)$ is conjugate to $(s,d)^{\pm 1}$ in $\widetilde{G}$. For that purpose we first consider the case that $r$ and therefore also $s$ is trivial. In this case, $c$ and $d$ have the same normal closure in $C$. Thus, $c$ is conjugate to $d$ in $C$. Hence, $(1,c)$ is conjugate to $(1,d)^{\pm 1}$ in $\widetilde{G}$ and we can assume in the following that $r$ is non-trivial. We denote the exponential sum of an element $w$ with respect to a generator $z$ by $w_{z}$. 

\noindent\underline{Case 1.} Let $r_{b}=0$.\\
We use the presentation
\begin{eqnarray*}
\widetilde{G} \ \ = \ \ \big( G \underset{u^{-1}=[a,b]}{\ast} F(a,b) \big) \times C.
\end{eqnarray*}
By renaming the generators $a$, $b$, $u$ to $b$, $x$, $u^{-1}$ we get a presentation of the form $\widetilde{H}$ from Proposition~\ref{propmain} along with the demanded conditions. Therefore Theorem~\ref{thmonedirection} follows directly from Proposition~\ref{propmain}.

\noindent\underline{Case 2.} Let $r_{b} \neq 0$.\\
In this case there exists the canonical embedding
\begin{eqnarray} \label{eqG'}
G' \ \ := \ \ G \underset{u=[a,b]}{\ast} F(a,b) \ \ \hookrightarrow \ \ H' := G' \underset{a=x^{r_{b}}}{\ast} F(x), 
\end{eqnarray}
which induces the canonical embedding 
\begin{eqnarray*}
\widetilde{G} \ \ \hookrightarrow \ \ \widetilde{H}' \ := \ H' \times C \ \ = \ \ \big(G \underset{u=[x^{r_{b}},b]}{\ast} F(x,b) \big) \times C.
\end{eqnarray*}
Because the normal closures of $(r,c)$ and $(s,d)$ are equal in $\widetilde{G}$ they are also equal in $\widetilde{H}'$. Let $\bar{b} = x^{-r_{a}}b$. Then we get the presentation
\begin{eqnarray} \label{eqH'}
H' \ \ = \ \ G \underset{u=[x^{r_{b}}, \bar{b}]}{\ast} F(x,\bar{b}).
\end{eqnarray}
We consider the element $r \in G'$ as an element of $H'$, denote the exponential sums of $r$ with respect to $a$, $b$ in $G'$ by $r_{a}$, $r_{b}$ and the exponential sums of $r$ with respect to $x$, $\bar{b}$ in $H'$ by $r_{x}$, $r_{\bar{b}}$. It is
\begin{eqnarray*}
r_{x} \ = \ r_{a} r_{b} - r_{b} r_{a} \ = \ 0 \ \ \ \text{and} \ \ \ r_{\bar{b}} \ = \ r_{b} \ \neq \ 0.
\end{eqnarray*}
Thus, all conditions of Proposition~\ref{propmain} are given and $(r,c)$ is conjugate to $(s,d)^{\pm 1}$ in $\widetilde{H}'$. By inverting $(s,d)$ if necessary, we can assume w.\,l.\,o.\,g. that $(r,c)$ is conjugate to $(s,d)$ in $\widetilde{H}'$. Thus we have $r \sim_{H'} s$ and $c \sim_{C} d$. Because of \eqref{eqG'} we have $r \sim_{G'} s$ by Lemma~\ref{lemassum12} with assumption (i) and therefore $(r,c) \sim_{\widetilde{G}} (s,d)$.
\qed

As a preparation for the next subsection we introduce the following notations.

\begin{notation} \label{notK}
We define
\begin{eqnarray*}
H \ \ := \ \ G \underset{u=[x^{k},b]}{\ast} F(x,b) \ \ \text{and} \ \ \widetilde{H} \ \ := \ \ H \times C \ \ \text{(see } \widetilde{H} \ \text{from Proposition~\ref{propmain})}.
\end{eqnarray*}
Recall that we denote the exponential sum of $h$ with respect to $x$ by $h_{x}$. We introduce the homomorphisms
\begin{eqnarray*}
&& \varphi \colon \ H \ \rightarrow \ \mathbb{Z}, \ \ h \mapsto h_{x},\\
&& \widetilde{\varphi} \colon \ \widetilde{H} \ \rightarrow \ \mathbb{Z}, \ \ (h,c) \mapsto h_{x} \ \ (h \in H, \ c \in C).
\end{eqnarray*}
Further, we define
\begin{eqnarray*}
M \ := \ \ker(\varphi) \ \ \ \ \text{and} \ \ \ \ K \ := \ \ker(\widetilde{\varphi}).
\end{eqnarray*}
Then we have $K=M \times C$.
\end{notation}

\begin{remark}
By the assumptions of Proposition~\ref{propmain}, $r$ has exponential sum $0$ with respect to $x$. Thus, we have $r \in M$ and $(r,c) \in K$. We will show that $(r,c)$ is conjugate to $(s,d)^{\pm 1}$ in the subgroup $K$ of $\widetilde{H}$. The assumptions of Proposition~\ref{propmain} allow us to consider the kernel $K$ instead of $\widetilde{H}$ in our proof.
\end{remark}

\subsection[Structure of $K$ and  $b$-left/-right generating sets]{Structure of {\boldmath $K$} and {\boldmath $b$}-left/-right generating sets} \label{subsecKleftright}

We use the notation introduced in Notation~\ref{notK}. Applying the rewriting process of Reidemeister--Schreier we will find a presentation of the kernel $K= M \times C$ of $\widetilde{\varphi}$.

The group $H$ is generated by $G \cup \{x,b\}$. Let $H^{*}$ be the free group with basis $G \cup \{x,b\}$ and let $\xi \colon H^{*} \rightarrow H$ be the canonical homomorphism. We define $M^{*} := \xi^{-1}(M)$. One can verify easily that $\mathcal{T} = \{x^{i} \mid i \in \mathbb{Z} \}$ is a Schreier transversal for $M^{*} \subset H^{*}$. Applying the rewriting process of Reidemeister--Schreier we get the generators $x^{-i}gx^{i}$ and $x^{-i}bx^{i}$ ($g \in G$, $i \in \mathbb{Z}$) of $M$.

\begin{notation} \label{notgi}
For an element $v \in H$ with $v_{x}=0$ and $i \in \mathbb{Z}$ we define $v_{i}$ as the element $x^{-i} v x^{i}$ of $M$. So we get the generators
\begin{eqnarray} \label{eqgibi}
g_{i} \ = \ x^{-i} g x^{i} \ \ \ \text{and} \ \ \ b_{i} \ = \ x^{-i} b x^{i} \ \ \ (g \in G, \ i \in \mathbb{Z})
\end{eqnarray}
of $M$ and define the subgroups $G_{i} := x^{-i} G x^{i}$ ($i \in \mathbb{Z}$) of $M$.
\end{notation}

Applying the rewriting process of Reidemeister--Schreier we get for every relation $v$ of $H$ and every $t$ in the Schreier transversal $\mathcal{T}$ a relation of $M$ which corresponds to the element $t^{-1}vt$ written in the generators of \eqref{eqgibi}. For the relation $[x^{k},b]u^{-1}=1$ of $H$ we get the relations
\begin{eqnarray*}
x^{-i}([x^{k},b]u^{-1})x^{i}=1 \  \Leftrightarrow \  x^{-k-i}b^{-1}x^{k+i} \cdot x^{-i}bx^{i} \cdot x^{-i}u^{-1}x^{i} =1 \ \Leftrightarrow  \ b_{k+i}=b_{i}u_{i}^{-1}  \ \ \ (i \in \mathbb{Z})
\end{eqnarray*}
of $M$. Every relation $w$ of $G$ gives us countable many relations $w_{i}$ ($i \in \mathbb{Z}$). For every $\ell \in \mathbb{Z}$ we have
\begin{eqnarray} \label{eqM}
M &=& N_{\ell} \ast N_{\ell +1} \ast \cdots \ast N_{\ell +k -1} \ \ \ \text{with} \\
N_{i} &=& \cdots \underset{Z_{i-k}}{\ast} ( G_{i-k} \ast \langle b_{i-k} \mid \rangle ) \underset{Z_{i}}{\ast} ( G_{i} \ast \langle b_{i} \mid \rangle ) \underset{Z_{i+k}}{\ast} ( G_{i+k} \ast \langle b_{i+k} \mid \rangle ) \underset{Z_{i+2k}}{\ast} \cdots \ \ \ (i \in \mathbb{Z}), \nonumber
\end{eqnarray}
where $Z_{i+k}$ is the cyclical subgroup generated by $b_{i}u_{i}^{-1}$ in $G_{i} \ast \langle b_{i} \mid \rangle$ and by $b_{i+k}$ in $G_{i+k} \ast \langle b_{i+k} \mid \rangle$. 

\textit{Note:} In the following, $G_{ij}$ denotes the group $G_{\lambda}$, where $\lambda$ is the product of $i$ and $j$. This group is different from the group $G_{i,j}$ introduced in Notation~\ref{notMst}.

The next proposition provides a more concise presentation of $M$ as a free product.

\begin{proposition} \label{propEi}
For every $i \in \mathbb{Z}$
\begin{eqnarray*}
\mathcal{E}(i) \ \ := \ \ \{b_{i},b_{i+1},\dots, b_{i+k-1}\} \ \cup \ \underset{j \in \mathbb{Z}}{\bigcup} G_{j}
\end{eqnarray*}
is a generating set of $M$. Further we have for every $i \in \mathbb{Z}$ a presentation
\begin{eqnarray*}
M \ \ = \ \ \langle b_{i},b_{i+1},\dots, b_{i+k-1} \mid \rangle \ \ast \ \underset{j \in \mathbb{Z}}{\bigast} G_{j}. 
\end{eqnarray*}
\end{proposition}

\begin{proof}
Let
\begin{eqnarray*}
S_{i} \ \ := \ \ \{b_{i}\} \ \cup \ \underset{j \in \mathbb{Z}}{\bigcup} G_{jk+i} \ \ \ (i \in \mathbb{Z}).
\end{eqnarray*}
First, we show that $S_{0}$ is a generating set of $N_{0}$. For $p \in \mathbb{N}_{0}$ we define the subgroups
\begin{eqnarray} \label{eqM0p}
N_{0,p}&:=&(G_{-pk} \ast \langle b_{-pk} \mid \rangle ) \underset{Z_{-(p-1)k}}{\ast} \cdots  \underset{Z_{0}}{\ast}  (G_{0} \ast \langle b_{0} \mid \rangle )  \underset{Z_{k}}{\ast}  \cdots \underset{Z_{pk}}{\ast} (G_{pk} \ast \langle b_{pk} \mid \rangle ) \nonumber \\
&=& \big\langle b_{-pk},b_{-(p-1)k},\dots,b_{pk},G_{-pk},G_{-(p-1)k},\dots ,G_{pk} \mid  \nonumber \\
& & b_{jk}=b_{(j+1)k}u_{jk} \ (-p \leqslant j \leq -1), \ b_{(j+1)k}=b_{jk}u_{jk}^{-1} \ (0 \leqslant j \leq p-1) \big\rangle
\end{eqnarray}
of $N_{0}$. Then $N_{0}$ is the union of the infinite ascending chain $N_{0,0} \subset N_{0,1} \subset N_{0,2} \subset \dots$. We apply a Tietze transformation of the form ``removing a generator'' on the generator $b_{-pk}$ in \eqref{eqM0p}. By doing so the relation $b_{-pk} = b_{(-p+1)k} u_{-pk}$ is deleted. Note that the generator $b_{-pk}$ does not appear in any other relation. Applying Tietze transformation of the same form step by step to the generators $b_{-(p-1)k},b_{-(p-2)k}, \dots , b_{-k}$ and $b_{pk},b_{(p-1)k}, \dots, b_{k}$ we get the presentation:
\begin{eqnarray*}
N_{0,p} \ \ = \ \ \langle b_{0},G_{-pk},G_{-(p-1)k},\dots , G_{pk} \mid \rangle \ \ = \ \ \langle b_{0} \mid \rangle \ \ast \ \underset{|j| \leqslant p}{\bigast} G_{jk}
\end{eqnarray*}

Thus, $S_{0,p} := \{b_{0}\} \cup \bigcup_{-p \leqslant j \leqslant p} G_{jk}$ is a generating set of $N_{0,p}$. Since $S_{0}$ is the union of the generating sets in the infinite, ascending chain $S_{0,0} \subset S_{0,1} \subset S_{0,2} \subset \dots$, we deduce that $S_{0}$ is a generating set of $N_{0}$. Further $N_{0}$ is the union of the subgroups in the infinite, ascending chain $N_{0,0} \subset N_{0,1} \subset N_{0,2} \subset \dots$ and therefore possesses the presentation 
\begin{eqnarray} \label{eqN0}
N_{0} \ \ = \ \ \langle b_{0} \mid \rangle \ \ast \ \underset{j \in \mathbb{Z}}{\bigast} G_{j,k}.
\end{eqnarray}
The operation of $x$ on $M$ by conjugation induces an automorphism of $M$. Because of $x^{-i}G_{0} x^{i} = G_{i}$ and $x^{-i}b_{0}x^{i} =b_{i}$  for all $i \in \mathbb{Z}$ we have $x^{-i} N_{0} x^{i} = N_{i}$ and $x^{-i} S_{0} x^{i} = S_{i}$. Thus, $S_{i}$ is a generating set of $N_{i}$ for all $i \in \mathbb{Z}$. Considering the free product from \eqref{eqM} we get a generating set of $M$ for every $i \in \mathbb{Z}$ by uniting the generating sets $S_{j}$ of $N_{j}$ for $j \in \{i,i+1,\dots, i +k-1\}$. This union corresponds to the set $\mathcal{E}(i)$. Finally we get the desired presentation of $M$ because of \eqref{eqM} and \eqref{eqN0}.
\end{proof}

The next lemma plays an important role in the conclusion of the proof of Proposition~\ref{propmain}.

\begin{lemma} \label{lemKMP}
The kernel $K$ of $\widetilde{\varphi} \colon \widetilde{H} \rightarrow \mathbb{Z}$ (see Notation~\ref{notK}) possesses the Magnus property.
\end{lemma}

\begin{proof}
Because of $K=M \times C$ and Proposition~\ref{propEi} we may write
\begin{eqnarray*}
K \ \ = \ \ \big( \underbrace{\langle b_{i},b_{i+1},\dots , b_{i+k-1} \mid \rangle}_{= F_{k}} \ \ast \ \underset{j \in \mathbb{Z}}{\bigast} G_{j} \big) \ \times \ C,
\end{eqnarray*}
where $F_{k}$ is the free group of rank $k$. In the case that $C$ is non-trivial we want to apply Theorem~\ref{thmedjvetgen} with the presentation $M$ from Proposition~\ref{propEi}. For $C = \{1\}$ we want to apply Corollary~\ref{corEdj}. Thus, we shall check the following claims:
\begin{itemize}
\item[$(i)$] The group $F_{k}$ is indicable and locally indicable.
\item[$(ii)$] The groups $G_{j}$ ($j \in \mathbb{Z}$) are locally indicable and, if $C$ is non-trivial, also indicable.
\item[$(iii)$] The groups $F_{k} \times C$ and $G_{j} \times C$ ($j \in \mathbb{Z}$) possess the Magnus property.
\end{itemize}

\emph{To $(i)$.} Every non-trivial free group is indicable and subgroups of free groups are free.

\emph{To $(ii)$.} The subgroup $G_{0}$ of $K$ is by construction isomorphic to the subgroup $G$ of $\widetilde{H}$. By the assumptions of Proposition~\ref{propmain}, $G_{0}$ is locally indicable and, if $C$ is non-trivial, also indicable. The subgroup $G_{j}$ is the image of $G_{0}$ under the automorphism of $K$ which is induced by the conjugation with $x^{j}$. Thus, the desired properties translate from $G_{0}$ to $G_{j}$ for all $j \in \mathbb{Z}$. 

\emph{To $(iii)$.} As noted in the proof of $(ii)$, $G_{j}$ is isomorphic to $G$ for all $j \in \mathbb{Z}$. Therefore $G_{j} \times C$ possesses the Magnus property by the assumptions of Proposition~\ref{propmain} for all $j \in \mathbb{Z}$. For $C=\{1\}$ the free group $F_{k} \times C \cong F_{k}$ possesses the Magnus property. We did not use that $G$ is indicable here. If $C$ is non-trivial, $G_{j}$ is indicable by $(ii)$ and we deduce with the help of Lemma~\ref{lemindC} that $\mathbb{Z} \times C$ possesses the Magnus property. Finally it follows with Theorem~\ref{thmdirect} that $L \times C$ possesses the Magnus property for every limit group $L$. In particular, $F_{k} \times C$ possesses the Magnus property.
\end{proof}

At this point our proof of Proposition~\ref{propmain} that $(r,c)$ is conjugate to $(s,d)^{\pm 1}$ is in no way finished. Even though we found the subgroup $K$ of $\widetilde{H}$ that possesses the Magnus property and contains the elements $(r_{0},c)$ and $(s_{0},d)$, the normal closure of $(r,c)$ in $\widetilde{H}$ does in general not correspond to the normal closure of $(r_{0},c)$ in $K$. Therefore it is unclear if the normal closures of $(r_{0},c)$ and $(s_{0},d)$ in $K$ are equal. However, we know by Lemma~\ref{lemnormcleq} that the normal closures of $\{(r_{i},c) \mid i \in \mathbb{Z} \}$ and $\{ (s_{i},d) \mid i \in \mathbb{Z} \}$ in $K$ are equal. 

Before we introduce some tool which will later allow us to pass on to the normal closure of $(r_{0},c)$ in $K$ we need some further notation.

\begin{notation} \label{notMst}
For $s$, $t \in \mathbb{Z}$ we define the following subgroups of $M$:
\begin{eqnarray*}
M_{s,t} &:=& \langle G_{j},b_{j} \mid s \leqslant j \leqslant t \rangle, \ \ M_{s,\infty} := \langle G_{j},b_{j} \mid j \geqslant s \rangle, \ \  M_{-\infty,t} := \langle G_{j},b_{j} \mid j \leqslant t \rangle, \\
G_{s,t} &:=& \langle G_{j} \mid s \leqslant j \leqslant t  \rangle, \ \ G_{s,\infty} := \langle G_{j} \mid j \geqslant s \rangle, \ \ G_{-\infty,t} := \langle G_{j} \mid j \leqslant t \rangle \medskip
\end{eqnarray*}
\end{notation}

\begin{remark} \label{remMiinf}
Analogously to the proof of Proposition~\ref{propEi} we get for every $s \in \mathbb{Z}$ and every $i \in \mathbb{Z}$ greater or equal $s$ the presentation
\begin{eqnarray*}
M_{s,\infty} = \langle b_{i},b_{i+1},\dots,b_{i+k-1} \mid \rangle \ast \underset{j \geqslant s}{\bigast} \ G_{j}
\end{eqnarray*}
along with the presentation
\begin{eqnarray*}
M_{-\infty,t} = \langle b_{i-k+1},b_{i-k+2},\dots,b_{i} \mid \rangle \ast \underset{j \leqslant t}{\bigast} \ G_{j}
\end{eqnarray*}
for every $t \in \mathbb{Z}$ and every $i \in \mathbb{Z}$ less or equal $t$.
\end{remark}

\begin{notation}[{\boldmath $b$}\textbf{-right/left-generating set}] \label{NotrlE} For $s$, $t \in \mathbb{Z}$ let
\begin{eqnarray*}
\mathcal{E}_{s,\infty} := \{ b_{s},b_{s+1}, \dots ,b_{s+k-1} \} \ \cup \ \underset{j \geqslant s}{\bigcup} \ G_{j}
\end{eqnarray*}
be the \emph{$b$-left-generating set} of $M_{s,\infty}$ and
\begin{eqnarray*}
\mathcal{E}_{-\infty,t} := \{ b_{t-k+1},b_{t-k+2}, \dots ,b_{t} \} \ \cup \ \underset{j \leqslant t}{\bigcup} \ G_{j}
\end{eqnarray*}
be the \emph{$b$-right-generating set} of $M_{-\infty,t}$.
\end{notation}

\begin{definition} \textbf{(presentations of group elements)}
Let $\langle \mathcal{E} \mid \mathcal{R} \rangle$ be a group presentation, $v$ an element of $U$ and let $\zeta \colon \langle \mathcal{E} \mid \rangle \rightarrow U$ be the canonical homomorphism. Then we call each element $\bar{v} \in \langle \mathcal{E} \mid \rangle$ with $\zeta(\bar{v}) = v$ a \emph{presentation of $v \in U$}. In slight abuse of notation we denote the element of $U$ corresponding to a presentation $w \in \langle \mathcal{E} \mid \rangle$ also by $w$.
\end{definition}

\begin{definition} \textbf{(reduced presentations of elements of {\boldmath $M$})} \label{defredpr} Let $w$ be an element of $M$. We denote a presentation $\bar{w}$ of $w$ with respect to the generating set $\mathcal{E}(i)$ ($i \in \mathbb{Z}$) as a \emph{reduced presentation of $w$ with respect to $\mathcal{E}(i)$} if it is of the form
\begin{eqnarray*}
\bar{w} \ \ = \ \ \overset{n}{\underset{\ell = 1}{\Pi}} \tau(\ell),
\end{eqnarray*}
where the following statements hold:
\begin{itemize}
\item[(i)] We have $n \in \mathbb{N}_{0}$. (For $n=0$ we have $\bar{w}=1$.)
\item[(ii)] Every $\tau(\ell)$ is a non-trivial, freely reduced word in the group $\langle b_{i},b_{i+1},\dots,b_{i+k-1} \mid \rangle$ or a generator from $G_{j} \backslash \{1\}$ for some $j \in \mathbb{Z}$.
\item[(iii)] Consecutive $\tau(\ell)$, $\tau(\ell+1)$ originate from different generating sets $G_{j}$ or from a generating set $G_{j}$ and $\langle b_{i},b_{i+1},\dots,b_{i+k-1} \mid \rangle$. 
\end{itemize}
We call $n$ the \emph{length} of the reduced presentation of $w$ with respect to $\mathcal{E}(i)$ and denote the elements $\tau(\ell)$ ($1 \leqslant \ell \leqslant n$) as \emph{pieces} of the reduced presentation.

Analogously we define for every element $w$ from $M_{s,\infty}$ resp. $M_{-\infty,t}$ reduced presentation of $w$ with respect to $\mathcal{E}_{s,\infty}$ resp. $\mathcal{E}_{-\infty,t}$. 
\end{definition}

Concerning the different possibilities of presenting an element $w \in M$ we note the following.

\begin{remark} \textbf{(uniqueness and reducing process)}
Let $w$ be an arbitrary presentation with respect to a generating set $\mathcal{E}(i)$ of $M$. By consolidating generators of the same generating set $G_{i}$ next to each other and freely reducing the generators $b_{i}$ ($i \in \mathbb{Z}$), $w$ becomes a reduced presentation with respect to $\mathcal{E}(i)$. We refer to such a transition from a presentation to a reduced presentation as \emph{reducing} of a presentation. Note, that reduced presentations are unique because of the presentation of $M$ from Proposition~\ref{propEi} as a free product. The same uniqueness holds for the reduced presentations of elements from $M_{s,\infty}$ with respect to $\mathcal{E}_{s,\infty}$ or from $M_{-\infty,t}$ with respect to $\mathcal{E}_{-\infty,t}$. 
\end{remark}

Finally, we give an algorithm for the reducing process from a presentation to a reduced presentation.

\begin{algorithm} \label{algrewrprred}
Let $w$ be a presentation of an element of $M$ with respect to some generating set $\{b_{i} \mid i \in \mathbb{Z} \} \cup \bigcup_{i \in \mathbb{Z}} G_{i}$. Further let $s$ be the smallest index $j$ such that $w$ uses a generator with index $j$. The following process rewrites a given presentation of $w$ into a reduced presentation with respect to the generating set $\mathcal{E}_{s,\infty}$: First, we replace every generator $b_{j}$ in $w$ by $b_{j-k}u_{j-k}^{-1}$ if $j \geqslant s+k$. After that we reduce. If the new presentation still contains generators $b_{j}$ which are not in $\mathcal{E}_{s,\infty}$, we repeat the process and finally get to a reduced presentation with respect to $\mathcal{E}_{s,\infty}$. The process terminates after finitely many iterations since the greatest index $j$ such that the current presentation uses a generator $b_{j}$ strictly decreases with every iteration.

Let $t$ be the greatest index $j$ such that the initially given presentation of $w$ contains a generator with index $j$. Then we can rewrite $w$ analogously into a reduced presentation with respect to $\mathcal{E}_{-\infty,t}$.
\end{algorithm}

\subsection[Dual structure of $M$]{Dual structure of {\boldmath $M$}} \label{subsecdual}
In this short, technical subsection we point out a special structure of the group $K$ which we call \emph{dual} to the structure of $K$ described in Subsection~\ref{subsecKleftright}. The dual structure will help us to significantly shorten certain parts of our proof (e.\,g.~the deduction of Lemma~\ref{lemcentralrefl} from Lemma~\ref{lemcentral} and the proof of the finiteness of Algorithm~\ref{algalpha} in the dual case).

For every $i \in \mathbb{Z}$ and $g \in G$ we define
\begin{eqnarray*}
b'_{i} \ \ := \ \ b_{-i}u_{-i}^{-1} \ \ \ \ \text{and} \ \ \ \ g'_{i} \ \ := \ \ g_{-i}. 
\end{eqnarray*}

We denote the group $G_{-i}$ presented using the generators $g'_{i}$ instead of $g_{-i}$ by $G'_{i}$. Thus, $G_{i}'$ and $G_{-i}$ are two different names for the same group. We call the generators from $\{b'_{i} \mid i \in \mathbb{Z} \} \cup \bigcup_{i \in \mathbb{Z}} G'_{i}$ \emph{dual generators} of $M$ and we say that a generator $g'_{i}$ ($g \in G$, $i \in \mathbb{Z}$) is \emph{dual} to $g_{-i}$.

Rewriting the relation $b_{i}u_{i}^{-1}=b_{i+k}$ with the help of dual generators we get $b'_{-i} = b_{-i-k}' u_{-i-k}'^{-1}$, where the presentation $u_{-i-k}' \in G'_{-i-k}$ is obtained from the presentation $u_{i+k} \in G_{i+k}$ by inverting and afterwards replacing all generators by dual generators from $G'_{-i-k}$. Applying an index shift $i'=-i-k$ we get the form $b'_{i'+k} = b'_{i'} u_{i'}'^{-1}$ ($i' \in \mathbb{Z}$) of the original relation $b_{i}u_{i}^{-1} = b_{i+k}$ ($i \in \mathbb{Z}$). This justifies the name ``dual''.

Analogously to Notation~\ref{notMst} we define for $s$, $t \in \mathbb{Z}$ the subgroups
\begin{eqnarray*}
M'_{s,t} &:=& \langle G'_{j},b'_{j} \mid s \leqslant j \leqslant t \rangle, \ \ M'_{s,\infty} := \langle G'_{j},b'_{j} \mid j \geqslant s \rangle, \ \  M'_{-\infty,t} := \langle G'_{j},b'_{j} \mid j \leqslant t \rangle, \\
G'_{s,t} &:=& \langle G'_{j} \mid s \leqslant j \leqslant t  \rangle, \ \ G'_{s,\infty} := \langle G'_{j} \mid j \geqslant s \rangle \ \ \text{and} \ \ G'_{-\infty,t} := \langle G'_{j} \mid j \leqslant t \rangle
\end{eqnarray*}
of $M$. Thus, the subgroups $M'_{s,t}$, $M'_{s,\infty}$ and $M'_{-\infty,t}$ correspond to the subgroups $M_{-t,-s}$, $M_{-\infty,-s}$ and $M_{-t,\infty}$. Analogously to Subsection~\ref{subsecdual} we get the presentations
\begin{eqnarray*}
M &=& \langle b'_{i},b'_{i+1},\dots,b'_{i+k-1} \mid \rangle \ \ast \ \underset{j \in \mathbb{Z}}{\bigast} \ G'_{j} \ \ \ \ \  \text{for all} \ i \in \mathbb{Z}, \\
M'_{s,\infty} &=& \langle b'_{i},b'_{i+1},\dots,b'_{i+k-1} \mid \rangle \ \ast \ \underset{j \geqslant s}{\bigast} \ G'_{j} \ \ \ \ \ \text{for all} \ i \in \mathbb{Z} \ \text{with} \ i \geqslant s, \\
M'_{-\infty,t} &=& \langle b'_{i-k+1},b'_{i-k+2},\dots,b'_{i} \mid \rangle \ \ast \ \underset{j \leqslant t}{\bigast} \ G'_{j} \ \ \text{for all} \ i \in \mathbb{Z} \ \text{with} \ i \leqslant t
\end{eqnarray*}
along with the generating sets
\begin{eqnarray*}
\mathcal{E}'(i) \ \ := \ \ \ \ \ \{b'_{i},b'_{i+1},\dots,b'_{i+k-1}\} \ \cup \ \underset{j \in \mathbb{Z}}{\bigcup} \ G'_{j}  \ \ &\text{of}& \ \ M \ \ \ \ \ \ \, \ \text{for all} \ i \in \mathbb{Z}, \\
\{b'_{i},b'_{i+1},\dots,b'_{i+k-1}\} \ \cup \ \underset{j \geqslant s}{\bigcup} \ G'_{j} \ \ &\text{of}& \ \ M'_{s,\infty} \ \ \ \ \text{for all} \ i \in \mathbb{Z} \ \text{with} \ i \geqslant s, \\
\{ b'_{i-k+1},b'_{i-k+2}, \dots ,b'_{i} \} \ \cup \ \underset{j \leqslant t}{\bigcup} \ G'_{j} \ \ &\text{of}& \ \ M'_{-\infty,t} \ \ \ \text{for all} \ i \in \mathbb{Z} \ \text{with} \ i \leqslant t.
\end{eqnarray*}
Analogously to Notation~\ref{NotrlE} we define for $s$, $t \in \mathbb{Z}$ the $b'$-left generating set
\begin{eqnarray*}
\mathcal{E}'_{s,\infty} := \{ b'_{s},b'_{s+1}, \dots ,b'_{s+k-1} \} \ \cup \ \underset{j \geqslant s}{\bigcup} \ G'_{j} \ \ \text{of} \ \ M'_{s,\infty}
\end{eqnarray*}
and the \emph{$b'$-right generating set}
\begin{eqnarray*}
\mathcal{E}'_{-\infty,t} \ := \ \{ b'_{t-k+1},b'_{t-k+2}, \dots ,b'_{t} \} \ \cup \ \underset{j \leqslant t}{\bigcup} \ G'_{j} \ \ \ \ \text{of} \ \ \ \ M'_{-\infty,t}.
\end{eqnarray*}
Finally, we define the \emph{reduced presentation} with respect to dual generating sets in the same way as in Definition~\ref{defredpr}.

\subsection[$\alpha$- and $\omega$-limits]{{\boldmath $\alpha$}- and {\boldmath $\omega$}-limits}

Before defining $\alpha$- and $\omega$-limits we motivate this definition by considering different presentations of $b_{k}b_{k+1}^{-1} \in M$:
\begin{eqnarray*}
b_{0}u_{0}^{-1}u_{1}u_{1-k}b_{1-k}^{-1} \ = \ b_{0}u_{0}^{-1}u_{1}b_{1}^{-1}=b_{k}b_{k+1}^{-1}=b_{2k}u_{k}u_{k+1}^{-1}b_{2k+1}^{-1}\ = \ b_{3k}u_{2k}u_{k}u_{k+1}^{-1}u_{2k+1}^{-1}b_{3k+1}^{-1}
\end{eqnarray*}
Note that the presentations use generators with indices in different intervals. However, there seems to be no presentation of $b_{k}b_{k+1}^{-1}$ using exclusively generators that are greater than $k$. Furthermore there seems to be no presentation using exclusively generators that are less than $1$.

\begin{definition}\textbf{({\boldmath$\alpha$}-/{\boldmath$\omega$}-limits and {\boldmath $\alpha$}-{\boldmath $\omega$}-length)} \label{defalpha}
Let $r$ be a non-trivial element of $M$. Then the $\alpha$-limit $\alpha_{r}$, the $\omega$-limit $\omega_{r}$ and the $\alpha$-$\omega$-length $|r|_{\alpha,\omega}$ of $r$ are defined by
\begin{eqnarray*}
\alpha_{r} &:=& \max\{j \in \mathbb{Z} \mid r \in M_{j,\infty} \}, \\
\omega_{r} &:=& \min\{j \in \mathbb{Z} \mid r \in M_{-\infty,j} \}, \\
|r|_{\alpha,\omega} &:=& \omega_{r}-\alpha_{r}+1.
\end{eqnarray*}
Analogously we define the dual $\alpha$- and $\omega$-limits
\begin{eqnarray*}
\alpha'_{r} \ := \ \max\{j \in \mathbb{Z} \mid r \in M'_{j,\infty} \} \ \ \ \ \text{and} \ \ \ \ \omega'_{r} \ := \ \min\{j \in \mathbb{Z} \mid r \in M'_{-\infty,j} \}
\end{eqnarray*}
along with the dual $\alpha$-$\omega$-length $|r|_{\alpha',\omega'} := \omega'_{r}-\alpha'_{r}+1$.
\end{definition}

We prove step by step that $\alpha$- and $\omega$-limits are well-defined. First, we consider the $\alpha$-limit $\alpha_{r}$ for a non-trivial element $r \in M$. Lemma~\ref{lemalgvalid} shows, that the following algorithm indeed rewrites a presentation $r$ written with the generators from $\{b_{i} \mid i \in \mathbb{Z} \} \cup \bigcup_{i\in \mathbb{Z}} G_{i}$ in finitely many iterations to a presentation $r^{*}$ of $r$ with respect to the generating set $\mathcal{E}_{\alpha_{r},\infty}$ of $M_{\alpha_{r}, \infty}$.

\begin{algorithm} \label{algalpha}
Let $r$ be a non-trivial element of $M$.
\begin{itemize}
\item[(1)] We consider the presentation of $r$ with respect to the generating set $\{b_{i} \mid i \in \mathbb{Z} \} \cup \bigcup_{i \in \mathbb{Z}} G_{i}$. Let $\lambda$ be the smallest index $\mu \in \mathbb{Z}$ such that the presentation of $r$ uses a generator with index $\mu$. In particular, we have $r \in M_{\lambda,\infty}$. Using Algorithm~\ref{algrewrprred} we determine a reduced presentation $r[1]$ of $r$ with respect to the generating set $\mathcal{E}_{\lambda,\infty}$. Let $j$ be the smallest index $v \in \mathbb{Z}$ such that $r[1]$ uses a generator with index $\nu$. Note that $r[1]$ is also a reduced presentation with respect to the generating set $\mathcal{E}_{j,\infty}$.
\item[(2)] Let $r[2]$ be the reduced presentation of $r$ emerging from $r[1]$ by replacing all generators $b_{j}$ with $b_{j+k}u_{j}$ and reducing afterwards. Then $r[2]$ is a presentation of $r$ with respect to the generating set
\begin{align} \label{eqgsplus1}
\{ b_{j+1},b_{j+2}, \dots ,b_{j+k}\} \cup \bigcup_{i \geqslant j} G_{i} .
\end{align}
\begin{itemize}
\item[(2a)] If $r[2]$ contains no generator from $G_{j}$, $r[2]$ is a reduced presentation with respect to the generating set $\mathcal{E}_{j+1,\infty}$. We choose $\ell$ as the smallest index $\mu \in \mathbb{Z}$ such that $r[2]$ uses a generator with index $\mu$. With the input $j:= \ell$ and $r[1] := r[2]$ we begin Step (2) again. Note that $j$ strictly increased.
\item[(2b)] If $r[2]$ contains a generator from $G_{j}$, the algorithm terminates with the output $\alpha_{r} := j$ and $r^{*} := r[1]$. 
\end{itemize}
\end{itemize}
\end{algorithm}

\begin{lemma} \label{lemalgvalid}
The output $\alpha_{r}$ of Algorithm~\ref{algalpha} is the $\alpha$-limit of $r$ introduced in Definition~\ref{defalpha}. The output $r^{*}$ is a reduced presentation of $r$ with respect to the generating set $\mathcal{E}_{\alpha_{r}, \infty}$.
\end{lemma}

\begin{proof}
To avoid a confusion with the $\alpha$-limit $\alpha_{r}$ from Definition~\ref{defalpha} we denote the output $\alpha_{r}$ from Algorithm~\ref{algalpha} for the purpose of this proof with $\widetilde{\alpha}_{r}$. We want to show that the algorithm ends after finitely many iterations with $\widetilde{\alpha}_{r}=\alpha_{r}$.

To prove the finiteness of the algorithm we show that the length of the reduced presentation $r[1]$ strictly decreases with every iteration which does not terminate the algorithm. The presentation $r[2]$ arises from $r[1]$ by replacing every generator $b_{j}$ with $b_{j+k}u_{j}$. If the reduced presentation $r[1]$ contains no piece from $G_{j} \backslash \{1\}$ (cf. Definition~\ref{defredpr}), there cannot be any cancellation of the new pieces $u_{j}^{\pm 1} \in G_{j}$ in $r[2]$ and so the algorithm terminates. If $r[1]$ contains a piece from $G_{j} \backslash \{1\}$ and we arrive at Step~(2a), there must have been complete cancellations of the new pieces $u_{j}^{\pm 1}$ with pieces of $r[1]$ from $G_{j}$. Therefore the reduced presentation $r[2]$ is strictly shorter than $r[1]$. Note that $r[2]$ is the presentation $r[1]$ of the next iteration.

Since Algorithm~\ref{algalpha} rewrites the element $r$ as an element of $M_{\widetilde{\alpha}_{r},\infty}$ we have $\widetilde{a}_{r} \leqslant \alpha_{r}$ by Definition~\ref{defalpha}. It remains to show that $r$ is no element of $M_{\widetilde{\alpha}_{r}+1,\infty}$. Algorithm~\ref{algalpha} can only end in Step~(2b). Therefore the reduced presentation $r[2]$ of $r$ written with the generators  from \eqref{eqgsplus1} (for $j = \widetilde{a}_{r}$) uses at least one generator from $G_{\widetilde{\alpha}_{r}}$. Note that $r[2]$ is the unique reduced presentation of $r$ with respect to $\mathcal{E}(\widetilde{\alpha}_{r} + 1)$. Further the reduced presentation of an arbitrary element $w \in M_{\widetilde{\alpha}_{r}+1}$ with respect to $\mathcal{E}(\widetilde{\alpha}_{r}+1)$ corresponds to the reduced presentation of $w$ with respect to $\mathcal{E}_{\widetilde{\alpha}_{r}+1,\infty}$. Since $r[2]$ contains a generator of $\mathcal{E}(\widetilde{\alpha}_{r}+1) \backslash \mathcal{E}_{\widetilde{\alpha}_{r}+1,\infty}$, we deduce that $r$ cannot be an element of $M_{\widetilde{\alpha}_{r}+1,\infty}$.
\end{proof}

\begin{remark}
By replacing all objects of Algorithm~\ref{algalpha} with the corresponding dual objects we get an algorithm which calculates to a given presentation $r \in M \backslash \{1\}$ the dual $\alpha$-limit $\alpha'$ and a reduced presentation of $r$ with respect to the dual generating set $\mathcal{E}'_{\alpha'_{r}, \infty}$. The proof of the functionality of this algorithm is analogous to Lemma~\ref{lemalgvalid}.
\end{remark}

Thus, the $\alpha$-limits $\alpha_{r}$ and $\alpha'_{r}$ are well-defined for every non-trivial element $r \in M$. Because of point (a) of the following lemma this also secures that the $\omega$-limits $\omega_{r}$ and $\omega_{r}'$ of $r$ are well-defined.

\begin{lemma} \label{lemalphaomega}
Let $r$ be a non-trivial element of $M$.
\begin{itemize}
\item[$(a)$] We have $\omega_{r} = -\alpha'_{r}$ and $\omega'_{r} = - \alpha_{r}$. In particular we have $|r|_{\alpha,\omega} = |r|_{\alpha', \omega'}$.
\item[$(b)$] For all $i,j \in \mathbb{Z}$ we have $\alpha_{r_{i+j}} = \alpha_{r_{i}}+j$ and $\omega_{r_{i+j}} = \omega_{r_{i}}+j$. In particular we have $|r_{i}|_{\alpha,\omega} = |r_{i+j}|_{\alpha,\omega}$. (Recall $r_{i}:=x^{-i}rx^{i}$; see Notation~\ref{notgi}.)
\end{itemize}
\end{lemma}

\begin{proof}
$(a)$. This statement follows immediately from Definition~\ref{defalpha} since we noted in Subsection~\ref{subsecdual} that the subgroup $M_{-\infty,t}$ corresponds to the subgroup $M'_{-t,\infty}$ for all $t \in \mathbb{Z}$. If $i$ is the smallest index such that $r$ is an element of $M_{-\infty,i}$ then $-i$ is the greatest index such that $r$ is an element of $M'_{-i,\infty}$. So we have $\omega_{r}=-\alpha'_{r}$. The equality $\omega'_{r}=-\alpha_{r}$ can be deduced analogously. Finally, note that $|r|_{\alpha,\omega} = \omega_{r}-\alpha_{r}+1=-\alpha'_{r}-(-\omega'_{r})+1=\omega'_{r}-\alpha'_{r}+1=|r|_{\alpha',\omega'}$.\\
$(b)$. It is sufficient to note that a presentation $r_{i+j}$ is obtained from $r_{i}$ by increasing all indices by $j$.
\end{proof}

\subsection{Suitable conjugates}

This section will allow us later to apply Lemma~\ref{lemmachi}. For this purpose we want to find a conjugate $\widetilde{r}$ of $r$ such that every reduced presentation of $\widetilde{r}$ with respect to the generating sets $\mathcal{E}(i)$ ($i \in \mathbb{Z}$) and their dual correspondents $\mathcal{E}'(i)$ ($i \in \mathbb{Z}$) are cyclically reduced in the sense of Definition~\ref{defcycred}.

\begin{notation}
Let $r$ be a non-trivial element of $M$. Then we denote the reduced presentation of $r$ with respect to $\mathcal{E}(i)$ by $r(i)$. We denote the reduced presentation of $r$ with respect to $\mathcal{E}'(i)$ by $r'(i)$. 

Further, we call the elements of the set $\{b_{j} \mid j \in \mathbb{Z}\}$ \emph{$b$-generators} and generators from the set $\{b'_{j} \mid j \in \mathbb{Z} \}$ \emph{$b'$-generators}.

As a short notation we write that a reduced presentation begins with a $b$- resp. $b'$-generator meaning that the first piece of the reduced presentation begins with a $b$- resp. $b'$-generator. Analogously we say that a presentation ends with a $b$- or $b'$-generator.
\end{notation}

\begin{definition} \label{defcycred}
Let $r$ be a non-trivial element of $M$ and $i \in \mathbb{Z}$. We call a reduced presentation $r(i)$ resp. $r'(i)$ \emph{cyclically reduced} if one of the following statements hold:
\begin{itemize}
\item[$(i)$] The reduced presentation $r(i)$ resp. $r'(i)$ has a length smaller or equal 1.
\item[$(ii)$] The first and last piece of the reduced presentation $r(i)$ resp. $r'(i)$ neither originate from the same generating set $G_{j}$ nor consist both of $b$- resp. $b'$-generators.
\item[$(iii)$] The reduced presentation $r(i)$ resp. $r'(i)$ begins or ends with $b$- resp. $b'$-generators which are not inverse to each other.
\end{itemize} 
\end{definition}

\begin{lemma} \label{lemeq4}
Let $r$ be a non-trivial element of $M$. Then the following statements are equivalent:
\begin{itemize}
\item[$(i)$] There exists an $i \in \mathbb{Z}$ such that $r(i)$ begins with a positive power of a $b$-generator.
\item[$(ii)$] For all $i \in \mathbb{Z}$ $r(i)$ begins with a positive power of a $b$-generator.
\item[$(iii)$] There exists an $i \in \mathbb{Z}$ such that $r'(i)$ begins with a positive power of a $b'$-generator.
\item[$(iv)$] For all $i \in \mathbb{Z}$ $r'(i)$ begins with a positive power of a $b'$-generator.
\end{itemize}
\end{lemma}

\begin{proof}
The reduced presentation $r(i+1)$ emerges from the reduced presentation $r(i)$ by replacing every generator $b_{i}$ in $r(i)$ with $b_{i+k}u_{i}$ and reducing afterwards. Conversely we get the reduced presentation $r(i+1)$ from the presentation $r(i)$ by replacing all generators $b_{i+k}$ with $b_{i}u_{i}^{-1}$ and reduce afterwards. The generator $b_{i+k}$ is not in the set of the $b$-generators used in $r(i)$ and $b_{i}$ is not in the set of $b$-generators used in $r(i+1)$. Because of the presentation of $M$ from Proposition~\ref{propEi} as a free product there cannot be any cancellations of $b$-generators. So a presentation $r(i)$ begins with a positive $b$-generator if and only if $r(i+1)$ begins with a positive $b$-generator. This proves the equivalence of $(i)$ and $(ii)$.

Analogously we can prove the equivalence of $(iii)$ and $(iv)$. For the equivalence of $(i)$ and $(iii)$ we show that $r(i)$ begins with a positive power of a $b$-generator if and only if $r'(-i-k+1)$ begins with a positive power of a $b'$-generator. We recall the notation $b'_{j}=b_{-j}u_{-j}^{-1}$ and $g_{j}=g'_{j}$ ($\forall j \in \mathbb{Z}$, $g_{j} \in G_{j}$) from Subsection~\ref{subsecdual}. The reduced presentation of $r(i)$ in dual generators emerges from $r(i)$ by replacing all generators $b_{j}$ with $b'_{-j}u_{j}=b'_{-j}u_{-j}'^{-1}$, replacing all generators $g_{j} \in G_{j}$ with $g'_{-j}$ and reducing afterwards. This shows that the dual presentations begins with a positive power of a $b'$-generator if $r(i)$ begins with a positive power of a $b$-generator. Since $r(i)$ uses $b$-generator with indices of the set $\{i,i+1,\dots,i+k-1\}$ the dual presentation uses $b$-generators exclusively with indices from $\{-i-k+1,-i-k+2,\dots,-i\}$. This is the presentation $r'(-i-k+1)$. Analogously one can show that $r(i)$ begins with a positive power of a $b$-generator if $r'(-i-k+1)$ begins with a positive power of a $b'$-generator.
\end{proof}

\begin{corollary} \label{corsuitconj}
For every non-trivial $r$ in $M$ there exists a conjugate $\widetilde{r}$ of $r$ such that $\widetilde{r}(i)$ and $\widetilde{r}'(i)$ is cyclically reduced for all $i \in \mathbb{Z}$ (in the sense of Definition~\ref{defcycred}).
\end{corollary}

\begin{proof}
We choose a conjugate $\widetilde{r}$ of $r$ such that $\widetilde{r}(0)$ is cyclically reduced.

\noindent\underline{Case 1.} The conjugate $\widetilde{r}(0)$ contains no $b$-generator.\\
In this case all presentations $r(i)$ with $i \in \mathbb{Z}$ are identical and thus cyclically reduced. The presentations $\widetilde{r}'(i)$ are in this case also identical for all $i \in \mathbb{Z}$ and emerge from $\widetilde{r}(0)$ by replacing every generator $g_{j} \in G_{j}$ ($j \in \mathbb{Z}$) with the dual generator $g'_{-j} \in G'_{-j}$ (see Subsection~\ref{subsecdual}). Therefore all presentations $\widetilde{r}'(i)$ ($i \in \mathbb{Z}$) are cyclically reduced as well.

\noindent\underline{Case 2.} The conjugate $\widetilde{r}(0)$ contains a $b$-generator.\\
Note that, by inverting, Lemma~\ref{lemeq4} can be applied on presentations that end with a negative power of a $b$-generator. W.\,l\,o.\,g. we can assume that $\widetilde{r}(0)$ begins with a positive power of a $b$-generator or ends with a negative power of a $b$-generator (but not both for a $b$-generator with the same index). Because of Lemma~\ref{lemeq4} $(i)$,$(ii)$ the same holds for all $\widetilde{r}(i)$ ($i \in \mathbb{Z}$). In particular, $\widetilde{r}(i)$ ($i \in \mathbb{Z}$) is therefore cyclically reduced. By applying Lemma~\ref{lemeq4} $(iv)$ we deduce that all presentations $\widetilde{r}'(i)$ ($i \in \mathbb{Z}$) are cyclically reduced.
\end{proof}

\begin{definition} \textbf{(suitable conjugates)} We call a non-trivial element $\widetilde{r}$ of $M$ \emph{suitable conjugate} if the reduced presentation of $\widetilde{r}$ with respect to the generating sets $\mathcal{E}(i)$ and $\mathcal{E}'(i)$ is cyclically reduced for every $i \in \mathbb{Z}$ in the sense of Definition~\ref{defcycred}.
\end{definition}

\begin{remark} \label{remconjtosconj}
Because of Corollary~\ref{corsuitconj} there exist a suitable conjugate for every non-trivial element $r$ of $M$. Since normal closures are invariant under conjugation, we may in the following consider suitable conjugates $\widetilde{r}$ and $\widetilde{s}$ instead of $r$ and $s$. 
\end{remark}

\subsection[Proof of Proposition~\ref{propmain} for $\widetilde{r}$ with positive $\alpha$-$\omega$-length]{Proof of Proposition~\ref{propmain} for {\boldmath $\widetilde{r}$} with positive {\boldmath$\alpha$-$\omega$-}length}

This subsection completes the proof of Proposition~\ref{propmain} in the case that $\widetilde{r}$ possesses positive $\alpha$-$\omega$-length. In the first part of this subsection we show some auxiliary lemmata that only hold in this case. In the second part we prove the central Lemma~\ref{lemcentral} which finally enables us to pass over from the normal closure of $\{(r_{i},c) \mid i \in \mathbb{Z} \}$ in $K$ to the normal closure of a single element $(r_{i},c)$ in $K$ and thereby finish the proof of this case in the third part of this subsection with the help of Lemma~\ref{lemKMP}.

\subsubsection[Properties of $\widetilde{r}$ with positive $\alpha$-$\omega$-length]{Properties of { \boldmath $\widetilde{r}$ with positive $\alpha$-$\omega$-length}}

\begin{lemma} \label{lemcontainsindex}
Let $r \in M \backslash \{1\}$ with $|r|_{\alpha,\omega} \geqslant 1$. Then every presentation of $r$ written in the generators of the generating set $\mathcal{E}_{\alpha_{r},\infty}$ of $M_{\alpha_{r},\infty}$ contains at least one generator $g_{i} \in G_{i}$ with $i \geqslant \omega_{r}$. Every presentation of $r$ written with generators of the generating set $\mathcal{E}'_{\alpha'_{r},\infty}$ of $M'_{\alpha'_{r},\infty}$ contains at least one generator $g'_{i} \in G'_{i}$ with $i \geqslant \omega'_{r}$.
\end{lemma}

\begin{proof}
We only prove the non-dual statement of this lemma. Because of the dual structure of $M$ described in Section~\ref{subsecdual} the dual statement can be shown analogously.

By assumption we have $\alpha_{r} \leqslant \omega_{r}$. For the purpose of a contradiction let
\begin{eqnarray*}
r \in \langle b_{\alpha_{r}},b_{\alpha_{r}+1},\dots , b_{\alpha_{r}+k-1}, G_{\alpha_{r}}, G_{\alpha_{r}+1}, \dots , G_ {\omega_{r}-1} \mid \rangle.
\end{eqnarray*}
Using the relations $b_{i}=b_{i-k}u_{i-k}^{-1}$ ($i \in \mathbb{Z}$) we get
\begin{eqnarray*}
r \in \langle b_{\alpha_{r}-k},b_{\alpha_{r}-k+1}, \dots , b_{\alpha_{r}-1},G_{\alpha_{r}-k}, G_{\alpha_{r}-k+1}, \dots , G_{\omega_{r}-1} \mid \rangle \subset \langle \mathcal{E}_{-\infty,\omega_{r}-1} \mid \rangle = M_{-\infty,\omega_{r}-1}
\end{eqnarray*}
which contradicts Definition~\ref{defalpha}.
\end{proof}

Using the embedding theorem from Lemma~\ref{lemmachi} we get the following corollary.

\begin{corollary} \label{corNCtilde}
Let $(\widetilde{r},c)$ be an element of $K$ with $|\widetilde{r}|_{\alpha,\omega} \geqslant 1$, where $\widetilde{r}$ is a suitable conjugate. We define
\begin{eqnarray*}
N \ \ := \ \ \langle \! \langle (\widetilde{r},c) \rangle \! \rangle_{K} \cap C \ \ \ \ \text{and} \ \ \ \ \widetilde{C} \ := \ C/N.
\end{eqnarray*}
Then $M_{\alpha_{\widetilde{r}}+1,\infty} \times \widetilde{C}$ and $M_{-\infty,\omega_{\widetilde{r}}-1} \times \widetilde{C}$ embed canonically into $K / \langle \! \langle ( \widetilde{r},c) \rangle \! \rangle$. Further $M'_{\alpha'_{\widetilde{r}}+1,\infty} \times \widetilde{C}$ and $M'_{-\infty,\omega'_{\widetilde{r}}-1} \times \widetilde{C}$ embed canonically into $K / \langle \! \langle (\widetilde{r},c) \rangle \! \rangle$.
\end{corollary}

\begin{proof}
We only prove the first of the two embeddings in the non-dual case since the other embeddings can be proved analogously. Because of Proposition~\ref{propEi}, Notation~\ref{notMst} and Remark~\ref{remMiinf} we may write
\begin{eqnarray} \label{eqKorDarK1}
K&=&M\times C= \big(\langle b_{\alpha_{\widetilde{r}}+1},b_{\alpha_{\widetilde{r}}+2}, \dots, b_{\alpha_{\widetilde{r}}+k} \mid \rangle \ \ast \ \underset{j \in \mathbb{Z}}{\bigast} \ G_{j} \big) \times C \nonumber \\
&=& \big(G_{-\infty,\alpha_{\widetilde{r}}} \ \ast \ M_{\alpha_{\widetilde{r}}+1, \infty} \big) \times C \\
&\cong &  \ (G_{-\infty,\alpha_{\widetilde{r}}} \times C) \ \underset{C}{\ast} \ (M_{\alpha_{\widetilde{r}}+1,\infty} \times C). \label{eqKorDarK2}
\end{eqnarray}

\noindent\underline{Case 1.} Let $\widetilde{r}$ be an element of $G_{-\infty,\alpha_{\widetilde{r}}}$.\\
Applying the projection $\pi \colon \big(G_{-\infty,\alpha_{\widetilde{r}}} \ \ast \ M_{\alpha_{\widetilde{r}}+1, \infty} \big) \times C \rightarrow G_{-\infty,\alpha_{\widetilde{r}}} \times C$ we get the equation $N=\langle \! \langle (\widetilde{r},c) \rangle \! \rangle_{G_{-\infty,\alpha_{\widetilde{r}}}} \cap C$. So by the definition of $\widetilde{C}$ and \eqref{eqKorDarK2} we have
\begin{eqnarray*}
K / \langle \! \langle (\widetilde{r},c) \rangle \! \rangle \ \ \ \cong \ \ \ \big( (G_{-\infty,\alpha_{\widetilde{r}}} \times C) / \langle \! \langle (\widetilde{r},c) \rangle \! \rangle \big) \ \ \underset{\widetilde{C}}{\ast} \ \ (M_{\alpha_{\widetilde{r}}+1,\infty} \times \widetilde{C})
\end{eqnarray*}
and therefore the desired embedding.

\noindent\underline{Case 2.} Assume that $\widetilde{r}$ is not an element of $G_{-\infty,\alpha_{\widetilde{r}}}$.\\
In this case $\widetilde{r}$ written as an element of the free product in \eqref{eqKorDarK1} uses the factor $M_{\alpha_{\widetilde{r}}+1, \infty}$. By Definition~\ref{defalpha} the presentation of $\widetilde{r}$ also uses the factor $G_{-\infty,\alpha_{\widetilde{r}}}$. Since $\widetilde{r}$ is a suitable conjugate and, therefore, cyclically reduced (in the sense of Definition~\ref{defcycred}), $\widetilde{r}$ can neither be conjugate to an element of $G_{-\infty,\alpha_{\widetilde{r}}}$ nor to an element of $M_{\alpha_{\widetilde{r}}+1,\infty}$. Hence, all preconditions are given to apply Lemma~\ref{lemmachi}. 
\end{proof}

\subsubsection{Structure of some quotients of {\boldmath $K$}}

We begin with some further notation.

\begin{notation} \label{notNrCtilde}
Let $(\widetilde{r},c)$ be an element of $K$ where $\widetilde{r}$ is a suitable conjugate. Then we define as before in Corollary~\ref{corNCtilde}
\begin{eqnarray*}
N_{\widetilde{r}} \ \ := \ \ \langle \! \langle (\widetilde{r},c) \rangle \! \rangle_{K} \cap C \ \ \ \ \ \text{and} \ \ \ \ \ \widetilde{C}_{\widetilde{r}} \ \ :=\ \ C / N_{\widetilde{r}}.
\end{eqnarray*}
Note that the endomorphism of $K$ that maps every generator $b_{i}$ ($i \in \mathbb{Z}$) to $b_{i+1}$, every generator $g_{i} \in G_{i}$ ($i \in \mathbb{Z}$) to $g_{i+1}$ and every generator of $C$ to itself is by Section~\ref{subsecKleftright} an automorphism of $K$ which restricts to the identity on $C$. Hence $\widetilde{C}_{\widetilde{r}}$ is for all $i \in \mathbb{Z}$ identical to $\widetilde{C}_{\widetilde{r}_{i}}$ (where $\widetilde{r}_{i} := x^{-i}\widetilde{r} x^{i}$; see Notation~\ref{notgi}).
\end{notation}

In the following we slightly abuse the notion denoting the image of $c \in C$ in the quotient group $\widetilde{C}_{\widetilde{r}}$ of $C$ also with $c$.

\begin{lemma} \label{lemcentral}
Let $(\widetilde{r},c)$ be an element of $K$ with $|\widetilde{r}|_{\alpha,\omega} \geqslant 1$ where $\widetilde{r}$ is a suitable conjugate. Further let $m$, $n \in \mathbb{Z}$ with $m<n$. We define $s:=\alpha_{\widetilde{r}_{n}}$, $t:=\omega_{\widetilde{r}_{n}-1}$ and $\omega_{i} := b_{i}u_{i}^{-1}$ for all $i \in \mathbb{Z}$. Then we have:
\begin{itemize}
\item[\emph{(1)}] $\ \ \ \ \, K / \langle \! \langle (\widetilde{r}_{i},c) \mid m \leqslant i \leqslant n \rangle \! \rangle$ \\
$\cong \ \big( (M_{-\infty,t} \times C) / \langle \! \langle (\widetilde{r}_{i},c) \mid m \leqslant i \leqslant n-1 \rangle \! \rangle \big) \underset{P \times\widetilde{C}_{\widetilde{r}}}{\ast} \big( (M_{s,\infty} \times C) /  \langle \! \langle (\widetilde{r}_{n},c) \rangle \! \rangle \big)$ \\
with $P\cong\langle w_{t-k+1},w_{t-k+2},\dots,w_{t} \mid \rangle \ast G_{s,t} = \langle b_{t+1},b_{t+2},\dots, b_{t+k} \mid \rangle \ast G_{s,t}$
\item[\emph{ (2)}] The group $M_{s+1,\infty} \times \widetilde{C}_{\widetilde{r}}$ embeds canonically into $K / \langle \! \langle (\widetilde{r}_{i},c) \mid m \leqslant i \leqslant n \rangle \! \rangle$.
\end{itemize}
\end{lemma}

\begin{proof}
First, we show that Statement~(1) implies Statement~(2) for some fixed $m$, $n$. By Corollary~\ref{corNCtilde}, $M_{s+1,\infty} \times \widetilde{C}_{\widetilde{r}}$ embeds canonically into the factor $(M_{s,\infty} \times C) / \langle \! \langle (\widetilde{r}_{n},c) \rangle \! \rangle$ of the amalgamated product in Statement (1). Thus, $M_{s+1,\infty} \times \widetilde{C}_{\widetilde{r}}$ also embeds into $K / \langle \! \langle (\widetilde{r}_{i},c) \mid m \leqslant i \leqslant n \rangle \! \rangle$.

Now, we fix an arbitrary element $m \in \mathbb{Z}$ and prove Statement~(1) by induction on $n$. Because of our preliminary consideration we can use not only Statement~(1), but also Statement~(2) as the induction hypothesis.

\noindent\textbf{Base of induction {\boldmath($m+1=n$)}.} Our aim is to show
\begin{eqnarray} \label{eqIndA}
K / \langle \! \langle (\widetilde{r}_{m},c), (\widetilde{r}_{n},c) \rangle \! \rangle \cong \big( (M_{-\infty,t} \times C)/  \langle \! \langle (\widetilde{r}_{m},c) \rangle \! \rangle \big)  \underset{P \times \widetilde{C}_{\widetilde{r}}}{\ast} \big( (M_{s,\infty} \times C) /  \langle \! \langle (\widetilde{r}_{n},c) \rangle \! \rangle \big),
\end{eqnarray}
where $P \cong \langle w_{t-k+1},w_{t-k+2},\dots , w_{t} \mid \rangle \ast G_{s,t} \ \cong \ \langle b_{t+1},b_{t+2},\dots,b_{t+k} \mid \rangle \ast G_{s,t}$. To show this canonical isomorphism it is sufficient to prove the following claims:
\begin{itemize}
\item[(a)] The group $P \times \widetilde{C}_{\widetilde{r}}$ embeds into $(M_{-\infty,t} \times C)/  \langle \! \langle (\widetilde{r}_{m},c) \rangle \! \rangle$  and $(M_{s,\infty} \times C) /  \langle \! \langle (\widetilde{r}_{n},c) \rangle \! \rangle$. (This justifies the notation as an amalgamated product in \eqref{eqIndA}.)
\item[(b)] The amalgamated product on the right side of \eqref{eqIndA} is isomorphic to the left side $K / \langle \! \langle (\widetilde{r}_{m},c), (\widetilde{r}_{n},c) \rangle \! \rangle$.
\end{itemize}

\emph{Proof of $(a)$.} We only prove the embedding of $P \times \widetilde{C}_{\widetilde{r}}$ into $(M_{s,\infty} \times C) / \langle \! \langle ( \widetilde{r}_{n},c) \rangle \! \rangle$. The other embedding can be proved analogously. Note that:
\begin{eqnarray*}
P&=& \langle b_{t+1},b_{t+2},\dots,b_{t+k} \mid \rangle  \ast  G_{s,t} = \langle b_{t}u_{t}^{-1},b_{t+1},\dots,b_{t+k-1} \mid \rangle  \ast  G_{s,t} \\
&=&  \langle b_{t},b_{t+1},\dots,b_{t+k-1} \mid \rangle  \ast  G_{s,t} = \ \dots \ = \langle b_{s},b_{s+1},\dots,b_{s+k-1} \mid \rangle  \ast  G_{s,t}
\end{eqnarray*}
In particular, $\mathcal{P} := \mathcal{E}_{s, \infty} \backslash \underset{i \geqslant t+1}{\bigcup} G_{i}$ is a generating set of $P$. Using Remark~\ref{remMiinf} we write:
\begin{eqnarray} \label{eqIndAnfDar}
 M_{s,\infty} &=&  \langle b_{s},b_{t+2},\dots,b_{s+k-1} \mid \rangle  \ast  G_{s,t} \ast G_{t+1,\infty} \ = \ P \ast G_{t+1,\infty} 
\end{eqnarray}
Because of Lemma~\ref{lemcontainsindex} and $s=\alpha_{\widetilde{r}_{n}}$, $\widetilde{r}_{n}$ written with generators from the generating set $\mathcal{E}_{s,\infty}$ of $M_{s, \infty}$ contains at least one generator from the generating set $\mathcal{Q} := \underset{i \geqslant t+1}{\bigcup} G_{i}$ of $G_{t+1,\infty}$. Since $\mathcal{E}_{s, \infty}$ is the disjoint union of $\mathcal{P}$ and $\mathcal{Q}$, $\widetilde{r}_{n}$ written as an element of $P \ast G_{t+1,\infty}$ contains at least one piece from $G_{t+1,\infty}$. If $\widetilde{r}_{n}$ uses no piece from $P$, we may write
\begin{eqnarray*}
(M_{s,\infty} \times C) / \langle \! \langle (\widetilde{r}_{n},c) \rangle \! \rangle \cong (P \times \widetilde{C}_{\widetilde{r}}) \underset{\widetilde{C}_{\widetilde{r}}}{\ast} \big( (G_{t+1,\infty} \times C) / \langle \! \langle (\widetilde{r}_{n},c) \rangle \! \rangle \big)
\end{eqnarray*}
because of the definition of $\widetilde{C}_{\widetilde{r}}$ (see Notation~\ref{notNrCtilde}) and \eqref{eqIndAnfDar}. Thus, $P \times \widetilde{C}_{\widetilde{r}}$ embeds into $(M_{s, \infty} \times C) / \langle \! \langle ( \widetilde{r}_{n},c) \rangle \! \rangle$. So we can assume that $\widetilde{r}_{n}$ written as an element of $P \ast G_{t+1,\infty}$ contains at least one piece from each of the two factors. Since $\widetilde{r}_{n}$ is cyclically reduced (in the sense of Definition~\ref{defcycred}), $(\widetilde{r}_{n},c)$ can neither be conjugate to an element of $P \times C$ nor to an element of $G_{t+1,\infty} \times C$. Note that $P$ and $G_{t+1,\infty}$ are free products of locally indicable groups and, therefore, locally indicable themselves. Applying Lemma~\ref{lemmachi}, we get the embedding of $P \times \widetilde{C}_{\widetilde{r}}$ into $(M_{s, \infty} \times C) / \langle \! \langle (\widetilde{r}_{n},c) \rangle \! \rangle$.

Claim~(b) can easily be shown by finding a common presentation for the two sides of  \eqref{eqIndA}. This finishes the base of induction.

\noindent\textbf{Inductive step {\boldmath($(m,n) \rightarrow (m,n+1)$)}.}\\
Recall $t=\omega_{\widetilde{r}_{n}}$ and $s= \alpha_{\widetilde{r}_{n}}$. Our aim is to show the presentation
\begin{eqnarray*}
& & K / \langle \! \langle (\widetilde{r}_{i},c) \mid m \leqslant i \leqslant n+1 \rangle \! \rangle \\
&\cong& \big( (M_{-\infty,t+1} \times C) / \langle \! \langle (\widetilde{r}_{i},c) \mid m \leqslant i \leqslant n \rangle \! \rangle \big) \underset{P \times \widetilde{C}_{\widetilde{r}}}{\ast} \big( (M_{s+1,\infty} \times C) /  \langle \! \langle (\widetilde{r}_{n+1},c) \rangle \! \rangle \big),
\end{eqnarray*}
where $P\cong\langle w_{t-k+2},w_{t-k+3},\dots,w_{t+1} \mid \rangle \ast G_{s+1,t+1} = \langle b_{t+2},b_{t+3},\dots, b_{t+k+1} \mid \rangle \ast G_{s+1,t+1}$. It is sufficient to prove that this amalgamated product is well-defined. Then the two sides can easily be shown to be isomorphic by finding a common presentation. Thus, we want to prove the embeddings of $P \times \widetilde{C}_{\widetilde{r}}$ into $L:=(M_{-\infty,t+1} \times C) / \langle \! \langle (\widetilde{r}_{i},c) \mid m \leqslant i \leqslant n \rangle \! \rangle$ and $(M_{s+1,\infty} \times C) /  \langle \! \langle (\widetilde{r}_{n+1},c) \rangle \! \rangle$. The embedding of $P \times \widetilde{C}_{\widetilde{r}}$ into $\big( (M_{s+1,\infty} \times C) /  \langle \! \langle (\widetilde{r}_{n+1},c) \rangle \! \rangle \big)$ can be proved in the same way as in the base of induction. To prove the embedding of $P \times \widetilde{C}_{\widetilde{r}}$ into $L$ we consider the following commutative diagram where all non-labeled embeddings are canonical embeddings as subgroups.
\begin{displaymath}
	\xymatrix@C=2cm@R=2cm{
        P \times \widetilde{C}_{\widetilde{r}} \ \ar@{^{(}->}[r] \ \ar@{^{(}->}[d] \ar@{-->}[rrd]^{\text{\large $\varphi$}} & M_{s+1,\infty} \times \widetilde{C}_{\widetilde{r}} \ \ar@{^{(}->}[r]^{\hspace{-1.1cm} by \ (2)} &  \ K / \langle \! \langle (\widetilde{r}_{i},c) \mid m \leqslant i \leqslant n \rangle \! \rangle\\
        M_{-\infty,t+1} \times \widetilde{C}_{\widetilde{r}} \ \ar[rr]_{}  && \ \ \ \ \  L \ \ \ \ \ar@{^{(}->}[u] }
\vspace{0.3cm}
\end{displaymath}
The marked homomorphism $\varphi$ shall be the composition of the embedding of $P \times \widetilde{C}_{\widetilde{r}}$ as a subgroup in $M_{-\infty,t+1} \times \widetilde{C}_{\widetilde{r}}$ and the homomorphism of the group $M_{-\infty,t+1} \times \widetilde{C}_{\widetilde{r}}$ to its quotient group $L$. We want to show that $\varphi$ is an embedding. For this purpose we consider $P$ as a subgroup of $M_{s+1,\infty}$ and apply Statement (2) to embed $M_{s+1,\infty} \times \widetilde{C}_{\widetilde{r}}$ into $K/ \langle \! \langle (\widetilde{r}_{i},c) \mid m \leqslant i \leqslant n \rangle \! \rangle$. Overall we get the embedding of $P \times \widetilde{C}_{\widetilde{r}}$ into $K/ \langle \! \langle (\widetilde{r}_{i},c) \mid m \leqslant i \leqslant n \rangle \! \rangle$. We find this group in the upper right corner of the diagram. To the contrary, assume that $\varphi$ is not an embedding into $L$, then the combination of $\varphi$ with the embedding of $L$ as a subgroup into $K/ \langle \! \langle (\widetilde{r}_{i},c) \mid m \leqslant i \leqslant n \rangle \! \rangle$ is also not an embedding. This is a contradiction.
\end{proof}

Finally, we need the following ``reflected'' version of Lemma~\ref{lemcentral}

\begin{lemma} \label{lemcentralrefl}
Let $(\widetilde{r},c)$ be an element of $K$ with $|\widetilde{r}|_{\alpha,\omega} \geqslant 1$ where $\widetilde{r}$ is a suitable conjugate. Further let $m$, $n \in \mathbb{Z}$ with $m < n$. We define $s:=\alpha_{\widetilde{r}_{m}}+1$, $t:=\omega_{\widetilde{r}_{m}}$ and $w_{i}:=b_{i}u_{i}^{-1}$ for all $i \in \mathbb{Z}$. Then we have:
\begin{itemize}
\item[\emph{ (1)}] $\ \ \ \ \, K / \langle \! \langle (\widetilde{r}_{i},c) \mid m \leqslant i \leqslant n \rangle \! \rangle$ \\
$\cong \ \big( (M_{-\infty,t} \times C) / \langle \! \langle (\widetilde{r}_{m},c) \rangle \! \rangle \big) \underset{P \times \widetilde{C}_{\widetilde{r}}}{\ast} \big( (M_{s,\infty} \times C) /  \langle \! \langle (\widetilde{r}_{i},c) \mid m+1 \leqslant i \leqslant n \rangle \! \rangle \big)$, \\
where $P\cong\langle w_{s-k},w_{s-k+1},\dots,w_{s-1} \mid \rangle \ast G_{s,t} = \langle b_{s},b_{s+1},\dots, b_{s+k-1} \mid \rangle \ast G_{s,t}$
\item[\emph{ (2)}] The group $M_{-\infty,t-1} \times \widetilde{C}_{\widetilde{r}}$ embeds canonically into $K / \langle \! \langle (\widetilde{r}_{i},c) \mid m \leqslant i \leqslant n \rangle \! \rangle$.
\end{itemize}
\end{lemma}.

\begin{proof}
We deduce this lemma with the help of the dual structure of $K$ from Lemma~\ref{lemcentral}. For this purpose let $\widetilde{r}'_{i}$ ($i \in \mathbb{Z}$) be a presentation of $\widetilde{r}_{-i}$ in dual generators. Using the notation of Section~\ref{subsecdual} we write
\begin{eqnarray*}
w_{i} \ \ = \ \ b_{i}u_{i}^{-1} \ \ = \ \ b'_{-i} \ \ \ \ \ \text{and} \ \ \ \ \ b_{i} \ \ = \ \ b_{i}u_{i}^{-1}u_{i} \ \ = \ \ b_{-i}' u_{-i}'^{-1} \ \ =: \ \ w'_{-i}.
\end{eqnarray*}
Further we define $m':=-n$, $n':=-m$, $s':=-t$ and $t':=-s$. Rewriting the Statements (1) and (2) using the dual objects we get:
\begin{itemize}
\item[\textrm{ (1)}] $\ \ \ \ \, K / \langle \! \langle (\tilde{r}'_{i},c) \mid m' \leqslant i \leqslant n' \rangle \! \rangle$ \\
$\cong \ \big( (M'_{-\infty,t'} \times C) / \langle \! \langle (\tilde{r}'_{i},c) \mid m' \leqslant i \leqslant n'-1 \rangle \! \rangle \big) \underset{P' \times \widetilde{C}_{\widetilde{r}}}{\ast} \big( (M'_{s',\infty} \times C) /  \langle \! \langle (\tilde{r}'_{n'},c) \rangle \! \rangle \big)$, \\
where $P' \cong \langle b'_{t'+1},b'_{t'+2},\dots,b'_{t'+k} \mid \rangle \ast G_{s',t'} \cong \langle w'_{t'-k+1},w'_{t'-k+2},\dots,w'_{t'} \mid \rangle \ast G_{s',t'}$
\item[\textrm{ (2)}] The group $M_{s'+1,\infty} \times \widetilde{C}_{\widetilde{r}}$ embeds canonically into $K / \langle \! \langle (\tilde{r}'_{i},c) \mid m' \leqslant i \leqslant n' \rangle \! \rangle$.
\end{itemize}
Because of Lemma~\ref{lemalphaomega} we have $s'=-t=-\omega_{\widetilde{r}_{m}}=-\omega_{\widetilde{r}'_{-m}} = \alpha'_{\widetilde{r}'_{n'}}$ and $t'=-s=-\alpha_{\widetilde{r}_{m}}-1=\omega'_{\widetilde{r}'_{-m}}-1=\omega'_{\widetilde{r}'_{n'}}-1$. Therefore both statements have the same form as the statements of Lemma~\ref{lemcentral}. However we are now working with the dual objects. Since all results used in the proof of Lemma~\ref{lemcentral} are also given for the dual objects, Lemma~\ref{lemcentralrefl} can be proved analogously to Lemma~\ref{lemcentral}.
\end{proof}

\subsubsection{Conclusion of the proof for {\boldmath $\widetilde{r}$ with positive $\alpha$-$\omega$-length}}

In Proposition~\ref{propmain} there are two elements $(r,c)$, $(s,d) \in \widetilde{H}$ with the same normal closure and $r_{x}=0$. Thus, as described in Section~\ref{subsecKleftright} the normal closures of $\{(r_{i},c) \mid i \in \mathbb{Z} \}$ and $\{ (s_{i},d) \mid i \in \mathbb{Z} \}$ are equal in $K$. We already noticed in Remark~\ref{remconjtosconj} that we can w.\,l.\,o.\,g. consider $\{(\widetilde{r}_{i},c) \mid i \in \mathbb{Z} \}$ and $\{ (\widetilde{s}_{i},d) \mid i \in \mathbb{Z} \}$
instead of $\{(r_{i},c) \mid i \in \mathbb{Z} \}$ and $\{ (s_{i},d) \mid i \in \mathbb{Z} \}$, where $\widetilde{r}_{i}$ and $\widetilde{s}_{i}$ ($i \in \mathbb{Z}$) are suitable conjugates. Note that $(\widetilde{s}_{0},d)$ is trivial in $K / \langle \! \langle (\widetilde{r}_{i},c) \mid i \in \mathbb{Z} \rangle \! \rangle$. So there exist $m$, $n \in \mathbb{Z}$ with $m \leqslant n$ such that $(\widetilde{s},d)$ is trivial in $K / \langle \! \langle (\widetilde{r}_{i},c) \mid m \leqslant i \leqslant n \rangle \! \rangle$. We choose a pair $(m,n)$ with this property such that $n-m$ is minimal. Our aim is to show $m=n$. For this purpose we first prove the following lemma.

\begin{lemma} \label{leminequabc}
Let $|\widetilde{r}|_{\alpha,\omega} \geqslant 1$. Then we have
\begin{eqnarray*}
\text{\emph{(a)}} \ \ \omega_{\widetilde{s}_{0}} \geqslant \omega_{\widetilde{r}_{n}}\geqslant \omega_{\widetilde{r}_{m}}, \ \ \ \ \ \text{\emph{(b)}} \ \ \alpha_{\widetilde{s}_{0}}\leqslant \alpha_{\widetilde{r}_{m}}\leqslant \alpha_{\widetilde{r}_{n}} \ \ \ \ \text{and} \ \ \ \ \text{\emph{(c)}} \ \ |\widetilde{s}|_{\alpha,\omega} \geqslant |\widetilde{r}|_{\alpha,\omega}.
\end{eqnarray*}
\end{lemma}

\begin{proof}
To prove (a) we assume to the contrary that $\omega_{\widetilde{s}_{0}} < \omega_{\widetilde{r}_{n}}$. Then $\widetilde{s}_{0}$ is an element of $M_{-\infty,\omega_{\widetilde{r}_{n}-1}}$. By the choice of $m$, $n$ the element $(\widetilde{s}_{0},d)$ is trivial in $K /  \langle \! \langle (\widetilde{r}_{i},c) \mid m \leqslant i \leqslant n \rangle \! \rangle$. In the case $m<n$ we have by Lemma~\ref{lemcentral}~(1) the presentation
\begin{eqnarray} \label{eqPrFK}
\big( (M_{-\infty,\tau} \times C) / \langle \! \langle (\widetilde{r}_{i},c) \mid m \leqslant i \leqslant n-1 \rangle \! \rangle \big) \underset{P \times \widetilde{C}_{\widetilde{r}}}{\ast} \big( (M_{\sigma,\infty} \times C) /  \langle \! \langle (\widetilde{r}_{n},c) \rangle \! \rangle \big)
\end{eqnarray}
of $K /  \langle \! \langle (\widetilde{r}_{i},c) \mid m \leqslant i \leqslant n \rangle \! \rangle$, where $\sigma = \alpha_{\widetilde{r}_{n}}$ and $\tau= \omega_{\widetilde{r}_{n}}-1$. Thus, $(\widetilde{s}_{0},d)$ is for $m<n$ already trivial in the left factor of the amalgamated product \eqref{eqPrFK}. This contradicts the minimality of $n-m$. In the case $m=n$, $(\widetilde{s}_{0},d)$ would already be trivial in $M_{-\infty,\tau} \times \widetilde{C}_{\widetilde{r}}$ because of
\begin{eqnarray*}
K / \langle \! \langle (\widetilde{r}_{i},c) \mid m \leqslant i \leqslant n \rangle \! \rangle \ \ \cong \ \ (M_{-\infty,\tau} \times \widetilde{C}_{\widetilde{r}}) \underset{P \times \widetilde{C}_{\widetilde{r}}}{\ast} \big( (M_{\sigma,\infty} \times C) /  \langle \! \langle (\widetilde{r}_{n},c) \rangle \! \rangle \big).
\end{eqnarray*}
This contradicts the fact that $r$ and therefore also $s$ is not trivial in $K$ by the assumptions of Proposition~\ref{propmain}. The inequality $\omega_{\widetilde{r}_{n}} \geqslant \omega_{\widetilde{r}_{m}}$ follows from Lemma~\ref{lemalphaomega} because of $n \geqslant m$. Inequality~(b) can be proved analogously to (a) with the help of Lemma~\ref{lemcentralrefl}.\\
Using Lemma~\ref{lemalphaomega} along with (a) and (b) we get
\begin{eqnarray*}
|\widetilde{s}|_{\alpha,\omega} = |\widetilde{s}_{0}|_{\alpha,\omega} = \omega_{\widetilde{s}_{0}}-\alpha_{\widetilde{s}_{0}}+1 \geqslant \omega_{\widetilde{r}_{m}}-\alpha_{\widetilde{r}_{m}}+1 = |\widetilde{r}_{m}|_{\alpha,\omega} = |\widetilde{r}|_{\alpha,\omega}.
\end{eqnarray*}
\end{proof}

The inequality $|\widetilde{s}|_{\alpha,\omega} \leqslant |\widetilde{r}|_{\alpha,\omega}$ follows by symmetry because of Lemma~\ref{leminequabc} and $|\widetilde{r}|_{\alpha,\omega} \geqslant 1$. Altogether we have $|\widetilde{s}|_{\alpha,\omega} = |\widetilde{r}|_{\alpha,\omega}$.

\begin{remark} \label{rembigcases}
We have just shown that $\widetilde{r}$ possesses a positive $\alpha$-$\omega$-length if and only if $\widetilde{s}$ possesses a positive $\alpha$-$\omega$-length.
\end{remark}

It follows with Lemma~\ref{lemalphaomega}:
\begin{eqnarray} \label{eq3}
|\widetilde{r}_{m}|_{\alpha,\omega} = |\widetilde{s}_{0}|_{\alpha,\omega} \ \ \ \Leftrightarrow \ \ \ \omega_{\widetilde{r}_{m}}-\alpha_{\widetilde{r}_{m}} = \omega_{\widetilde{s}_{0}}-\alpha_{\widetilde{s}_{0}}  \ \ \ \Leftrightarrow \ \ \ \omega_{\widetilde{r}_{m}}- \omega_{\widetilde{s}_{0}} =\alpha_{\widetilde{r}_{m}}-\alpha_{\widetilde{s}_{0}} 
\end{eqnarray}

Using \eqref{eq3} along with Lemma~\ref{leminequabc}~(a) and (b) we get
\begin{eqnarray*}
0 \geqslant \omega_{\widetilde{r}_{m}}- \omega_{\widetilde{s}_{0}} = \alpha_{\widetilde{r}_{m}}-\alpha_{\widetilde{s}_{0}} \geqslant 0.
\end{eqnarray*}
Therefore we have $\omega_{\widetilde{r}_{m}} = \omega_{\widetilde{s}_{0}}$. Because of Lemma~\ref{leminequabc}~(a) and Lemma~\ref{lemalphaomega} it follows $m=n$. This means that $(\widetilde{s}_{0},d)$ is trivial in $K/ \langle \! \langle (\widetilde{r}_{m},c) \rangle \! \rangle$ where the index $m$ is uniquely determined by $\omega_{\widetilde{r}_{m}} = \omega_{\widetilde{s}_{0}}$. In the same way one can show that $(\widetilde{r}_{m},c)$ is trivial in $K/ \langle \! \langle (\widetilde{s}_{0},d) \rangle \! \rangle$. So the normal closures of $(\widetilde{r}_{m},c)$ and $(\widetilde{s}_{0},d)$ are equal in $K$ and, by Lemma~\ref{lemKMP}, $(\widetilde{r}_{m},c)$ is conjugate in $K$ to $(\widetilde{s}_{0},d)$ or $(\widetilde{s}_{0},d)^{-1}$. Since $K$ is a subgroup of $\widetilde{H}$ and we have $\widetilde{r}_{m}=x^{-m} \widetilde{r} x^{m}$ as well as $\widetilde{s}_{0}=\widetilde{s}$ (see Notation ~\ref{notgi}), $(\widetilde{r},c)$ is conjugate to $(\widetilde{s},d)$ or $(\widetilde{s},d)^{-1}$ in $\widetilde{H}$. This finishes the proof of Proposition~\ref{propmain} in the case $|\widetilde{r}|_{\alpha, \omega} \geqslant 1$.

\subsection[Proof of Proposition~\ref{propmain} for $\widetilde{r}$ with non-positive $\alpha$-$\omega$-length]{Proof of Proposition~\ref{propmain} for {\boldmath $\widetilde{r}$ with non-positive $\alpha$-$\omega$-length}}

This subsection completes the proof of Proposition~\ref{propmain} for the remaining case that $\widetilde{r}$ possesses non-positive $\alpha$-$\omega$-length (i.\,e. $|\widetilde{r}|_{\alpha,\omega} <1$). In the first part of this subsection we prove some lemmata that will be useful in this case and recall two results by A. Karras, W. Magnus and D. Solitar concerning torsion elements in one-relator groups. With the help of these results we finish the proof of Proposition~\ref{propmain} in the second part of this subsection.

\subsubsection{Preparations and auxiliary results} 

\begin{lemma} \label{lemexclb}
Let $r$ be an element in $M$ with non-positive $\alpha$-$\omega$-length. Then every reduced presentation of $r$ with respect to the generating set $\mathcal{E}_{\alpha_{r},\infty}$ of $M_{\alpha_{r},\infty}$ uses exclusively $b$-generators.
\end{lemma}

\begin{proof}
Because of $|r|_{\alpha,\omega} < 1$ we have $\omega_{r} < \alpha_{r}$. By Definition~\ref{defalpha}, $r$ is an element of $M_{-\infty,\omega_{r}}$ as well as an element of $M_{\alpha_{r},\infty}$. For the purpose of this proof we introduce the following notation:
\begin{itemize}
\item Let $r(a)$ be the reduced presentation of $r$ with respect to the generating set\\
$\mathcal{E}_{\alpha_{r},\infty} = \{b_{\alpha_{r}},b_{\alpha_{r}+1},\dots,b_{\alpha_{r}+k-1} \} \cup \underset{j \geqslant \alpha_{r}}{\bigcup} G_{j}$ of $M_{\alpha_{r},\infty}$.
\item Let $r(\omega)$ be the reduced presentation of $r$ with respect to the generating set\\
$\mathcal{E}_{-\infty,\omega_{r}} = \{b_{\omega_{r}-k+1}, b_{\omega_{r}-k+2},\dots,b_{\omega_{r}}\} \cup \underset{j \leqslant \omega_{r}}{\bigcup} G_{j}$ of $M_{-\infty,\omega_{r}}$. 
\end{itemize}

For every index $i \in \{\omega_{r}-k+1,\omega_{r}-k+2, \dots, \omega_{r} \}$ there exists a unique index $j \in \{\alpha_{r},\alpha_{r}+1, \dots, \alpha_{r}+k-1\}$ with $j \equiv i \mod k$. Thus, we get the presentation $r(\alpha)$ from $r(\omega)$ by replacing every generator $b_{i}$ in $r(\omega)$ with $b_{j}u_{j-k}u_{j-2k}\dots u_{i}$ and reducing afterwards. The indices of the generators $g_{\ell} \in \bigcup_{p \in \mathbb{Z}} G_{p}$ which were added during the replacements are in the interval $(-\infty,\alpha_{r}-1]$. The indices of the generators $g_{\ell} \in \bigcup_{p\in \mathbb{Z}} G_{p}$ originally contained in $r(\omega)$ are in the interval $(-\infty,\omega_{r}] \subset (-\infty, \alpha_{r}-1]$. Thus, the indices of all generators of $r(\alpha)$ which are no $b$-generators are in the interval $(-\infty,\alpha_{r}-1]$. Because of the definition of $r(\alpha)$ it follows that $r(\alpha)$ exclusively uses $b$-generators.
\end{proof}

The following two statements prepare a case-by-case analysis in the next part of this subsection.

\begin{corollary} \label{cor0pos}
Let $r \in G \underset{u=[x,b]}{\ast} F(x,b)$ be the element from the assumptions of Proposition~\ref{propmain}. If $r$ has an exponential sum of $0$ with respect to $b$, $r$ considered as an element of $M$ has a positive $\alpha$-$\omega$-length.
\end{corollary}

\begin{proof}
Because of the assumptions of Proposition~\ref{propmain} we have $k=1$. W.\,l.\,o.\,g.~let $r \in M$ be a suitable conjugate. For the purpose of a contradiction we assume $|r|_{\alpha,\beta} \leqslant 0$. By Lemma~\ref{lemexclb} the reduced presentation of $r$ with respect to $\mathcal{E}_{\alpha_{r},\infty} = \{b_{\alpha_{r}}\} \cup \bigcup_{j \geqslant \alpha_{r}} G_{j}$ exclusively uses $b$-generators. Thus, we get the presentation $r = x^{-\alpha_{r}} b^{j} x^{\alpha_{r}}$ in $G \ast_{u = [x,b]} F(x,b)$ for a number $j \in \mathbb{Z}$. Since the exponential sum of $r$ with respect to $b$ is $0$ by assumption, we have $r=1$. This contradicts the assumptions of Proposition~\ref{propmain}.
\end{proof}
 
\begin{lemma} \label{lemmunu}
Let $B$, $C$ be groups and let $Z=\langle z \rangle$ be an infinite cyclic subgroup of $B$. Further let $(r,c)$ be an element of $B \times C$. Then there exist unique numbers $\mu \in \mathbb{N}_{0}$, $\nu \in \mathbb{Z}$ and a normal subgroup $D$ of $C$ with
\begin{eqnarray} \label{eqXY}
\langle \! \langle (r,c) \rangle \! \rangle_{B \times C} \ \cap \ (Z \times C) \  = \ \langle (z^{\mu},c^{\nu}) \rangle_{Z \times C} \ \text{\emph{\textbullet}} \ (\{1\} \times D),
\end{eqnarray}
where $\text{\emph{\textbullet}}$ denotes the product of subgroups.
\end{lemma} 
 
\begin{proof}
Let $X$ resp. $Y$ be the left resp. right side of equality \eqref{eqXY}. Our proof will be by constructing $\mu$, $\nu$ and $D$. First, we define the normal subgroup $D$ of $C$ by
\begin{eqnarray*}
\{1\} \times D \ \ \ = \ \ \ \langle \! \langle (r,c) \rangle \! \rangle_{B \times C} \ \cap \ (\{1\} \times C) \ \ \ \subseteq \ \ \ \langle \! \langle (r,c) \rangle \! \rangle_{B \times C} \ \cap \ (Z \times C) \ \ \ = \ \ \ X.
\end{eqnarray*}
Therefore, we have ensured that all elements of the form $(1,f)$ with $f \in C$ which are contained in $X$ are also in $Y$. If $X$ contains no element $(z^{i},f)$ with $i \neq 0$ and $f \in C$, we set $\mu = \nu = 0$. In this case there is nothing more to show. Elsewise we define $\mu$ as the smallest number $i \in \mathbb{N}$ such that $X$ contains an element $(z^{i},f)$ with $f \in C$. Then there is a presentation
\begin{eqnarray*}
z^{\mu} = \overset{n}{\underset{j=1}{\Pi}} v_{j}^{-1} r^{\varepsilon_{j}} v_{j} \ \ \text{in} \ \ B , \ \ \text{where} \ \ v_{j} \in B, \ \ \varepsilon_{j} \in \mathbb{Z} \ \ \text{and} \ \ n \in \mathbb{N}.
\end{eqnarray*}
We choose a presentation of that kind such that the absolute value of $\overset{n}{\underset{j=1}{\Sigma}} \varepsilon_{j}$ is minimal and denote this absolute value by $\nu$. Then
\begin{eqnarray*}
\langle (z^{\mu},c^{\nu}) \rangle_{Z \times C} = \{(z^{\mu},c^{\nu})^{i} \mid i \in \mathbb{Z} \} = \{(z^{\mu i},c^{\nu i}) \mid i \in \mathbb{Z} \}
\end{eqnarray*}
is a subset of $X$. Because of the choice of $\mu$ we have conversely for every  element of the form $(z^{\ell},f) \in X$ ($f \in C$) that $\ell$ is a multiple of $\mu$. If there exist an element $(z^{\mu i},f)$ in $X$ with $f \neq c^{\nu i}$, then $c^{-\nu i} f$ is in $D$ by definition. It follows that $(z^{\mu i},f) = (z^{\mu i},c^{\nu i}) (1,c^{-\nu i}f)$ is also an element of $Y$. Thus, we have $X \subseteq Y$. Since both factors of the product $Y$ are subsets of $X$ by construction, we also have $Y \subseteq X$ and therefore $X=Y$.
\end{proof}

Finally, we recall two results by A. Karras, W. Magnus and D. Solitar concerning torsion elements of one-relator groups.

\begin{theorem} \textbf{\emph{\cite[cf. Theorem 1]{ArtMaKaSo}}} \label{thmMKS1}
Let $G$ be a group with generators $a$, $b$, $c, \dots$ and a single defining relation $R(a,b,c,\dots)=1$. If $G$ has an element of finite order, then $R$ must be a proper power, i.\,e. $R=W^{k}$, $k >1$, $W \neq 1$. 
\end{theorem}

\begin{theorem}\textbf{\emph{ \cite[cf. Theorem 3]{ArtMaKaSo}}} \label{thmMKS2}
Let $G$ be a group with generators $a$, $b$, $c, \dots$ and a single defining relator $V^{k}(a,b,c,\dots)$, $k > 1$, where $V(a,b,c,\dots)$ is not itself a proper power. Then $V$ has order $k$ and the elements of finite order in $G$ are just the powers of $V$ and their conjugates.
\end{theorem}

\subsubsection{Conclusion of the proof for {\boldmath $\widetilde{r}$ with non-positive $\alpha$-$\omega$-length}}

In Proposition~\ref{propmain} there are two elements $(r,s)$, $(s,d) \in \widetilde{H}$ with the same normal closure and $r_{x}=0$. As in the case for positive $\alpha$-$\omega$-length we consider $\{(\widetilde{r}_{i},c) \mid i \in \mathbb{Z} \}$ and $\{(\widetilde{s}_{i},d) \mid i \in \mathbb{Z} \}$, where $\widetilde{r}_{i}$ and $\widetilde{s}_{i}$ ($i \in \mathbb{Z}$) are suitable conjugates. Since we assume $| \widetilde{r}|_{\alpha,\omega} <1$, we also have $| \widetilde{s}|_{\alpha,\omega} <1$ by Remark~\ref{rembigcases}. Because of Lemma~\ref{lemexclb} we can present the elements $\widetilde{s}_{0}$ and $\widetilde{r}_{0}$ of $K$ only using $b$-generators. Thus, the elements $\widetilde{r}$ and $\widetilde{s}$ of
\begin{eqnarray*}
G \ \underset{u=[x^{k},b]}{\ast} \ \langle x,b \mid \rangle 
\end{eqnarray*}
can be presented exclusively using the generators $x$ and $b$. In the following, we no longer want to consider $K$, but work with the elements $(\widetilde{r},c)$, $(\widetilde{s},d)$ in the group
\begin{eqnarray*}
\widetilde{H}=(G \underset{u=[x^{k},b]}{\ast} \langle x,b \mid \rangle ) \times C = (G \times C) \underset{Z \times C}{\ast} (\langle x,b \mid \rangle \times C),
\end{eqnarray*}
where $Z$ is the infinite cyclic subgroup generated by $u$ in $G$ and by $[x^{k},b]$ in $\langle x,b \mid \rangle$. Because of the preliminary consideration we know that $(\widetilde{r},c)$ and $(\widetilde{s},d)$ are elements of the right factor $\langle x,b \mid \rangle \times C$ of the amalgamated product.

Our next aim is to show that $(\langle x,b \mid \rangle \times C) / \langle \! \langle (\widetilde{r},c) \rangle \! \rangle$ embeds into $\widetilde{H} / \langle \! \langle (\widetilde{r},c) \rangle \! \rangle$. For this purpose we consider three cases for the number $\mu \in \mathbb{N}_{0}$ given uniquely by Lemma~\ref{lemmunu} with $B=\langle x,b \mid \rangle$ and $Z = \langle [x^{k},b] \rangle$.

\noindent\textbf{Case 1.} Let $\mu =0$.\\
Then we have $\langle \! \langle (\widetilde{r},c) \rangle \! \rangle_{\langle x,b \mid \rangle \times C} \ \cap \ (Z \times C) \  = \ \{1\} \times D$ and therefore
\begin{eqnarray*}
\widetilde{H} / \langle \! \langle (\widetilde{r},c) \rangle \! \rangle = (G \times (C/D) ) \underset{Z \times (C/D)}{\ast} \big( (\langle x,b \mid \rangle \times C)/ \langle \! \langle (\widetilde{r},c) \rangle \! \rangle \big).
\end{eqnarray*}
So we have the embedding of $(\langle x,b \mid \rangle \times C) / \langle \! \langle (\widetilde{r},c) \rangle \! \rangle$ into $\widetilde{H} / \langle \! \langle (\widetilde{r},c) \rangle \! \rangle$.

\noindent\textbf{Case 2.} Let $\mu =1$.\\
If $([x^{k},b],1)$ is an element of $\langle \! \langle (\widetilde{r},c) \rangle \! \rangle_{\langle x,b \mid \rangle \times C} \ \cap \ (Z \times C)$, we have by Lemma~\ref{lemmunu}
\begin{eqnarray*}
X:=\langle \! \langle (\widetilde{r},c) \rangle \! \rangle_{\langle x,b \mid \rangle \times C} \ \cap \ (Z \times C) \  = Z \times D,
\end{eqnarray*}
where $D$ is a normal subgroup of $C$. So we can write
\begin{eqnarray*}
\widetilde{H} / \langle \! \langle (\widetilde{r},c) \rangle \! \rangle = \big((G/\langle \! \langle u \rangle \! \rangle) \times (C/D) \big) \underset{\{1\} \times (C/D)}{\ast} \big( (\langle x,b \mid \rangle \times C)/ \langle \! \langle (\widetilde{r},c) \rangle \! \rangle \big).
\end{eqnarray*}
This gives the desired embedding and we may assume in the following that $([x^{k},b],1)$ is no element of $X$, but that there exists some $c \in C$ such that $([x^{k},b],c)$ is an element of $X$. In particular $[x^{k},b]$ is in the normal closure of $\widetilde{r}$ in $\langle x,b \mid \rangle$ and we get a presentation
\begin{eqnarray} \label{eqExpS}
[x^{k},b] = \overset{n}{\underset{j=1}{\Pi}} v_{j}^{-1} \widetilde{r}^{\varepsilon_{j}} v_{j} \ \ \text{in} \ \ \langle x,b \mid \rangle , \ \ \text{where} \ \ v_{j} \in \langle x,b \mid \rangle, \ \ \varepsilon_{j} \in \mathbb{Z} \ \ \text{and} \ \ n \in \mathbb{N}.
\end{eqnarray}
For $\overset{n}{\underset{j=1}{\Sigma}} \varepsilon_{j}=0$ we have
\begin{eqnarray*}
([x^{k},b],1) = \overset{n}{\underset{j=1}{\Pi}} (v_{j},1)^{-1} (\widetilde{r},c)^{\varepsilon_{j}} (v_{j},1) \ \ \text{in} \ \ \langle x,b \mid \rangle \times C.
\end{eqnarray*}
It follows that $([x^{k},b],1)$ is an element of $X$. This contradicts the assumption. Thus, we have $\overset{n}{\underset{j=1}{\Sigma}} \varepsilon_{j} \neq 0$. Since the exponential sums with respect to $x$ and $b$ in the free group $\langle x, b \mid \rangle$ are well defined, $\widetilde{r}$ possesses because of \eqref{eqExpS} an exponential sum of $0$ with respect to $x$ as well as with respect to $b$. By Corollary~\ref{cor0pos}, that contradicts the assumption of this subsection that $r$ has non-positive $\alpha$-$\omega$-length.

\noindent\textbf{Case 3.} Let $\mu > 1$.\\
From Lemma~\ref{lemmunu} (with $B=\langle x,b \mid \rangle$, $Z= \langle [x^{k},b] \rangle )$ it follows that $[x^{k},b]^{\mu}$ is in the normal closure of $\widetilde{r}$ in $\langle x,b \mid \rangle$, but $[x^{k},b]$ is no element of this normal closure. Thus, $[x^{k},b]$ is a torsion element of the group $\langle x,b \mid \widetilde{r} \rangle$. By Theorem~\ref{thmMKS1}, $\widetilde{r}$ is a proper power. We define $\widetilde{r}=\rho^{n}$, where $\rho$ is no proper power and $n \in \mathbb{N} \backslash \{1\}$. Further we have by Theorem~\ref{thmMKS2}
\begin{eqnarray*}
[x^{k},b]=w^{-1} \rho^{m} w  \ \ \text{in} \ \ \langle x,b \mid \rho^{n} \rangle
\end{eqnarray*}
for some $m$ with $0< m <n$ and some $w \in \langle x,b \mid \rho^{n} \rangle$. Therefore, we have
\begin{eqnarray} \label{eqcase3}
[x^{k},b]=\widetilde{w}^{-1} \rho^{m} \widetilde{w} \ \mathcal{N}(\rho^{n}) \ \ \text{in} \ \ \langle x,b \mid \rangle
\end{eqnarray}
for some $\widetilde{w} \in \langle x,b \mid \rangle$ and an element $\mathcal{N}(\rho^{n})$ in the normal closure of $\rho^n$ in $\langle x,b \mid \rangle$. Let $\beta$ be the exponential sum of $\rho$ with respect to $b$ in $\langle x,b \mid \rangle$. Then it follows from \eqref{eqcase3} that $m \beta$ is a multiple of $n \beta$. Because of $0<m<n$ we have $\beta =0$. This means that the exponential sum of $\rho$ with respect to $b$ and therefore also the exponential sum of $\widetilde{r} = \rho^{n}$ with respect to $b$ equals $0$. By Corollary~\ref{cor0pos}, this contradicts the assumption of this subsection that $r$ has non-positive $\alpha$-$\omega$-length.\\

Altogether, we may use in this subsection that $(\langle x,b \mid \rangle \times C) / \langle \! \langle (\widetilde{r},c) \rangle \! \rangle$ embeds into $\widetilde{H} / \langle \! \langle (\widetilde{r},c) \rangle \! \rangle$. With the help of this embedding we deduce that $(\widetilde{s},d)$ is an element of the normal closure of $(\widetilde{r},c)$ in $\langle x,b \mid \rangle \times C$ since $(\widetilde{s},d)$ is an element of $\langle x,b \mid \rangle \times C$ as well as an element of the normal closure of $(\widetilde{r},c)$ in $\widetilde{H}$. By symmetry, we further deduce that $(\widetilde{r},c)$ is in the normal closure of $(\widetilde{s},d)$ in $\langle x,b \mid \rangle \times C$. Thus, $(\widetilde{r},c)$ and $(\widetilde{s},d)$ have the same normal closure in $F_{2} \times C$, where $F_{2}$ is the free group $\langle x,b \mid \rangle$. By assumption, $G$ is indicable for $C \neq \{1\}$ and $G \times C$ possesses the Magnus property. Because of Lemma~\ref{lemindC}, $\mathbb{Z} \times C$ also possesses the Magnus property. Thus, $F_{2} \times C$ has the Magnus property by Theorem~\ref{thmdirect}. Finally, $(\widetilde{r},c)$  is conjugate to $(\widetilde{s},d)^{\pm 1}$ in the subgroup $\langle x,b \mid \rangle \times C$ of $\widetilde{H}$. \qed

\section{Two applications of the main theorem~} \label{secapp}

In this section, we have a look at two applications of Main Theorem~\ref{main} on fundamental groups. The first application is a generalization of \cite{ArtMyMAP} (which is a generalization of \cite{ArtOneRel}) that uses Main Theorem~\ref{main} for $C = \{1\}$ while the second application uses Main Theorem~\ref{main} for cyclic $C$.

In \cite{ArtMyMAP}, we proved the Magnus property for the fundamental group of the graph of groups illustrated in Figure~\ref{AbbArt}, i.\,e. we proved the following theorem:

\begin{theorem} \textbf{\emph{\cite[Main Theorem]{ArtMyMAP}}} \label{thmMymap}
Let $G = \langle \{a,b\} \cup \mathcal{Y} \mid [a,b]u \rangle$, where $\mathcal{Y}$ is a nonempty set disjoint from the 2-element set $\{a,b\}$ and $u$ is a non-trivial, reduced word in the alphabet $\mathcal{Y}$. Then $G$ possesses the Magnus property.
\end{theorem}

\begin{figure}[H]
\centering
\includegraphics[scale=0.15]{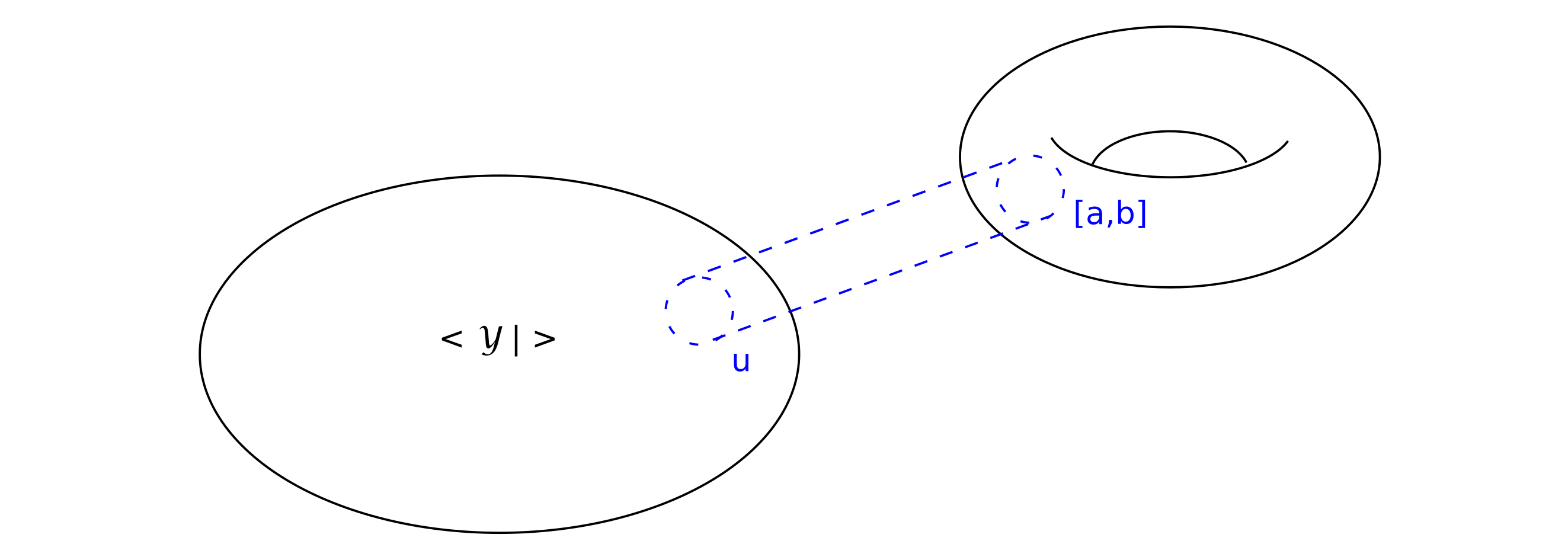}
\caption{Illustration of the group $G$ from Theorem~\ref{thmMymap}}
\label{AbbArt}
\end{figure}

With the next corollary we prove the Magnus property for the fundamental group of the graph of groups illustrated in Figure~\ref{AbbArtAnw}. Theorem~\ref{thmMymap} is a result about torsion-free 1-relator groups, where the relation identifies a commutator of two generators with an element not containing any of the two generators. The following corollary allows for group presentations with up to countably many relations of that kind, where the commutators use generators that are exclusively used in the particular commutator, but the elements identified with the commutators may use common generators.

\begin{corollary} \label{corappl}
Let $\mathcal{I}$, $\mathcal{J} \subseteq \mathbb{N}$ be index sets. Then the group
\begin{eqnarray*}
G \ = \ \langle \{a_{i} \mid i \in \mathcal{I} \} \cup \{b_{i} \mid i \in \mathcal{I} \} \cup \{ f_{j} \mid j \in \mathcal{J} \} \mid \{[a_{i},b_{i}]=u_{i} \mid i \in \mathcal{I}\} \rangle,
\end{eqnarray*}
where the $u_{i}$ are non-trivial elements from $\langle f_{j} \ (j \in \mathcal{J} ) \mid \rangle$, possesses the Magnus property. (Moreover, the group $G$ is locally indicable.)
\end{corollary}

\begin{figure}[H]
\centering
\includegraphics[scale=0.135]{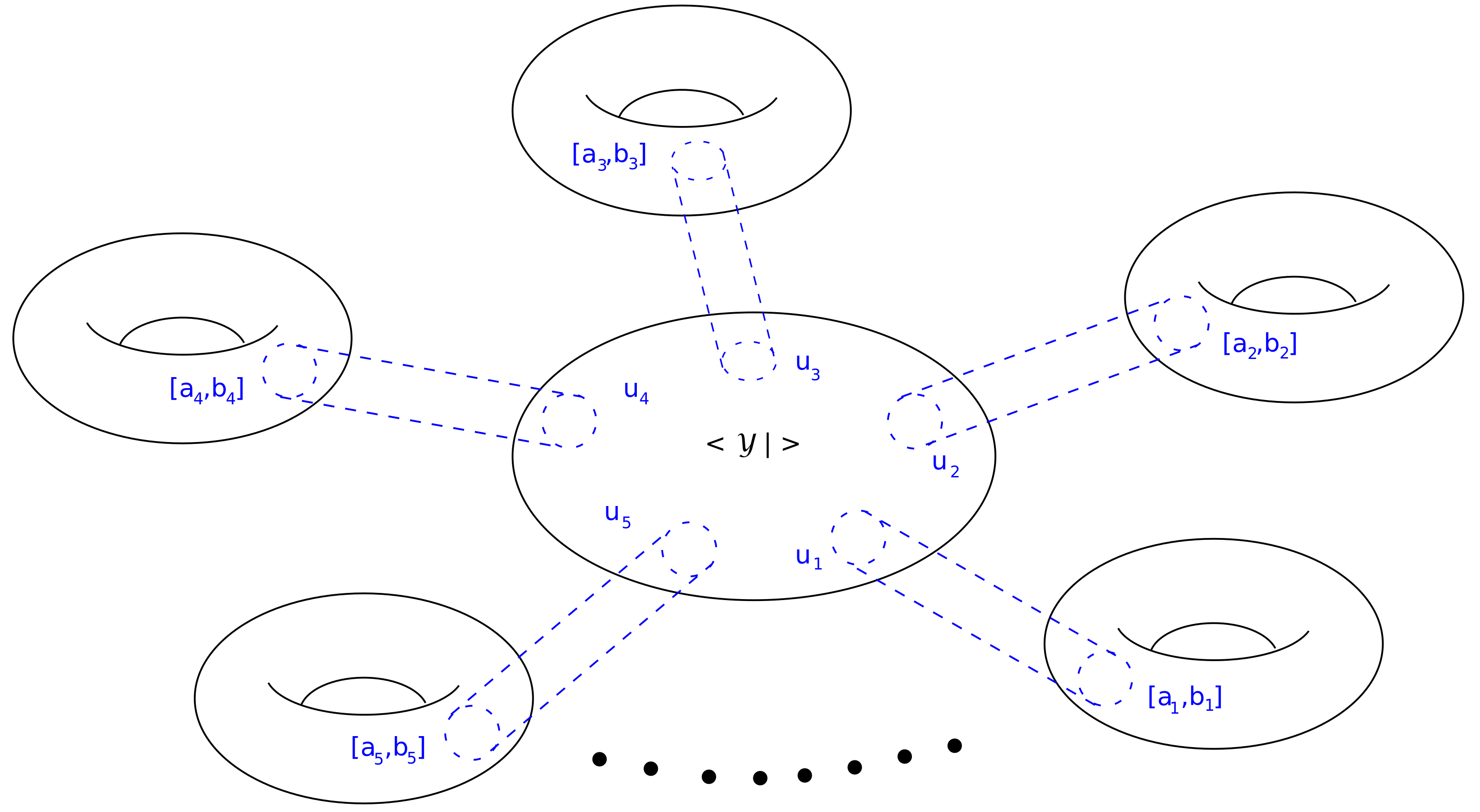}
\caption{Illustration of the group $G$ from Corollary~\ref{corappl}, where \mbox{$\mathcal{Y}:=\{f_{j} \mid j \in \mathcal{J} \}$}}
\label{AbbArtAnw}
\end{figure}

\begin{proof}
We may assume w.\,l.\,o.\,g. $\mathcal{I}=\{1,2,\dots,m\}$ for some $m \in \mathbb{N}_{0}$ or $\mathcal{I}=\mathbb{N}$. For $n \in \mathbb{N}_{0}$ we define
\begin{eqnarray*}
G_{n} := \langle \{a_{i} \mid 1 \leqslant i \leqslant n\} \cup \{b_{i} \mid 1 \leqslant i \leqslant n\} \cup  \{f_{j} \mid j \in \mathcal{J}\} \ \mid \ \{[a_{i},b_{i}]=u_{i} \mid 1 \leqslant i \leqslant n \} \rangle.
\end{eqnarray*}
Since every counterexample for the Magnus property of $G$ is contained in a group $G_{n}$ for some $n \in \mathbb{N}_{0}$, it is sufficient to prove the Magnus property of $G_{n}$ for all $n \in \mathbb{N}_{0}$.

Our proof will be by induction over $n \in \mathbb{N}_{0}$. The group $G_{0}$ is free and possesses the Magnus property by \cite{ArtMagnus}. In particular $G_{0}$ is locally indicable. For the induction step $(n \rightarrow n+1)$ we assume that the locally indicable group 
\begin{eqnarray*}
G_{n} = \langle \{a_{i} \mid 1 \leqslant i \leqslant n \} \cup \{b_{i} \mid 1 \leqslant i \leqslant n \} \cup \{f_{j} \mid j \in \mathcal{J}\} \ \mid \ \{[a_{i},b_{i}]=u_{i} \mid 1 \leqslant i \leqslant n \} \rangle
\end{eqnarray*}
possesses the Magnus property. Our aim is to show that
\begin{eqnarray*}
G_{n+1} &=& \langle \{a_{i} \mid 1 \leqslant i \leqslant n+1\} \cup \{b_{i} \mid 1 \leqslant i \leqslant n+1\} \cup \{ f_{j} \mid j \in \mathcal{J} \} \ \mid\\
&& \{[a_{i},b_{i}]=u_{i} \mid 1 \leqslant i \leqslant n+1 \} \rangle
\end{eqnarray*}
possesses the Magnus property and is locally indicable. For this purpose we write
\begin{eqnarray*}
G_{n+1}= G_{n} \underset{u_{n+1}=[a_{n+1},b_{n+1}]}{\ast}  \langle a_{n+1},b_{n+1} \mid \rangle.
\end{eqnarray*}
The left factor is locally indicable by the induction hypothesis and the right factor is free. Thus, it follows from Theorem~\ref{Howlokind} that $G_{n+1}$ is locally indicable. Because of the induction hypothesis we can apply Main Theorem~\ref{main} with $C = \{1\}$ and deduce the Magnus property of $G_{n+1}$.
\end{proof}

For the second application we consider the fundamental group $\widetilde{G}$ of the orientable Seifert manifold with invariants $r=0$, $b=0$, $g \in \mathbb{N}_{0}$ and $\varepsilon_{i}=1$ for $1 \leqslant i \leqslant g$ (see \cite[p. 91]{BuchOrlik}). Recall that by \cite{BuchOrlik} this fundamental group $\widetilde{G}$ is given by 
\begin{eqnarray*}
K &:=& \langle a_{1},b_{1},\dots,a_{g},b_{g} \mid [a_{1},b_{1}][a_{2},b_{2}] \cdots [a_{g},b_{g}]=1 \ \ (1 \leqslant i \leqslant g) \rangle \\
&=& \langle a_{2},b_{2},\dots,a_{g},b_{g} \mid \rangle \ \ \underset{([a_{2},b_{2}][a_{3},b_{3}] \cdots [a_{g},b_{g}])^{-1}=[a_{1},b_{1}]}{\ast} \ \ \langle a_{1},b_{1} \mid \rangle \ \ \ \ \ \ \ \ \ \ \ \ \ \ \text{and}\\
\widetilde{G} &:=& K \times \langle h \mid \rangle.
\end{eqnarray*}
The Magnus property of $\widetilde{G}$ follows directly by applying Main Theorem~\ref{main} with
\begin{eqnarray*}
G=\langle a_{2},b_{2},\dots,a_{g},b_{g} \mid \rangle, \ \ \ u=([a_{2},b_{2}][a_{3},b_{3}] \cdots [a_{g},b_{g}])^{-1}, \ \ \ a=a_{1}, \ \ \ b=b_{1}
\end{eqnarray*}
and $C= \langle h \mid \rangle$ since $G$ is a free group and $G \times C$ possesses the Magnus property by Theorem~\ref{thmFxF}.\\

\textbf{Acknowledgements.} This work is based on Chapter 2 and 3 of my dissertation, Heinrich-Heine-Universit\"at D\"usseldorf, 2020. I would like to express my special thanks to my supervisor Professor Oleg Bogopolski for the subject proposal and his support. Further, I wish to thank the anonymous referee for helpful remarks and suggestions.

\addcontentsline{toc}{section}{References} 
\bibliography{Literatur}
    \bibliographystyle{alpha}

\end{document}